\documentclass[reqno]{amsart}
\usepackage[foot]{amsaddr}
\usepackage[bottom=3cm, top=3cm, left=3cm, right=3cm]{geometry}

\usepackage{amsmath}
\usepackage{amsthm}
\usepackage{amssymb}
\usepackage{mdframed}
\usepackage{xfrac}
\usepackage{subcaption}
\usepackage[dvipsnames]{xcolor}
\usepackage{tikz}
\usepackage{float}
\usepackage{tikz-cd}
\usepackage{bbm}
\usepackage{pgfplots}
\pgfplotsset{compat=newest}
\usepgfplotslibrary{fillbetween}
\usetikzlibrary{hobby}
\usetikzlibrary{shapes.misc}
\usepackage{hyperref}
\usepackage[shortlabels,inline]{enumitem}
\usepackage[backref=true, backrefstyle=three, hyperref=true, maxcitenames=2, maxbibnames=50, style=alphabetic]{biblatex}
\usepackage[noabbrev,capitalize,nameinlink]{cleveref}
\crefname{equation}{}{}
\Crefname{equation}{Equation}{Equations}
\crefname{footnote}{Footnote}{Footnotes}

\newcommand{\RR}{\mathbb{R}}

\newcommand{\EE}{\mathbb{E}}

\newcommand{\Aa}{\mathcal{A}}
\newcommand{\Bb}{\mathcal{B}}

\newcommand{\Ff}{\mathcal{F}}

\newcommand{\Pp}{\mathcal{P}}

\newcommand{\Xx}{\mathcal{X}}
\newcommand{\Yy}{\mathcal{Y}}
\newcommand{\Zz}{\mathcal{Z}}

\newcommand{\eps}{\varepsilon}

\newcommand{\si}{\sigma}

\newcommand{\Si}{\Sigma}

\renewcommand{\phi}{\varphi}
\renewcommand{\epsilon}{\varepsilon}

\DeclareMathOperator*{\id}{id}
\DeclareMathOperator*{\graph}{gph}
\DeclareMathOperator*{\diam}{diam}
\DeclareMathOperator*{\rk}{rk}
\DeclareMathOperator*{\diag}{diag}

\newcommand{\gw}{\operatorname{GW}}
\newcommand{\W}{\operatorname{W}}
\newcommand{\vol}{\operatorname{vol}}
\newcommand{\Leb}{\operatorname{Leb}}
\DeclareMathOperator*{\supp}{supp}

\DeclareMathOperator*{\argmax}{arg\,max}
\DeclareMathOperator*{\argmin}{arg\,min}
\def\one{\mathbbm{1}}

\newcommand{\dd}{\mathop{}\!\mathrm d}

\newcommand{\defeq}{\coloneqq}
\newcommand{\defeqinv}{\eqqcolon}

\definecolor{tabblue}{rgb}{0.12156862745098039, 0.4666666666666667, 0.7058823529411765}
\definecolor{taborange}{rgb}{1.0, 0.4980392156862745, 0.054901960784313725}
\definecolor{tabgreen}{rgb}{0.17254901960784313, 0.6274509803921569, 0.17254901960784313}
\definecolor{tabred}{rgb}{0.8392156862745098, 0.15294117647058825, 0.1568627450980392}

\definecolor{tabblue}{HTML}{2b50aa}

\newcommand{\pimon}{\pi_\text{mon}^\uparrow}
\newcommand{\piantimon}{\pi_\text{mon}^\downarrow}

\newcommand{\submod}{S^\uparrow}
\newcommand{\supmod}{S^\downarrow}
\newcommand{\adv}{_\text{adv}}
\newcommand{\push}{_\pushonly}
\newcommand{\pushonly}{\#}
\newcommand{\opt}{^\star}

\newcommand{\capright}{\textbf{(Right)}}
\newcommand{\capleft}{\textbf{(Left)}}

\hypersetup{
  breaklinks=true,
  colorlinks=true,
  linkcolor=tabblue!100,
  urlcolor=tabred,
  citecolor=tabgreen,
  bookmarksdepth=3,
}

\newcounter{counter}
\numberwithin{counter}{section}
\newtheorem{theorem}[counter]{Theorem}
\newtheorem{lemma}[counter]{Lemma}
\newtheorem{Anblemma}{Lemma}
\newenvironment{nblemma}[1]{\renewcommand\theAnblemma{#1}\Anblemma}{\endAnblemma}
\newtheorem{proposition}[counter]{Proposition}
\newtheorem{Anbproposition}{Proposition}
\newenvironment{nbproposition}[1]{\renewcommand\theAnbproposition{#1}\Anbproposition}{\endAnbproposition}

\theoremstyle{definition}
\newtheorem{remark}[counter]{Remark}
\newtheorem{example}[counter]{Example}
\newtheorem{definition}[counter]{Definition}
\newtheorem{conjecture}[counter]{Conjecture}
\numberwithin{equation}{section}

\usepackage{algorithm}

\title[Existence of Monge maps for the Gromov--Wasserstein problem]{On the existence of Monge maps for the\\ Gromov--Wasserstein problem}
\date{}
\author{Théo Dumont$^\dagger$}
\author{Théo Lacombe$^\dagger$}
\author{François-Xavier Vialard$^\dagger$}
\address{\textnormal{$^\dagger$Laboratoire d'Informatique Gaspard Monge, Université Gustave Eiffel, CNRS, F-77454 Marne-la-Vallée, France.}}
\email{\{\href{mailto:theo.dumont@univ-eiffel.fr}{theo.dumont},\href{mailto:theo.lacombe@univ-eiffel.fr}{theo.lacombe},\href{mailto:francois-xavier.vialard@univ-eiffel.fr}{francois-xavier.vialard}\}@univ-eiffel.fr}

\usepackage{algpseudocode}
\usepackage{mathtools}
\makeatletter
\renewcommand{\tocsection}[3]{%
  \indentlabel{\@ifnotempty{#2}{\bfseries\ignorespaces#1 #2\quad}}\bfseries#3}
\renewcommand{\tocsubsection}[3]{%
  \indentlabel{\@ifnotempty{#2}{\ignorespaces#1 #2\quad}}#3}
\def\l@subsection{\@tocline{2}{0pt}{2.5pc}{5pc}{}}
\renewcommand{\tocsubsubsection}[3]{%
  \indentlabel{\@ifnotempty{#2}{\ignorespaces#1 #2\quad}}#3}
\def\l@subsubsection{\@tocline{2}{0pt}{4.5pc}{5pc}{}}
\makeatother
\usepackage{multirow}

\usepackage[symbol*]{footmisc}

\begin{document}

\setfnsymbol{wiley}

\begin{abstract}
    The Gromov--Wasserstein problem is a non-convex optimization problem over the polytope of transportation plans between two probability measures supported on two spaces, each equipped with a cost function evaluating similarities between points. 
    Akin to the standard optimal transportation problem, it is natural to ask for conditions guaranteeing some structure on the optimizers, for instance if these are induced by a (Monge) map.
    We study this question in Euclidean spaces when the cost functions are either given by (i) inner products or (ii) squared distances, two standard choices in the literature. 
    We establish the existence of an optimal map in case (i) and of an optimal $2$-map (the union of the graphs of two maps) in case (ii), both under an absolute continuity condition on the source measure. 
    Additionally, in case (ii) and in dimension one, we numerically design situations where optimizers of the Gromov--Wasserstein problem are $2$-maps but are not maps. 
    This suggests that our result cannot be improved in general for this cost. 
    Still in dimension one, we additionally establish the optimality of monotone maps under some conditions on the measures, thereby giving insight on why such maps often appear to be optimal in numerical experiments.
    
    \vspace{2mm}
    \noindent\textsc{Keywords.} Non-convex optimization $\cdot$ Optimal transportation $\cdot$ Monge map $\cdot$ Gromov--Wasserstein
    
    \vspace{2mm}
    \noindent\textsc{Mathematics Subject Classification.} 49Q22 $\cdot$ 49Q10 $\cdot$ 90C26 
\end{abstract}
\maketitle

\addtocontents{toc}{\protect\setcounter{tocdepth}{3}}
{
\hypersetup{linkcolor=black}
\tableofcontents
}
\renewcommand*{\thefootnote}{\arabic{footnote}}

\allowdisplaybreaks

\section{Introduction}
\label{sec:introduction}
Finding correspondences between objects that do not live in the same metric space is a problem of fundamental interest both in applications and in theory appearing in very different fields such as computer vision and shape analysis \cite{umeyama1988eigendecomposition,memoli2004comparing,memoli2005theoretical,berg2005shape,memoli2007use,memoli2011gromov}, mathematics \cite{sturm2012space}, biology \cite{demetci2020gromov} and machine learning \cite{redko2020co,alvarez2018gromov}. 
Graph matching \cite{zhou2015factorized} is a prominent example of such a problem.
A usual situation is when the objects of interest are metric spaces themselves, whose comparison is of practical and theoretical importance \cite{gromov1999metric}; in this context, the Gromov--Hausdorff distance has been used in different settings \cite{memoli2008gromov,memoli2021gromov} and its relaxation to $L^p$ spaces, the Gromov--Wasserstein distance, has been explored 
in \cite{sturm2006geometry,memoli2007use,memoli2008gromov}. Following this line of research, comparing metric measure spaces (i.e.~metric spaces endowed with probability measures) using Wasserstein-type distances has attracted a lot of interest \cite{memoli2011gromov,sturm2012space}. Follow-up works over the past decade include for instance \cite{sejourne2021unbalanced,de2022entropy}, which propose extensions to the case of metric spaces with positive measures (whose total volume is not normalized to $1$).
The Gromov--Wasserstein distance and its extensions are applied successfully in machine learning \cite{grave2019unsupervised,bunne2019learning,xu2020learning} and biology \cite{demetci2022scotv2}. 

Following \cite{memoli2007use}, the Gromov--Wasserstein approach consists in seeking for a correspondence between two metric measure spaces, called transport plan, or coupling \cite{santambrogio2015optimal}, that is of lowest distortion. 
In general, transport plans may not be deterministic: a point can be paired with (or transported to) several points of the other space, and conversely. 
A natural question is thus to ask for conditions under which \emph{optimal} transport plans for the Gromov--Wasserstein problem can be guaranteed to be deterministic, that is to be supported on the graph of a map from one space to the other. 
This problem has been put forward by \citeauthor{sturm2012space} \cite[Challenge~3.6]{sturm2012space} and by \citeauthor{MemNeedham22} \cite[Question~2.14]{MemNeedham22}. 
However, in contrast to optimal transport, which is a linear programming problem, the formulation of the Gromov--Wasserstein problem falls in the class of non-convex quadratic assignment problems \cite{koopmans1957assignment}, whose global optimizers are computationally harder to solve and theoretically harder to characterize. 
As a consequence, it is not surprising that fewer results are available in the literature, though some particular instances can be computed in polynomial time \cite{memoli2023ultrametric}. 

In optimal transport, the fact that the optimization can, under mild conditions, be restricted to the space of maps has been developed a lot since \citeauthor{brenier1987decomposition}'s work \cite{brenier1987decomposition} and further generalized by \citeauthor{mccann2001polar} \cite{mccann2001polar}. \citeauthor{brenier1987decomposition}'s result essentially states that, in Euclidean spaces, when the cost function is the squared Euclidean distance and the initial measure has a Lebesgue density, the optimal transport plan is a map that is given by the gradient of a convex function. 
Such results on the structure of optimal plans/maps are of great interest for reducing the optimization set~\cite{makkuva2020optimal}. They heavily depend on the choice of the cost function, and the same can be said for the Gromov--Wasserstein problem, as detailed in this work. 

In this work, we address the question of the existence of optimal maps for the $L^2$-Gromov--Wasserstein problem in Euclidean spaces in two particular cases.
The first one is when the distortion is measured in terms of Euclidean inner products, in which we establish the existence of optimal maps and detail their structure.
The second is when the distortion is measured in terms of squared Euclidean distances, in which we establish the existence of optimal $2$-maps (plans supported on the union of the graphs of two maps) in general, and of optimal maps in specific cases. 
In both cases, we require an absolute continuity condition similar to the one of \citeauthor{brenier1987decomposition}'s theorem.
We also study the second case (squared distances) in dimension one, which has attracted recent attention \cite{titouan2019sliced,beinert2022assignment}.
We provide numerical evidence for a counter-example to the existence of optimal maps, suggesting that our result asserting the existence of optimal $2$-maps cannot be improved in general.
Furthermore, it was numerically observed in \cite{titouan2019sliced} that monotone (non-decreasing or non-increasing) maps between discrete measures in dimension one were often optimal;
however, a counter-example to this claim was designed in \cite{beinert2022assignment}. 
Yet, we show that this holds under suitable conditions, shedding light on why the approach developed in \cite{titouan2019sliced} seems to work well in practice.
We refer the reader to \Cref{subsec:intro_contribs} for a detailed account of our contributions, while the background and state-of-the-art are presented respectively in \Cref{subsec:intro_GW,subsec:intro_relatedworks}.\\~\vspace{-2mm}\\
\noindent\textbf{Notation.}
In the following, 
$\Pp(\Xx)$ is the set of probability measures on a Polish space $\Xx$,
$\Leb_n$ is the $n$-dimensional Lebesgue measure on $\RR^n$,
$\vol_M$ is the volume measure on a manifold $M$,
$\ll$ denotes absolute continuity,
$T\push\mu$ is the pushforward of a measure $\mu\in \Pp(\Xx)$ by a measurable map $T:\Xx\to\Yy$, namely the measure of $\Pp(\Yy)$ defined on Borel sets $A$ by $T\push\mu(A)\defeq \mu(T^{-1}(A))$,
$\chi_S$ is the indicator function of a set $S$,
$O_n(\RR)$ is the set of orthogonal matrices of dimension $n\times n$ with real entries,
$\mathfrak S_n$ is the set of permutations,
$\one_n$ is the vector of $\RR^{n}$ with entries all equal to $1$,
and
``a.e.'' means ``almost everywhere'' or ``almost every'' depending on the context.
Also, to alleviate notation, $\langle \cdot, \cdot \rangle$ and $\|\cdot\|$ respectively denote the inner product and the $2$-norm $\|\cdot\|_2$ on $\RR^n$ for any $n\geq1$.

\subsection{The Gromov--Wasserstein problem}
\label{subsec:intro_GW}

\subsubsection{Formulation}
\label{sec:formulation}
The Gromov--Wasserstein (GW) problem, initially introduced by \citeauthor{memoli2007use}~\cite{memoli2007use},
can be seen as an extension of the Gromov--Hausdorff distance \cite{gromov1999metric} to the context of (probability) measure spaces $(\Xx,\mu)$ and $(\Yy,\nu)$ equipped with cost functions $c_\Xx : \Xx \times \Xx \to \RR$ and $c_{\Yy}:\Yy\times \Yy\to \RR$ (typically, $c_\Xx$ and $c_\Yy$ can be distance functions on $\Xx$ and $\Yy$, respectively).
The GW problem seeks a correspondence $\pi$ (i.e.~a joint law) between $\mu$ and $\nu$ that would make the quantity $\EE_{\pi\otimes\pi}\big[\|c(X,X')-c(Y,Y')\|_{L^p}\big]$ as small as possible, where $(X,Y)\sim\pi$ and $(X',Y')\sim\pi$ are independent coupled random variables. The formal definition is as follows.
\begin{definition}[GW problem]
\label{def:gw}
Let $\Xx$ and $\Yy$ be Polish spaces and $p\geq 1$. Given two probability measures $\mu \in \Pp(\Xx)$ and $\nu  \in \Pp(\Yy)$, two continuous symmetric \emph{cost functions} $c_{\Xx}:\Xx\times \Xx\to \RR$ and $c_{\Yy}:\Yy\times \Yy\to \RR$, the \emph{$L^p$-Gromov--Wasserstein problem} aims at finding
            \begin{align*}
                \tag{GW}
                \gw_p(\mu,\nu)\defeq\inf_{\pi \in \Pi(\mu,\nu)} \Big( \int _{\Xx\times \Yy}\int _{\Xx\times \Yy}\big|c_{\Xx}(x,x')-c_{\Yy}(y,y')\big|^p \dd\pi(x,y)\dd\pi(x',y')\Big)^{1/p},
                \label{eq:gw}
            \end{align*}
where $\Pi(\mu,\nu)$ denotes the subset of $\Pp(\Xx \times \Yy)$ of probability measures that admit $\mu$ and $\nu$ as first and second marginals.
Such an element $\pi\in\Pi(\mu,\nu)$ is said to be a \emph{plan} between $\mu$ and $\nu$. If there exists a measurable map $T : \Xx \to \Yy$ such that $\pi = (\id,T)\push \mu$, then $T$ is said to be a \emph{map} between $\mu$ and $\nu$ (because of the marginal constraints on $\pi$, this is equivalent to requiring $T\push\mu=\nu$).
Plans $\pi\in\Pi(\mu,\nu)$ that solve \cref{eq:gw} will be called \emph{optimal correspondence plans} between $\mu$ and $\nu$, and maps such that $(\id,T)\push \mu$ solves \cref{eq:gw} will be called \emph{optimal correspondence maps}, or \emph{Gromov--Monge maps}, between $\mu$ and $\nu$ (see \cref{rem:terminology} for a note on the terminology we adopt).
\end{definition}

While the existence of optimal correspondence plans holds under mild assumptions by compactness arguments as long as the above minimum is not $+\infty$, much less is known about the existence of optimal correspondence maps, even in simple cases.
In this work, we restrict to the case $\Xx \subset \RR^n$ and $\Yy \subset \RR^{d}$ for two integers $n\geq d$, we fix $p=2$ and omit the ``$L^2$'' when referring to the Gromov--Wasserstein problem for the sake of conciseness. 
We now consider the following two specific instances of \cref{eq:gw}:
\begin{enumerate}[(i),leftmargin=*]
    \item the \emph{inner product} case \cite{alvarez2018gromov,vayer2020contribution}, where $c_\Xx$ and $c_\Yy$ compare angles in $\RR^n$ and $\RR^d$, respectively:
        \begin{equation}
            \tag{GW-IP}
            \min _{\pi \in \Pi(\mu, \nu)} \int_{\Xx\times\Yy}\int_{\Xx\times\Yy}\big|\langle x, x'\rangle-\langle y, y'\rangle\big|^2  \dd\pi(x, y) \dd\pi(x', y'),
            \label{eqn:GW-inner-prod}
        \end{equation}
        \item and the \emph{squared distance} case \cite{sturm2012space,vayer2020contribution}, where $c_\Xx$ and $c_{\Yy}$ are the squared Euclidean $2$-norms on $\RR^n$ and $\RR^d$, respectively:
        \begin{equation}
            \tag{GW-SD}
            \min _{\pi \in \Pi(\mu, \nu)} \int_{\Xx\times\Yy}\int_{\Xx\times\Yy}\big|\|x-x'\|^2-\|y-y'\|^2\big|^2  \dd\pi(x, y) \dd\pi(x', y').
            \label{eqn:GW-squared-distance}
        \end{equation}
    \end{enumerate}
        The choice (ii) is standard, as it fits the more general setting of $c_{\Xx}$ and $c_\Yy$ being powers of some distance functions $d_\Xx$ and $d_\Yy$ on $\Xx$ and $\Yy$. In this case, if $\mu$ and $\nu$ have full support, $\gw_p(\mu,\nu) = 0$ is equivalent to the metric measure spaces $(\Xx, d_\Xx, \mu)$ and $(\Yy, d_\Yy, \nu)$ being {isomorphic}, where by isomorphic we mean that is there exists an isometry $\varphi : (\Xx, d_\Xx) \to (\Yy, d_\Yy)$ such that $\varphi \push \mu = \nu$~\cite[Lemma 9.2]{sturm2012space}.\footnote{Note that in the inner product case, $\gw(\mu,\nu)=0$ implies that the metric measure spaces are isomorphic, but the converse is not true---for instance in cases where $(\Xx, d_\Xx, \mu)$ and $(\Yy, d_\Yy, \nu)$ are point clouds in an Euclidean space that simply differ by a translation, meaning $\Yy=T(\Xx)$, $d_{\Xx}=d_{\Yy}$ the Euclidean distance, and $\nu=T\push\mu$, where $T:x\mapsto x+x_0$.\label{foot:IP-GW}}
        A~subcase of problem (ii) is given when $\mu = \frac{1}{N} \smash{\sum_{i=1}^N \delta_{x_i}}$ and $\nu = \frac{1}{N} \smash{\sum_{j=1}^N \delta_{y_j}}$ are uniform probability measures supported on $N$ points each. In this scenario, optimal correspondence plans $\pi$ can be chosen as permutations $\sigma$ of $\{1,\dots,N\}$ (see \cite[Theorem~1]{maron2018probably} and \cite[Theorem~4.1.2]{vayer2020contribution}), and the problem optimizes over the set of such permutations $\mathfrak{S}_N$,
        \begin{equation}
            \tag{QAP}
            \min_{\sigma\in \mathfrak{S}_N}\ \sum_{i,j} \big| \|x_i - x_j\|^2 - \|y_{\sigma(i)} - y_{\sigma(j)}\|^2 \big|^2,
            \label{eq:QAP}
        \end{equation}
        which is a particular case of the \emph{Quadratic Assignment Problem}, first introduced in \cite{koopmans1957assignment}.
\begin{remark}[Assumptions on the cost functions]
    The Gromov--Wasserstein problem has been carefully studied in \cite{memoli2007use,memoli2011gromov,sturm2012space} in the context of \emph{metric measure spaces}, where it is therefore required that the cost functions $c_{\Xx}$ and $c_{\Yy}$ are distance functions on $\Xx$ and $\Yy$, respectively. It was further observed by \citeauthor{sturm2012space} in~\cite[Section 5]{sturm2012space} that the definition of GW can be extended not only to powers of distance functions but also to the more general setting of \emph{gauged spaces}, where cost functions are merely required to be symmetric and square-integrable---in particular, they are no longer required to satisfy the triangle inequality. \citeauthor{chowdhury2019gromov}~\cite{chowdhury2019gromov} then studied \emph{network spaces}, where the symmetry assumption is further removed. Our definition of the GW problem solely requires cost functions to be continuous and symmetric.
\end{remark}

\subsubsection{Relation with the optimal transportation problem: a tight bilinear relaxation}
\label{subsec:intro_GW_and_OT}
Let us first recall the formulation of the Optimal Transportation (OT) problem, which will play an extensive role in this work.
\begin{definition}[OT problem]
    Let $\Xx$ and $\Yy$ be Polish spaces. Given two probability measures $\mu \in\Pp(\Xx)$ and $\nu \in\Pp(\Yy)$ and a continuous \emph{cost function} $c:\Xx\times \Yy\to \RR$, we consider the problem
    \begin{equation}
        \tag{OT}
        \min_{\pi \in \Pi(\mu,\nu)} \int _{\Xx\times \Yy}c(x,y)\dd\pi(x,y) .
        \label{eq:OT}
    \end{equation}
Plans $\pi\in\Pi(\mu,\nu)$ that solve \cref{eq:OT} are called \emph{optimal transport plans} between $\mu$ and $\nu$, and maps such that $(\id,T)\push \mu$ solves \cref{eq:OT} are called \emph{optimal transport maps}, or \emph{Monge maps}, between $\mu$ and $\nu$.
\end{definition}

\begin{remark}[A note on terminology]
    \label{rem:terminology}
    Note that while elements of $\Pi(\mu,\nu)$ will always be called \emph{plans} in this work, optimal plans for \cref{eq:OT} and \cref{eq:gw} will respectively be called optimal \emph{transport} plans and optimal \emph{correspondence} plans. By doing so, we deviate from the terminology of ``optimal couplings'', standard in the literature \cite{sturm2012space,chowdhury2019gromov,memoli2022comparison}, that does not differentiate between optimality for \cref{eq:OT} and optimality for \cref{eq:gw}. We make this distinction to avoid confusion in the rest of the paper, as we will be sometimes simultaneously dealing with plans that are optimal for \cref{eq:OT} and others that are optimal for \cref{eq:gw}. The same remark holds about \emph{Monge} and \emph{Gromov--Monge} maps. 
\end{remark}

The minimization problem in \cref{eq:gw} can be interpreted as the minimization of the map $\pi \mapsto F(\pi,\pi) \defeq \iint k \dd \pi\! \otimes\! \pi$ where $k((x,y),(x',y')) = |c_\Xx(x,x') - c_\Yy(y,y')|^2$, and $F$ is a symmetric bilinear map.
The first-order optimality condition ensures that if $\pi\opt$ minimizes \cref{eq:gw}, then it also minimizes $\pi \mapsto 2 F(\pi, \pi\opt)$.
Letting $C_{\pi\opt}(x,y) \defeq \int_{\Xx \times \Yy} k((x,y),(x',y')) \dd \pi\opt(x', y')$, this is up to a factor $2$ the linear problem
\begin{equation}
    \min_{\pi \in \Pi(\mu,\nu)} \int_{\Xx \times \Yy} C_{\pi\opt}(x,y) \dd \pi(x,y),
    \label{eq:linearized}
\end{equation}
which is nothing but the \cref{eq:OT} problem induced by the cost $C_{\pi\opt}$ on $\Xx \times \Yy$.
Therefore, any optimal \emph{correspondence} plan for \cref{eq:gw} with costs $c_\Xx$, $c_\Yy$ must be an optimal \emph{transportation} plan for \cref{eq:OT} with cost $C_{\pi\opt}$.
A crucial point, proved in \cite[Theorem~3]{sejourne2021unbalanced} as a generalization of a result by \citeauthor{konno1976maximization}~\cite{konno1976maximization}, is that if $k$ is symmetric negative on the set of (signed) measures on $\Xx \times \Yy$ with null marginals, that is $\int k \dd \alpha\!\otimes\!\alpha \leq 0$ for every such $\alpha$, then the converse implication holds: any solution $\gamma\opt \in \Pi(\mu,\nu)$ of the \cref{eq:OT} problem with cost $C_{\pi\opt}$ is also a solution of the \cref{eq:gw} problem, that is 
\begin{equation}\label{eq:biconvex_relaxation}
    F(\pi\opt,\pi\opt) = F(\gamma\opt,\gamma\opt) = F(\pi\opt,\gamma\opt).
\end{equation}
This result is useful in the present article to derive theoretical properties on the minimizers but also for the alternate minimization algorithms that were first proposed in \cite{memoli2007use} and extended in \cite{sejourne2021unbalanced}.
The question of the existence and structure of optimal transport maps has been extensively studied in optimal transportation and a review can be found in \cref{subsec:intro_relatedworks_MongeOT}.
Since in this case, the solutions of GW are in correspondence with the solutions of an OT problem, the tools and knowledge from optimal transportation theory can be used to derive existence and structure of optimal correspondence maps.
In particular, this holds for our two problems of interest \cref{eqn:GW-squared-distance} and \cref{eqn:GW-inner-prod}: if $\alpha$ denotes a finite signed measure on $\Xx \times \Yy \subset \RR^n \times \RR^d$ with null marginals, observe that
\begin{align*}
    &\hspace{-1.8cm}\int \big| \|x-x'\|^2 - \|y-y'\|^2\big|^2 \dd \alpha\!\otimes\!\alpha(x,y,x',y') \\
    &= \underbrace{\int \|x-x'\|^2 \dd \alpha\! \otimes\! \alpha}_{=\ 0} + \underbrace{\int \|y-y'\|^2 \dd \alpha\! \otimes\! \alpha}_{=\ 0} \ -\ 2 \int \|x-x'\|^2 \|y-y'\|^2 \dd \alpha\! \otimes\! \alpha \\
    & = -2 \int \big(\|x\|^2 - 2 \langle x, x' \rangle + \|x\|^2\big)\big(\|y\|^2 - 2 \langle y , y' \rangle + \|y'\|^2\big) \dd \alpha\! \otimes\! \alpha.
\end{align*}
Expanding the remaining integrand involves nine terms. Given that $\alpha$ has null marginals (in particular, null total mass), we get $\int \|x\|^2 \|y\|^2 \dd \alpha\! \otimes \!\alpha = 0$ (and similarly for the terms involving $\|x'\|^2 \|y'\|^2$, $\|x\|^2\|y'\|^2$ and $\|x'\|^2\|y\|^2$), as well as $\int \|x\|^2 \langle y,y' \rangle \dd \alpha\! \otimes\! \alpha = 0$ (and similarly for the other terms).
Eventually, the only remaining term is
\begin{equation}
    \label{Eqcovariance}
    - 8 \int \langle x, x' \rangle \langle y,y' \rangle \dd \alpha \!\otimes\! \alpha 
    =  - 8 \Big\| \int xy^\top \dd \alpha(x,y) \Big\|_F^2 \leq 0, 
\end{equation}
where $\| \cdot \|_F$ denotes the Frobenius norm of a matrix. The same holds for non-finite $\alpha$ by density.
The negativity of this term ensures that solutions of \cref{eqn:GW-squared-distance} are exactly the solutions of an OT problem.
Computations for \cref{eqn:GW-inner-prod} are similar---actually, they immediately boil down to the same last two equations. 
More generally, let $d_\Xx$ and $d_\Yy$ be distance functions on $\Xx$ and $\Yy$. When one considers the cost $(d_\Xx(x,x') - d_\Yy(y,y'))^2$, by expanding the square, the only term that matters in the optimization is $-2d_\Xx(x,x')d_\Yy(y,y')$. 
Whenever it is possible to write $d_\Xx$ and $d_\Yy$ as squared distances in Hilbert spaces, namely $d_\Xx(x,x') = \|\varphi(x) - \varphi(x')\|_{H_\Xx}^2$ and $d_\Yy(y,y') = \|\psi(y) - \psi(y')\|_{H_\Yy}^2$ for an embedding $\varphi : \Xx \to H_\Xx$ in a Hilbert space $H_\Xx$ and similarly for $\Yy$, computation \cref{Eqcovariance} holds. 
Such a property depends on the metric space. When it is satisfied, the metric space is said to be of \emph{negative type}, or the distance to be \emph{Hilbertian}. Several important Riemannian manifolds are of negative type; among them the real hyperbolic space, the sphere, and the Euclidean space. 
Counter-examples are for instance the hyperbolic space on the quaternions \cite{faraut1974distances} or the $2$-Wasserstein space on $\RR^d$ for $d\geq 3$ \cite{andoni2018snowflake}. We refer to \cite{lyons2013distance} for a thorough discussion. Another equivalent formulation is to say that $d_\Xx$ is a \emph{conditionally negative definite kernel} on $\Xx$.
\begin{definition}[CND kernel]
    A function $k_\Xx:\Xx \times \Xx \to \RR$ is a \emph{conditionally negative definite (CND) kernel} if it is symmetric and for all $N\geq 1$, $x_1,\ldots,x_N \in \Xx$ and $\omega_1,\ldots,\omega_N \in \RR$ such that $\sum_{i=1}^N \omega_i = 0$, one has $\sum_{i,j=1}^N\omega_i\omega_jk_\Xx(x_i,x_j) \leq 0$.
\end{definition}
Every CND kernel can be written as $k_\Xx(x,x') = f(x) + f(x') + \frac 12 \| \varphi(x) - \varphi(x')\|^2_{H}$ for an embedding $\varphi: \Xx \to H$ with $H$ a Hilbert space and $f:\Xx\to\RR$ a function \cite{schoenberg1938metric}. As far as the Gromov--Wasserstein functional is concerned, our discussion above shows that $c_\Xx$ and $c_\Yy$ can actually be replaced with CND kernels and that the relaxation still holds. Obviously, due to the square in the $L^2$-GW problem, the ``sign'' of both kernels does not matter. 
The following proposition sums up the discussion.
\begin{proposition}[Tight bilinear relaxation of GW]
Let $(\Xx,k_\Xx,\mu)$ and $(\Yy,k_\Yy,\nu)$ be two Polish spaces endowed each with a conditionally negative definite kernel (or each with a conditionally positive definite one) and a probability measure. Then the bilinear relaxation of GW is tight, in the sense of \cref{eq:biconvex_relaxation}. The corresponding kernel $k((x,y),(x',y'))$ yields a non-positive quadratic form on signed measures with null marginals on $\Xx \times \Yy$.
\end{proposition}

\pagebreak
\subsection{Related work}
\label{subsec:intro_relatedworks}

\subsubsection{Monge maps for the OT problem}
\label{subsec:intro_relatedworks_MongeOT}
The \cref{eq:OT} problem has been extensively studied (see \cite{santambrogio2015optimal,villani2009optimal,peyre2019computational} for a thorough introduction) and particular attention has been devoted to situations where the existence of Monge maps, or variations of, can be ensured.
\citeauthor{brenier1987decomposition}'s theorem \cite{brenier1987decomposition}, stated below, is the best-known of these results.

\begin{theorem}[\citeauthor{brenier1987decomposition}'s theorem]
    \label{theorem:brenier}
    Let $\Xx=\Yy=\RR^d$, $c(x,y)=\|x-y\|^2$, and $\mu,\nu \in\Pp(\RR^d)$ such that the optimal cost between $\mu$ and $\nu$ is finite. If $\mu\ll \Leb_d$, then there exists a unique (up to a set of $\mu$-measure zero) solution $\pi\opt$ of \cref{eq:OT} and it is induced by a map, which is the gradient of a convex function $f:\RR^d\to\RR$:
    \begin{equation*}
        \pi\opt=(\id,\nabla f)\push\mu.
    \end{equation*}
\end{theorem}
This central result admits a generalization in the manifold setting, initially proposed by \citeauthor{mccann2001polar}~\cite{mccann2001polar}, that we shall use later on:
\begin{theorem}[{\cite[Theorem~10.41]{villani2009optimal}}]
\label{prop:quad-cost-manifold-villani}
    Let $\Xx=\Yy=M$ be a Riemannian manifold with geodesic distance $d$, and $c(x, y)=d(x, y)^2$. Let $\mu, \nu\in\Pp(M)$ with compact support. If $\mu\ll\vol_M$,
    then there exists a unique solution $\pi\opt$ of \cref{eq:OT} and it is induced by a map, which is the gradient of a $d^2 / 2$-convex function $f:M\to\RR$ (see \cref{rem:c-cvx}):
    \begin{equation*}
        \pi\opt=(\id,\exp_x(\nabla f))\push\mu.
    \end{equation*}
\end{theorem}

This theorem can be extended in a few directions. The condition that $\mu$ has a density can be weakened to the fact that it does not give mass to sets of Hausdorff dimension smaller than $d-1$ (e.g.~hypersurfaces), and $c$ can be more general than the squared distance function, as long as it satisfies the \emph{twist condition}, defined below.
In the following, $\Xx$ and $\Yy$ are complete Riemannian manifolds and $c:\Xx\times\Yy\to \RR$ is a continuous cost function, differentiable with respect to $x$.
We refer to \cite{mccann2011five,villani2009optimal} for more information on the twist condition, to \cite{chiappori2010hedonic,ahmad2011optimal,mccann2012glimpse} for the introduction of the subtwist condition, and to \cite{moameni2016characterization} for that of the $m$-twist and generalized twist conditions.

\begin{proposition}[Twist condition]
\label{prop:twist}
    We say that $c$ satisfies the \emph{twist condition} if
    \begin{equation}
        \tag{Twist}
        \text{for all }x_0\in\Xx,\quad  y\mapsto \nabla_x c(x_0,y)\in T_{x_0}\Xx \text{ is injective.}
        \label{eq:twist}
    \end{equation}
    Suppose that $c$ satisfies \cref{eq:twist} and assume that any $c$-concave function is differentiable $\mu$-a.e.~on its domain. If $\mu$ and $\nu$ have finite transport cost, then \cref{eq:OT} admits a unique optimal transport plan $\pi\opt$ and it is induced by a map, which is the gradient of a $c$-convex function $f:\Xx\to\RR$:
    \begin{equation*}
    \pi\opt=(\id,c\text{-}\exp_x(\nabla f))\push\mu.
    \end{equation*}
\end{proposition}
\begin{remark}[$c$-convexity, $c$-exponential]
    \label{rem:c-cvx}
    The notion of $c$-convexity is a generalization of plain convexity, and we refer for instance to \citeauthor{villani2009optimal} for the definition \cite[Definition 5.2]{villani2009optimal} as well as some intuition on $d^2/2$-convex functions \cite[pp.~355--356]{villani2009optimal}.
    Following \citeauthor{loeper2009regularity}~\cite[Definition 2.6]{loeper2009regularity} and \citeauthor{villani2009optimal}~\cite[Definition 12.29]{villani2009optimal}, we also recall that the $c$-\emph{exponential map} is defined such that $c$-$\exp_x(p)$ is a $y\in\Yy$ satisfying $\nabla_xc(x,y)+p=0$, when it exists and is unique. This notion reduces to the usual Riemannian exponential map when $c=d^2/2$.
\end{remark}
\begin{remark}[Local twist condition and uniqueness]
Suppose $\Xx$ and $\Yy$ $d$-dimensional and $\mu$ and $\nu$ compactly supported. \citeauthor{mccann2012rectifiability}~\cite{mccann2012rectifiability} showed that if additionally $c$ is $C^2$ and satisfies the weaker condition of \emph{non-degeneracy} $\det D^2_{xy}c(x,y)\neq0$ for all $(x,y)\in \Xx\times\Yy$, any solution of \cref{eq:OT} is supported on a $d$-dimensional Lipschitz submanifold (instead of $2d$ in full generality); but with no guarantee of uniqueness. General conditions for the uniqueness of the solutions of the \cref{eq:OT} problem to hold have recently been established in \cite{moameni2020uniquely}.
\end{remark}
\noindent Costs of the form $c(x,y)=h(x-y)$ with $h$ strictly convex, and in particular the costs $c(x,y)=\|x-y\|^p$ for $p>1$, do satisfy the twist condition. Unfortunately, it cannot be satisfied for smooth costs on compact manifolds.
Two weaker notions can however be introduced to retain some (but less) structure on optimal plans.
\begin{proposition}[Subtwist condition]
\label{prop:subtwist}
    We say that $c$ satisfies the \emph{subtwist condition} if
    \begin{equation}
        \tag{Subtwist}
        \begin{array}{l}
            \text{for all } y_1\neq y_2\in\Yy,\ x\in \Xx\mapsto c(x,y_1)-c(x,y_2) \text{ has no critical points,}\\
            \text{save at most one global maximum and at most one global minimum.}
            \end{array}
        \label{eq:subtwist}
    \end{equation}
    Suppose that $c$ satisfies \cref{eq:subtwist} and is bounded. If $\mu$ vanishes on each hypersurface, then \cref{eq:OT} admits a unique optimal transport plan $\pi\opt$ and it is supported on the union of a graph (map from $\Xx$ to $\Yy$) and an anti-graph (map from $\Yy$ to $\Xx$):
    \begin{equation*}
        \pi\opt=(\id , G)\push \bar\mu+(H, \id)\push(\nu-G\push \bar\mu)
    \end{equation*}
    for some (Borel) measurable maps $G:\Xx\to\Yy$ and $H: \Yy\to\Xx$ and non-negative measure $\bar\mu \leq \mu$ such that $\nu-G\push \bar\mu$ vanishes on the range of $G$. We call such a plan a \emph{map/anti-map}.
\end{proposition}
\begin{proposition}[$m$-twist condition]
\label{prop:mtwist}
    We say that $c$ satisfies a \emph{$m$-twist condition} if
    \begin{equation}
        \tag{$m$-twist}
        \text{for all } (x_{0},y_{0}) \in \Xx\times \Yy,\text{ the set } \{y \mid \nabla_x c(x_{0}, y)=\nabla_x c(x_{0}, y_{0})\}\text{ has at most } m \text{ elements.}
        \label{eq:mtwist}
    \end{equation}
    Suppose that $c$ is bounded, satisfies \cref{eq:mtwist}, and assume that any $c$-concave function is differentiable $\mu$-almost surely on its domain. If $\mu$ has no atom and $\mu$ and $\nu$ have finite transport cost, then each optimal plan $\pi\opt$ of \cref{eq:OT} is supported on the graphs of $k\leq m$ measurable maps, i.e.~there exists non-negative functions $\alpha_i:\Xx\to[0,1]$ and (Borel) measurable maps $T_i:\Xx\to\Yy$ such that
    \begin{equation*}
    \pi\opt=\sum_{i=1}^{k} \alpha_{i}(\id, T_i)\push \mu,
    \end{equation*}
    in the sense $\pi\opt(S)=\sum_{i=1}^k\int_\Xx\alpha_i(x)\chi_S(x,T_i(x))\dd\mu(x)$ for any Borel $S\subset \Xx\times\Yy$. We call such a plan a \emph{$m$-map}.
\end{proposition}

\begin{example}[$2$-twist for squared distance GW in dimension one]
\label{example:2-twist}
If $\Xx = \Yy = \RR$ and $c(x,y)= x^2 y^2 + \lambda xy$ for some $\lambda\neq0$, the $2$-twist condition holds and there exists an optimal correspondence $2$-map, i.e.~a plan that is supported on the graph of two maps. As we shall see in \Cref{subsec:quadra1D}, such costs are closely related to the GW problem with squared distance cost \cref{eqn:GW-squared-distance} in dimension one.
\end{example}

\begin{remark}[Simplifying assumptions]
    \label{rem:simple}
    When measures $\mu$ and $\nu$ have compact support and $\mu$ has a density---which are assumptions that we make in the following---, the conditions on measures $\mu$ and $\nu$ of \cref{prop:twist,prop:subtwist,prop:mtwist} required to apply the twist conditions are satisfied \cite[Remark 10.33]{villani2009optimal}.
\end{remark}

\subsubsection{Gromov--Monge maps for the GW problem}
Besides being an interesting mathematical question in itself and in addition to reducing the computational complexity of the problem, preferring Gromov--Monge maps in the Gromov--Wasserstein setting is interesting for registration purposes via mappings in imaging, shape analysis, machine learning, and simulation-based inference \cite{HurSimulationGWMonge22,zhang2022cycle}.
The interest of restricting the Gromov--Wasserstein problem to the class of mappings is discussed by \citeauthor{MemNeedham22} in \cite{MemNeedham22}, where it is shown that this formulation retains some properties of a metric (Theorem~3). Cases where doing so doesn't change the value of the \cref{eq:gw} optimization problem are discussed in \cite{memoli2022comparison}.\\
In sharp contrast with the optimal transportation problem, there are very few results that ensure the existence of a Gromov--Monge map for the GW problem, even in the particular cases considered in this work.
The following \cref{prop:sota-titouan,prop:sturm,prop:sota-titouan-quad} ensure the existence of a Gromov--Monge map under restrictive conditions, either on the symmetry of the measures $\mu$ and $\nu$ or on the rank of a matrix that depends on an optimal correspondence plan, which is nontrivial to check in practice as we usually have no such plan beforehand. For the inner product case, one has the following result by \citeauthor{vayer2020contribution}:

\begin{proposition}[Inner product GW, {\cite[Theorem~4.2.3]{vayer2020contribution}}]
\label{prop:sota-titouan}
    Let $n\geq d$ and $\mu, \nu\in \Pp(\RR^n)\times \Pp(\RR^d)$ two measures of finite second order moment with $\mu\ll\Leb_n$. Suppose that there exists a solution $\pi\opt$ of \cref{eqn:GW-inner-prod} such that $M\opt \defeq \int y x^\top  \dd\pi\opt(x, y)$ is of full rank. Then there exists an optimal correspondence map between $\mu$ and $\nu$, which can be written as
    $T=\nabla f\circ M\opt$ with $f: \RR^d \to \RR$ convex.
\end{proposition}
\noindent For the squared distance case, in \cite{titouan2019sliced} is claimed that in the discrete case in dimension one with uniform mass and same number of points $N$, the optimal solution of the discrete GW problem \cref{eq:QAP} would either be the identity $\sigma(i)=i$ or the anti-identity $\sigma(i)=N+1-i$ (Theorem~3.1). However, a counter-example to this assertion has been recently provided by \citeauthor{beinert2022assignment}~\cite{beinert2022assignment}.
To the best of our knowledge, the only positive results on the existence of Gromov--Monge maps for the squared distance cost are the following.
\begin{proposition}[Squared distance GW, {\cite[Theorem~9.21]{sturm2012space}}]
    \label{prop:sturm}
    Let $\mu,\nu\in\Pp(\RR^n)$. Assume that $\mu,\nu\ll\Leb_n$ and that both measures are rotationally invariant around their barycenter. Then every solution $\pi\opt$ of \cref{eqn:GW-squared-distance} is induced by a map, unique up to composition with rotations, which rearranges the radial distributions of $\mu$ and $\nu$  around their barycenter in a monotone non-decreasing way.
\end{proposition}

\begin{proposition}[Squared distance GW, {\cite[Proposition 4.2.4]{vayer2020contribution}}]
\label{prop:sota-titouan-quad}
    Let $n\geq d$ and $\mu, \nu\in \Pp(\RR^n)\times \Pp(\RR^d)$ with compact support.
    Assume that $\mu\ll\Leb_n$ and that both $\mu$ and $\nu$ are centered. Suppose that there exists a solution $\pi\opt$ of \cref{eqn:GW-squared-distance} such that $M\opt=\int y x ^\top\dd\pi\opt(x, y)$ is of full rank. Then there exists $f:\RR^d\to\RR$ convex such that $T=\nabla f \circ M\opt$ pushes $\mu$ to $\nu$. Moreover, if there exists a differentiable convex function $F:\RR\to\RR$ such that $\|T(x)\|^2=F'(\|x\|^2)$ $\mu$-a.e., then $T$ is an optimal correspondence map between $\mu$ and $\nu$.
\end{proposition}
It is of interest to note that \cref{prop:sturm} implies that the optimal correspondence map between two Gaussian measures in dimension one is a monotone map (non-decreasing or non-increasing). Yet, in higher dimensions, the optimal correspondence plan in the Gaussian case is not known unless more constraints on the structure of the plan are imposed \cite[Theorem~4.1]{salmona2021gromov}.

\subsection{Outline, contributions and perspectives}
\label{subsec:intro_contribs}

This work is organized in the following way.
\Cref{subsec:general_thm} provides a general setting in which the existence of optimal \emph{transport} maps can be shown for optimal transport problems with costs that are invariant on the fibers of some projection.
We provide two versions of the result: \cref{theo:fibers-abstract} requires no structure and is fairly general, and \cref{theo:fibers-main} imposes a more structured setting and therefore yields more structure for the optimal maps. The latter has the benefit of being more usable in further theoretical or computational analysis. Its proof requires a measurability argument which is addressed in detail in \cref{prop:selection-manifold}.
In \Cref{subsec:applications}, these two results are applied to the GW problem in Euclidean spaces. When the cost functions are inner products \cref{eqn:GW-inner-prod}, we establish the existence of Gromov--Monge maps (see \cref{theorem:inner-main}), improving on \cite[Theorem~4.2.3]{vayer2020contribution} (see \cref{prop:sota-titouan}) by dispensing with the rank assumption, difficult to check in practice. 
When the cost functions are squared distances \cref{eqn:GW-squared-distance}, we establish the existence of optimal $2$-maps in the general case (see \cref{theorem:quad-main}). Yet, we show that Gromov--Monge maps can occur as well in some specific settings that depend on the rank of some matrix, akin to Theorem~4.2.3 in \cite{vayer2020contribution}. 

An important question is to know if the result in the squared distance case \cref{eqn:GW-squared-distance} is tight, that is if one can prove or disprove the existence of a Gromov--Monge \emph{map} in general, rather than merely a \emph{$2$-map}. 
We address this question in \Cref{subsec:quadra1D} in the one-dimensional case. 
We first prove some structural properties of the optimal plan, then conduct a numerical exploration that suggests the tightness of \cref{theorem:quad-main} in dimension one. 
Namely, we numerically exhibit two measures that satisfy the assumptions of our \cref{theorem:quad-main} and for which the solution of \cref{eqn:GW-squared-distance} is a $2$-map but not a map. 
We also prove a positive result on the optimality of the monotone rearrangements, which partly explains why such plans are often optimal in practice. 
Notably, this result highlights the importance of the long-range effects of the squared distance cost function. 
In comparison, \citeauthor{sturm2012space}'s result \cite{sturm2012space} (\cref{prop:sturm} above) concerns symmetric measures and does not account for the frequent optimality of the monotone rearrangements in dimension one. 
As a possible extension, it is of interest to know if \citeauthor{sturm2012space}'s optimality result is stable with respect to the input measures. 
We show that this is not the case in dimension one: there exist small perturbations of symmetric measures (namely, the degenerate $\mu=\nu=\delta_0$) for which the monotone rearrangements are not optimal. 
Note that this result does not prove the instability in non-degenerate symmetric cases. 
We also exhibit empirically a non-symmetric case with instability.

Several questions and research directions are left open by our work. 
Among them, we may mention the following ones:
\begin{enumerate*}
    \item Characterize (or at least give examples of) cost functions that yield a proper Gromov--Wasserstein distance (that would notably equal $0$ if and only if the metric measure spaces are isomorphic, contrary to the inner product case, see discussion in \cref{sec:formulation}) and that guarantee the existence of Gromov--Monge maps. Note that such cost functions, although well-behaved in the sense of our study, may not fit every practical application. Indeed, being free to choose the cost regardless of whether it defines a distance, for instance to favor local or long-range effects, may be highly relevant since it directly affects the optimal correspondence plans, as well as the optimization problem.
    \item Establish sharper sufficient conditions for the optimality of the monotone rearrangements in dimension one for \cref{eqn:GW-squared-distance} than the one we obtain in \cref{subsec:quadra_1D_positive}.
    \item Assess the stability of optimal correspondence plans or maps with respect to the input measures to understand the impact of discretization, sampling, etc, for computational matters. For instance, when attempting to solve \cref{eqn:GW-squared-distance} between two measures with densities on $\RR$ in \cref{subsec:quadra1D}, we naturally rely on a finite-grid discretization of the real line to run computations and assume that the thus obtained optimal correspondence plan between our discretized measures (which is a $2$-map) is a faithful approximation of the optimal correspondence plan between the underlying measures (which is thus expected to be a $2$-map, though we have no formal proof of this).
    \item Design algorithms that leverage the knowledge our work gives on the structure of the optimizers of \cref{eqn:GW-inner-prod,eqn:GW-squared-distance} (roughly speaking, they can be parameterized by gradients of convex functions), for instance by learning such a map modeled as an Input Convex Neural Network~\cite{amos2017input} in the vein of~\cite{korotin2021neural,korotin2022neural}. 
\end{enumerate*}

\section{Existence of Monge maps for fiber-invariant costs}
\label{subsec:general_thm}

This section provides the main result on the existence of Monge maps for OT problems for which the cost satisfies an invariance property.
As detailed in \Cref{subsec:applications}, this property will be satisfied by the transport costs $C_{\pi\opt}$ arising from the first-order condition of \cref{eqn:GW-squared-distance} and \cref{eqn:GW-inner-prod}---see \Cref{subsec:intro_GW_and_OT}.

\subsection{Statement of the results}

We briefly present the main idea, also represented on \cref{fig:visu-fibers}. Let $\mu,\nu$ be two probability measures supported on a measurable space $(E,\Sigma_E)$ and consider a measurable map $\varphi : E \to B$, for some measurable space $(B,\Sigma_B)$, referred to as the \emph{base space} later on. 
We shall omit to mention the $\sigma$-algebras afterward. 
Let $(\mu_u)_{u \in B}$ and $(\nu_u)_{u \in B}$ denote disintegrations of $\mu$ and $\nu$ with respect to $\varphi$ (see \cref{sec:disintegration} or for instance \cite[Theorem~5.3.1]{ambrosio2005gradient} for a definition).
Let $c : E \times E \to \RR$ that is invariant on the fibers of $\varphi$,
that is $c(x,y) = \tilde{c}(\varphi(x), \varphi(y))$ for all $(x,y) \in E\times E$ and some cost function $\tilde{c}:B \times B\to\RR$.
Solving the OT problem between $\mu$ and $\nu$ on $E\times E$ with cost $c$ reduces to solving the OT problem between $\varphi\push \mu$ and $\varphi\push \nu$ on $B \times B$ with cost $\tilde{c}$.
Assuming that there exists a Monge map $t_B$ between $\varphi\push\mu$ and $\varphi\push\nu$,
we build a Monge map $T$ between $\mu$ and $\nu$ by (i) transporting each fiber $\mu_u$ onto $\nu_{t_B(u)}$ using a \emph{map} $T_u$, and (ii) gluing the $(T_u)_{u \in B}$ together to define a \emph{measurable} map $T$ satisfying $T \push \mu = \nu$. This map $T$ will be optimal as it coincides with $t_B$ on $B$ and the cost $c$ does not depend on the fibers $(\varphi^{-1}(u))_{u \in B}$.
We emphasize that ensuring the measurability of the map $T$ is non-trivial and crucial from a theoretical standpoint.

    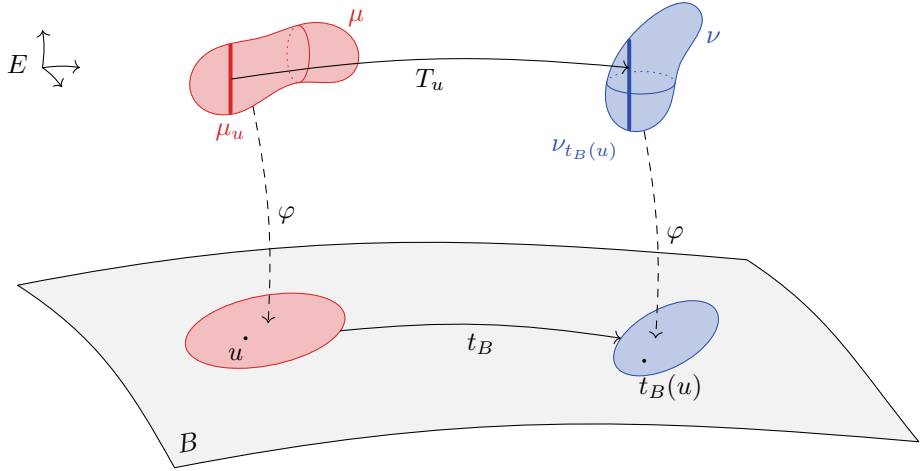
\begin{figure}[h]
        \centering
        \def\munushiftx{6}
\def\munushifty{0.7}
\def\Hlength{\munushiftx+2*\marginLR+0.2}
\def\marginLR{2}
\def\Hangle{1}
\def\Hheight{2.5}
\def\compXslant{0.3}
\def\colormu{tabred!30}
\def\colornu{tabblue!30}
\def\colorFibermu{tabred}
\def\colorFibernu{tabblue}
\def\colorT{tabpurple}
\begin{tikzpicture}[line cap=round,line join=round]
\tikzset{
    point/.style={
        thick,
        draw=gray,
        cross out,
        inner sep=0pt,
        minimum width=4pt,
        minimum height=4pt,
    },
}
\def\mushiftx{\marginLR cm}
\def\mushifty{4.2 cm }
\def\nushiftx{\marginLR cm +\munushiftx cm}
\def\nushifty{\munushifty cm + 4 cm - 0.5 cm}
\def\decalmu{0.5 cm}
\def\courbe{7}
\def\courbeSmall{7}
\def\asym{2}
\def\sides{1}
\begin{scope}[yshift=5cm,xshift=5mm]
\draw [->] (0,0) to[out=90-\courbeSmall,in=-90+\courbeSmall+\asym] (0,0.5);
\draw [->] (0,0) to[out=\courbeSmall,in=-180-\courbeSmall-\asym] (0.5,0);
\draw [->] (0,0) to[out=-45+\courbeSmall+\sides+\asym,in=-180-\courbeSmall-\sides-45] (0.26,-0.26);
\end{scope}
\begin{scope}[yscale=1,xscale=-1,xshift=-11.5 cm,rotate=-2]
\draw [closed, xshift=\compXslant cm-\Hheight*0.5 cm, fill=gray!10] (0.3,0) to[out=\courbe,in=-180-\courbe-\asym] (\Hlength,0) to[out=45+\courbe+\sides+10,in=180+45-\courbe-\sides-\asym] (\Hlength+\Hangle*\Hheight-0.5,\Hheight) to[out=-180-\courbe-\asym,in=\courbe] (\Hangle*\Hheight,\Hheight) to[out=-180+45-\courbe-\sides-\asym,in=45+\courbe+\sides] cycle;
\end{scope}
\draw [xslant=0.707, fill=\colormu, xshift=\compXslant cm+\decalmu-.4cm, yshift=.5cm,draw=\colorFibermu] (\marginLR,1) ellipse (1 and 0.5);
\draw [xslant=0.707, fill=\colornu, xshift=\compXslant cm+0.5cm, yshift=-.3cm,draw=\colorFibernu] (\marginLR+\munushiftx-1,1+\munushifty) ellipse (0.6 and 0.5);
\begin{scope}[yscale=1,xscale=-1,rotate=-2]
\node at (-2.2,-.35) [above right, rotate=10] {$B$};
\end{scope}
\node at (-.5+.5-.1,5+0.3) [below right] {$E$};
\draw[use Hobby shortcut,closed=true,fill=\colormu,xshift=\mushiftx+\decalmu,yshift=\mushifty,draw=\colorFibermu] (0,0.4) .. (0.6,0.2) .. (1.4,0.6) .. (2.1,0.7) .. (2.2,1) .. (1.6,1.4) .. (0.9,1.2) .. (0.2,1);
\draw[use Hobby shortcut,closed=true,fill=\colornu,xshift=\nushiftx-0.2cm,yshift=\nushifty,draw=\colorFibernu] (0.4,0) .. (0.2,0.6) .. (0.6, 1.3) .. (1,1.6) .. (1.4,1.6) .. (1.2,0.9) .. (1,0.2);
\node at (\marginLR+2.4,1.9)                           [below left, yshift=4cm, xshift=\decalmu,color=\colorFibermu] {$\mu$};
\node at (\marginLR+\munushiftx+1+0.2,0.4cm+\nushifty+1cm) [below right,color=\colorFibernu]            {$\nu$};

\draw[xshift=\mushiftx+14mm+\decalmu,yshift=\mushifty+6mm,draw=\colorFibermu]        (0,0) arc (-90:90:0.15 and 0.38);
\draw[xshift=\mushiftx+14mm+\decalmu,yshift=\mushifty+6mm,dotted,draw=\colorFibermu] (0,0.38*2) arc (90:270:0.15 and 0.38);
\draw[xshift=\nushiftx,yshift=\nushifty+6mm,draw=\colorFibernu]        (0,0) arc (-180:0:0.455 and 0.15);
\draw[xshift=\nushiftx,yshift=\nushifty+6mm,draw=\colorFibernu,dotted]        (0,0) arc (180:0:0.455 and 0.15);

\draw[\colorFibermu,line width=1.5,xshift=\mushiftx+5mm+\decalmu,yshift=\mushifty+2mm] (0,0.9) -- (0,0) node[below, \colorFibermu]      {$\mu_u$};
\draw[\colorFibernu,line width=1.5,xshift=\nushiftx+3mm,yshift=\nushifty] (0,1.16) -- (0,-0.02) node[below left, \colorFibernu] {$\nu_{t_B(u)}$};
\draw [dashed,->] (\mushiftx+1cm+\decalmu-0.2cm,  \mushifty+0.4cm-0.12cm) to[bend left=7] node[midway,right] {$\varphi$} (\mushiftx+1cm+\decalmu,  \mushifty-3cm+0.4cm) ;
\draw [dashed,->] (\nushiftx+0.5cm,\nushifty-0.05cm) to[bend left=7] node[midway,right] {$\varphi$} (\nushiftx+0.5cm+0.15cm,\nushifty-3cm+0.2cm) ;
\draw [->] (\mushiftx+2.06cm+\decalmu-0.1cm,\mushifty-2.9cm+0.2cm) to[bend left=7] node[midway,below] {$t_B$}          (\nushiftx-0.2cm+0.39cm,\nushifty-3.2cm+0.4cm) ;
\draw [->] (\mushiftx+5mm+\decalmu+0.2mm, \mushifty+6.5mm) to[bend left=7] node[midway,below] {$T_u$} (\nushiftx+3mm,\nushifty+8mm) ;
\draw [circle, inner sep=0pt, minimum size=1.5pt,xshift=\decalmu]  (2.7,0.9+0.5) node[fill, label={below left:$\phantom{(}u$}] {};
\draw [circle, inner sep=0pt, minimum size=1.5pt]  (8.5,1.1) node[fill, label={below right:$t_B(u)$}] {};

\end{tikzpicture}
        \caption{Illustration of the construction of a Monge map between $\mu$ and $\nu$: we optimally transport the projections of the measures in $B$ and then ``lift'' the resulting map $t_B$ to $E$ by sending each fiber $\mu_u$ onto the fiber $\nu_{t_B(u)}$, resulting from the disintegrations of $\mu$ and $\nu$ by $\varphi$.}
        \label{fig:visu-fibers}
    \end{figure}

We propose two theorems to formalize this idea: the first one (\cref{theo:fibers-abstract}) guarantees the existence of a Monge map for the \cref{eq:gw} problem in a fairly general setting, but its construction is quite convoluted and there is little to no hope that it can be leveraged in practice, either from a theoretical or computational perspective.
Assuming more structure (\cref{theo:fibers-main}), in particular on the fibers of $\varphi$, enables the construction of a Monge map for \cref{eq:gw} with a structure akin to the one of \Cref{prop:quad-cost-manifold-villani}.
As detailed in \Cref{subsec:applications}, both \cref{eqn:GW-squared-distance} and \cref{eqn:GW-inner-prod} fall in the latter setting.

\begin{theorem}
    \label{theo:fibers-abstract}
    Let $\Xx$ and $\Yy$ be two measurable spaces for which there exist two measurable maps $\Phi_\Xx : \Xx \to \RR^d$ and $\Phi_\Yy : \Yy \to \RR^d$
    that are injective, and whose inverses are measurable.
    Let $\mu \in \Pp(\Xx)$ and $\nu \in \Pp(\Yy)$ be two probability measures.
    Let $c : \Xx \times \Yy \to \RR$ and $B_{\Xx},B_{\Yy}$ be two measurable spaces along with measurable maps $\varphi : \Xx \to B_{\Xx}$ and $\psi : \Yy \to B_{\Yy}$.
    Assume that there exists a cost $\tilde{c} : B_{\Xx} \times B_{\Yy} \to \RR$ such that
    \begin{equation*}
        c(x,y)=\tilde c(\varphi(x),\psi(y))\qquad \text{ for all } (x,y)\in \Xx \times \Yy
    \end{equation*}
    and that there exists a Monge map $t_B : B_{\Xx} \to B_{\Yy}$ between $\varphi \push \mu$ and $\psi \push \nu$ for the cost $\tilde{c}$.
    Assume that there exists a disintegration $(\mu_u)_{u \in B_{\Xx}}$ of $\mu$ with respect to $\varphi$ such that $\varphi\push\mu$-a.e., $\mu_u$ is atomless.\\
    Then there exists a Monge map $T$ between $\mu$ and $\nu$ for the cost $c$. Furthermore, it can be chosen to project onto $t_B$ through $(\varphi,\psi)$, in the sense that $(\varphi,\psi)\push(\id,T)\push \mu = (\id, t_B) \push(\varphi \push \mu)$.
\end{theorem}
\noindent The proof of this theorem is provided in \Cref{subsec:proof-abstract}.

\begin{remark}[Atomless assumption]
The atomless assumption on the disintegration $(\mu_u)_u$ is a natural minimal requirement to expect the existence of a map (without specific assumption on the target measure $\nu$) and implies in particular that the fibers $(\varphi^{-1}(u))_{u \in B_{\Xx}}$ should not be discrete (at least $\varphi\push\mu$-a.e.).
Indeed, if for instance $\Xx = \Yy = B_{\Xx} = B_{\Yy} = \RR$ and $\varphi : x \mapsto \|x\|$, the fibers of $\varphi$ are of the form $\{-u, u\}$, hence the disintegrations $(\mu_u)_{u \geq 0}$ and $(\nu_u)_{u \geq 0}$ are discrete and given by $\mu_u(u) \delta_u + (1-\mu_u(u)) \delta_{-u}$ and $\nu_u(u) \delta_u + (1-\nu_u(u)) \delta_{-u}$, and there is no map $T_u$ between two such discrete measures unless we assume that $\mu_u(u) = \nu_u(u)$ or $1-\nu_u(u)$, $\varphi\push\mu$-a.e.\\
Observe also that $\varphi\push\mu$ may have atoms: as we assume the existence of the Monge map $t_B$, it would imply in that case that $\psi\push\nu$ also has atoms.
\end{remark}

\begin{remark}[Projection of Monge maps] The map $T$ projects onto $t_B$, in the sense of \cref{theo:fibers-abstract}, or equivalently
$\psi \circ T(x) = t_B \circ \varphi(x)$, for $\mu$-a.e.~$x$. A more general statement, that is ``any Monge map $T$ between $\mu$ and $\nu$ projects onto a Monge map between $\varphi\push\mu$ and $\varphi\push\nu$'' may not hold in general: $T$ may send two elements of the same fiber onto two different fibers and therefore not project onto a map downstairs.
The statement is however true if we can guarantee that there is a unique optimal transport plan between $\varphi\push\mu$ and $\psi\push\nu$ and that it is induced by a map $t_B$ (e.g.~if we can apply \citeauthor{brenier1987decomposition}'s \cref{theorem:brenier}). In that case, any $T$ necessarily projects onto $t_B$ in the aforementioned sense.
\end{remark}

Under additional assumptions, we can build a more structured Monge map.
Namely, as our goal is to apply \Cref{prop:quad-cost-manifold-villani}, we will assume that the (common) basis $B = B_{\Xx} = B_{\Yy}$ is a manifold, that \emph{almost all} the fibers of $\varphi : E \to B$ are homeomorphic to the \emph{same} manifold $F$, and that all source measures of interest $\mu$, $\mu_u$ and $\varphi\push\mu$ have a density.
We also introduce the following convention: if $\mu \in \Pp(E)$ for some measurable space $E$, $E' \subset E$, and $\varphi : E' \to B$, we let $\varphi\push\mu$ be the (non-negative) measure supported on $B$ defined by $\varphi\push\mu(A) \defeq \mu(\varphi^{-1}(A))$ for $A \subset B$ measurable.
Note that if $\mu(E') = 1$, then $\varphi\push\mu$ defines a probability measure on $B$.
This formalism allows us to state our theorem even when some assumptions only hold $\mu$-almost everywhere.
\begin{theorem}
    \label{theo:fibers-main}
    Let $E_0$ be a measurable space and $B_0$ and $F$ be complete Riemannian manifolds.
    Let $\mu,\nu \in \Pp(E_0)$ be two probability measures with compact support.
    Assume that there exists a set $E\subset E_0$ such that $\mu(E) = 1$ and that there exists a measurable map $\Phi : E \to B_0 \times F$ that is injective and whose inverse, when restricted to its image, is measurable as well.
    Let $p_B$ and $p_F$ denote the projections of $B_0 \times F$ on $B_0$ and $F$ respectively, and
    let $\varphi\defeq p_ B \circ \Phi: E\to B_0$.
    Let $c: E_0 \times E_0 \to \RR$ and suppose that there exists a cost function $\tilde c: B_0 \times B_0 \to \RR$ satisfying the twist condition such that
    \begin{equation*}
    c(x,y)=\tilde c(\varphi(x),\varphi(y))\quad \text{ for all } (x,y)\in E_0\times E_0.
    \end{equation*}
    Assume that $\varphi\push\mu\ll\vol_{B_0}$ and let thus $t_B$ denote the unique Monge map between $\varphi\push\mu$ and $\varphi\push\nu$ for this cost.
    To alleviate notation, we let $\mu' \defeq \Phi\push\mu$ and $\nu' \defeq \Phi\push\nu$ in the following.
    Suppose that there exists a disintegration $(\mu'_u)_{u\in B_0}$ of $\mu'$ by $p_B$ such that for $\varphi\push\mu$-almost every $u$, $\mu'_{u}\ll\vol_F$.\\
    Then there exists an optimal map $T:E\to E$ between $\mu$ and $\nu$ for the cost $c$, which can be decomposed as
    \begin{equation}\label{eq:structureMongeMap}
    \Phi \circ T\circ \Phi^{-1}(u,v)=\big(t_ B (u),t_ F (u,v)\big)=\big(\tilde c\text{-}\exp_u(\nabla f(u)), \exp_v(\nabla g_u(v))\big),
    \end{equation}
    where $f: B_0 \to\RR$ is $\tilde c$-convex, and $g_u: F \to\RR$ are $d_ F ^2/2$-convex for $\varphi\push\mu$-a.e.~$u$.
    Note that $t_ F$ could actually be any measurable function that sends each fiber $\mu'_u$ onto $\nu'_{\smash{t_B(u)}}$.
\end{theorem}
\noindent The proof of this theorem is provided in \Cref{subsec:proof-fibers-main}.
Let us give a simple example that illustrates the role played by our assumptions.
This example has connections with \cref{eqn:GW-squared-distance} as detailed in \Cref{subsec:applications_quadratic}.

\begin{example}[Squared distance GW]
\label{example:fiber-main}
Let $E_0 = \RR^d$ and $E = E_0\backslash \{0\}$, let $B_0 = \RR$ and $F = S^{d-1} = \{x \in E_0\mid \|x\|=1\}$ be the $(d-1)$-dimensional sphere. For convenience, we also introduce the space $B = \RR_{>0}$.
Consider the cost function $c(x,y) = (\|x\| - \|y\|)^2$, so that $c$ only depends on the norm of its entries.
The fibers of the map $x \mapsto \|x\|$ are spheres, with the exception of $x=0$, which invites us to consider the diffeomorphism
\begin{align*}
\Phi : E \to \RR_{>0} \times S^{d-1} = B \times F \subset B_0 \times F, \qquad
x \mapsto \big(\|x\|, x/\|x\|\big).
\end{align*}
From this, we can write $c(x,y) = \tilde{c}(\varphi(x),\varphi(y))$ where $\varphi(x) = \|x\|$ and $\tilde{c}(u,u') = (u - u')^2$ (which satisfies the twist condition).
Now, if $\mu$ has a density on $\RR^d$, so does $\mu'$ on $B_0 \times F$ as $\Phi$ is a diffeomorphism.
The coarea formula gives the existence of a disintegration $(\mu'_u)_{u \in B}$ of $\mu'$ by $p_B : (u,v) \mapsto u$ (note that $p_{B\pushonly}\mu' = \varphi\push\mu$ also has a density) such that all the $\mu'_u$ admit a density on $S^{d-1}$.\\
\cref{theo:fibers-main} thus applies, ensuring the existence of a structured Monge map between $\mu$ and (any) $\nu$ for the cost $c$. It decomposes for almost every $x \in \RR^d$ as a Monge map on the basis $B_0 = \RR$ (although it is actually only characterized on the image of $\varphi$, that is $B = \RR_{>0}$) obtained as the gradient of a convex function $f$ (there is no need for the exponential map here and $\nabla f$ is the non-decreasing mapping between the quantiles of $\varphi\push\mu$ and $\varphi\push\nu$) and a Monge map on each fiber $F = S^{d-1}$, also built from gradients of convex functions (via the exponential map on the sphere).\\
Note that our theorem only requires all the conditions to hold almost everywhere on $E_0 = \RR^d$, which is important since it allows one to ignore the singularity of $\varphi$ at $x = 0$.
\end{example}

\subsection{Proof of \texorpdfstring{\Cref{theo:fibers-abstract}}{Theorem 2.1}}
\label{subsec:proof-abstract}

The proof is composed of three steps.

\begin{itemize}[wide=0pt]
\item \textit{Step 1: Existence and optimality of lifts.} We know, by assumption, that there exists a Monge map $t_B$ that is optimal between the pushforward measures $\varphi\push\mu$ and $\psi \push\nu$.

As our goal is to build a Monge map between the initial measures $\mu$ and $\nu$, we first show that (i) there exists a transport plan $\pi \in \Pi(\mu,\nu)$ such that $(\varphi,\psi)\push\pi = (\id, t_B)\push\mu$ and that (ii) any such $\pi$ is an optimal transport plan between $\mu$ and $\nu$ for the cost $c$.
This is formalized by the following lemmas.

\begin{lemma}[Existence of a lift]
    \label{lemma:pullback-manifold}
    For any transport plan $\tilde\pi \in \Pi(\varphi\push\mu,\psi\push\nu)$, there exists a transport plan $\pi \in \Pi(\mu,\nu)$ such that $(\varphi,\psi)\push\pi=\tilde\pi$.
\end{lemma}
\begin{proof}
Let $(\mu_u)_{u\in B_{\Xx}}$ and $(\nu_v)_{v\in B_{\Yy}}$ be disintegrations of $\mu$ and $\nu$ by $\varphi$ and $\psi$, respectively.
Given $\tilde\pi \in \Pi(\varphi\push\mu,\psi\push\nu)$, we define
\begin{equation*}
\pi\defeq \iint _{ B_{\Xx}\times B_{\Yy}}(\mu_{u}\otimes \nu_{v})\dd \tilde\pi(u,v),
\end{equation*}
which belongs to $\Pi(\mu,\nu)$ since $\tilde{\pi}$ has $\phi\push\mu$ and $\psi\push\nu$ as marginals and, by disintegration, $\int_{B_{\Xx}} \mu_u \dd (\varphi\push\mu) = \mu$ (and same for $\nu$). 
The relation $(\varphi,\psi)\push\mu = \tilde{\pi}$ then follows from the fact that by disintegration $\mu_u$ is exactly supported on $\varphi^{-1}(u)$ so one has $\mu_u(\varphi^{-1}(U)) = \delta_U(u)$ for any measurable $U \subset B_{\Xx}$ (and same for $\nu_v$). 
\end{proof}

\begin{lemma}
    \label{lemma:decomp-manifold}
    Let $c: \Xx \times \Yy \to \RR$ and $\tilde c: B_{\Xx} \times B_{\Yy}\to \RR$ such that $$c(x,y)=\tilde c(\varphi(x),\psi(y))\quad \text{ for all } (x,y)\in \Xx \times \Yy.$$
    Then optimal plans for the base space cost $\tilde c$ are the projections of optimal plans for $c$: writing $\Pi\opt_{c}(\mu,\nu)$ the set of optimal transport plan between $\mu$ and $\nu$ for the cost $c$, and similarly for $\Pi\opt_{\tilde c}(\varphi\push\mu,\psi\push\nu)$, one has 
    \begin{equation*}
    \Pi\opt_{\tilde c}(\varphi\push\mu,\psi\push\nu)=(\varphi,\psi)\push\Pi\opt_{c}(\mu,\nu).
    \end{equation*}
\end{lemma}
\begin{proof}
Let us first remark that the relation between $c$ and $\tilde{c}$ implies that for any $\pi \in \Pi(\mu,\nu)$ and $\tilde\pi=(\varphi,\psi)\push\pi$, one has $\langle c,\pi\rangle = \langle \tilde c,\tilde\pi\rangle$.
Now, let $\tilde\pi\opt \in \Pi \opt _{\tilde c}(\varphi\push\mu,\psi\push\nu)$;
by \Cref{lemma:pullback-manifold}, there exists a $\pi \in \Pi(\mu,\nu)$ such that $(\varphi,\psi)\push\pi=\tilde\pi\opt$. 
Then for any $\gamma \in \Pi(\mu,\nu)$, one has $\langle c,\pi\rangle  =  \langle \tilde c,\tilde\pi\opt \rangle  \leq  \langle \tilde c,(\varphi,\psi)\push\gamma\rangle  =  \langle c,\gamma\rangle$,
hence the optimality of $\pi$.
Conversely, let $\pi \opt \in \Pi \opt _{c}(\mu,\nu)$;
by \Cref{lemma:pullback-manifold}, for any $\tilde\gamma  \in \Pi  (\varphi\push\mu,\psi\push\nu)$ there exists a $\gamma \in \Pi(\mu,\nu)$ such that $(\varphi,\psi)\push\gamma=\tilde\gamma$. 
We then have $\langle \tilde c,(\varphi,\psi)\push\pi \opt \rangle  = \langle c,\pi \opt \rangle  \leq  \langle c,\gamma\rangle  =  \langle \tilde c,\tilde\gamma\rangle$, 
hence the optimality of $(\varphi,\psi)\push\pi \opt$.
\end{proof}
\ 
\item \textit{Step 2: Existence of Monge maps between the fibers.} Using \Cref{lemma:pullback-manifold} with $\tilde{\pi} = (\id,t_B)\push(\varphi\push\mu)$, we know that we can build an optimal transportation plan $\pi \in \Pi(\mu,\nu)$ that essentially coincides with $t_B$ on $B_{\Xx} \times B_{\Yy}$ and transports each fiber $\mu_u$ onto $\nu_{t_B(u)}$ for $\mu$-a.e.~$u \in B_{\Xx}$.
To build a Monge map between $\mu$ and $\nu$, we must show that one can actually transport almost every $\mu_u$ onto $\nu_{t_B(u)}$ using a map rather than a plan.
For this, we use the following result by \citeauthor{santambrogio2015optimal}~\cite[Remark 1.23, Lemma 1.28, Corollary 1.29]{santambrogio2015optimal}.

\begin{proposition}
    \label{prop:fact-abstract}
    Let $\alpha,\beta$ be two measures supported on $\RR^d$ and suppose $\alpha$ atomless. Then:
    \begin{enumerate}[(i),nolistsep,leftmargin=*]
        \item if $d=1$, there exists a transport map $\tilde{T}$ that pushes $\alpha$ onto $\beta$. Furthermore, it is the \emph{unique} optimal map between these measures for the squared distance cost $(x,y) \mapsto \|x - y\|^2$;
        \item \label{item:sigma_d} there exists a map $\sigma_d : \RR^d \to \RR$ (that does not depend on $\alpha,\beta$) that is (Borel) measurable, injective, and its inverse is measurable as well.
    \end{enumerate}
\end{proposition}

As we assumed that the ground spaces $\Xx$ and $\Yy$ can be embedded in $\RR^d$ using the injective, measurable maps $\Phi_\Xx$ and $\Phi_\Yy$, we can apply \Cref{prop:fact-abstract} using $\sigma_\Xx = \sigma_d \circ \Phi_\Xx$ and $\sigma_\Yy = \sigma_d \circ \Phi_\Yy$.
As $\sigma_\Xx$ is injective, $\sigma_{\Xx\pushonly} \mu_u$ is atomless on $\RR$, and we can thus consider the \emph{unique} Monge map $\tilde{T}_u$ between $\sigma_{\Xx\pushonly}\mu_u$ and $\sigma_{\Yy\pushonly} \nu_{t_B(u)}$ for the squared distance cost on $\RR$.

From this, as the maps $\sigma_\Xx$ and $\sigma_\Yy$ are measurable and injective (thus invertible on their image), we can define $T_u = \sigma_\Yy^{-1} \circ \tilde{T}_u \circ \sigma_\Xx : \Xx \to \Yy$, which is a (measurable) transport map between $\mu_u$ and $\nu_{t_B(u)}$.\\

\item \textit{Step 3: Building a measurable global map.} Now that we have maps $(T_u)_u$ between each $\mu_u$ and $\nu_{t_B(u)}$, it may be tempting to simply define a given map $T : \Xx \to \Yy$ by $T(x) = T_{\varphi(x)}(x)$ whenever $\mu_{\varphi(x)}$ is atomless (which, by assumption, holds $\mu$-a.e.).
Intuitively, the map $T$ induces a transport plan $(\id, T)\push\mu$ satisfying $(\varphi,\psi)\push(\id,T)\push\mu = (\id, t_B)\push(\varphi \push \mu)$ on $B_{\Xx} \times B_{\Yy}$ and thus must be optimal according to \Cref{lemma:decomp-manifold}.
A crucial point, though, is to prove that this map $T$ can be defined in a measurable way.
For this, we use the following \emph{measurable selection theorem} due to \citeauthor{fontbona2010measurability}~\cite{fontbona2010measurability}.
\begin{proposition}[Measurable selection of maps, {\cite[Theorem~1.1]{fontbona2010measurability}}]
    \label{prop:fontbonaRd}
    Let $(B,\Sigma,m)$ be a $\sigma$-finite measure space and consider a measurable function
    $B \ni u \mapsto (\mu_u,\nu_u) \in \Pp(\RR^d)^2$.
    Let $c : \RR^d \times \RR^d \to \RR$ and
    assume that for $m$-a.e.~$u \in B$,
    there is a unique Monge map $T_u$ between $\mu_u$ and $\nu_u$ for the cost $c$.
    Then there exists a measurable function $(u,x) \mapsto T(u,x)$ such that for $m$-almost every $u$, $T(u,x) = T_u(x)$, $\mu_u$-a.e.
\end{proposition}

We can apply \cref{prop:fontbonaRd} in the case $d=1$ to the family of measures $(\sigma_{\Xx\pushonly} \mu_u, \sigma_{\Yy\pushonly} \nu_{t_B(u)})_{u \in B_{\Xx}}$, where the reference measure on $B_{\Xx}$ is $\varphi\push\mu$\footnote{Note that we cannot apply \cref{prop:fontbonaRd} to the measures $(\mu_u,\nu_{t_B(u)})_u$ and the maps $(T_u)_u$ directly, as $T_u$ may not be the unique Monge map between the measures, a required assumption of the proposition.}.
This family is indeed measurable: by definition of the measure disintegration, the map $v \in B_{\Yy} \mapsto \nu_{v}$ is measurable; and as the Monge map $t_B$ is measurable as well, so is the map $B_{\Xx} \ni u \mapsto \sigma_{\Yy\pushonly} v_{t_B(u)}$ as a composition of measurable maps, and thus so is the map $u \mapsto (\mu_u, \nu_{t_B(u)})$.
\cref{prop:fontbonaRd} therefore guarantees the existence of a measurable map $\smash{\tilde{T}} : B_{\Xx} \times \RR \to \RR$ such that $\smash{\tilde{T}}(u,x) = \smash{\tilde{T}}_u(x)$ for $\varphi\push\mu$-almost every $u$ and $\sigma_{\Xx\pushonly}\mu$-almost every $x$. Now, we can define
\begin{align*}
    T : \Xx \to \Yy, \qquad
         x \mapsto \sigma_\Yy^{-1} \circ \tilde{T}(\varphi(x), \sigma_\Xx(x)).
\end{align*}
This map is measurable as it is a composition of measurable maps.
Let us prove that it is a transport map between $\mu$ and $\nu$.
For any continuous function $\zeta : \Yy \to \RR$ with compact support, we can write
\begin{equation*}
    \int_\Yy \zeta \dd T\push\mu = \int_\Xx \zeta\circ T\dd \mu
    = \int_{u \in B_{\Xx}} \int_{x \in \varphi^{-1}(u)} \big(\zeta\circ \sigma_\Yy^{-1}\circ \tilde{T}_u\circ \sigma_\Xx\big)(x) \dd \mu_u(x) \dd \varphi \push \mu(u),
\end{equation*}
where we use the disintegration of $\mu$ with respect to $\varphi$ and the fact that the $\mu_u$ are supported on $\varphi^{-1}(u)$, allowing us to write $\tilde{T}(\varphi(x),\sigma_\Xx(x)) = \tilde{T}_u(\sigma_\Xx(x))$ on that fiber ($\varphi\push\mu$-a.e.).
Now, recall that $T_u =\sigma_\Yy^{-1}\circ \tilde{T}_u \circ \sigma_\Xx$ defines a transport map between $\mu_u$ and $\nu_{t_B(u)}$. In particular, the image of the fiber $\varphi^{-1}(u)$ by this map is $\psi^{-1}(t_B(u)) \subset \Yy$.
Since $\nu_{t_B(u)}$ is supported on $\psi^{-1}(t_B(u))$, we get
\begin{align*}
    \int_\Yy \zeta \dd T\push \mu 
    = \int_{u \in B_{\Xx}} \int_{y \in \Yy} \zeta(y) \dd \nu_{t_B(u)}(y) \dd \varphi\push\mu(u)= \int_{v \in B_{\Yy}} \int_{y \in \Yy} \zeta(y) \dd \nu_{v}(y) \dd \psi\push\nu(v),
\end{align*}
where we use the fact that $t_B$ pushes $\varphi\push\mu$ to $\psi\push\nu$; by disintegration of $\nu$ with respect to $\psi$, we then have $T \push \mu = \nu$.

By \cref{lemma:decomp-manifold}, this map $T$ is optimal if and only if it satisfies $(\varphi,\psi)\push(\id,T)\push \mu = (\id, t_B)\push(\varphi\push\mu)$.
To show this, let $\zeta : \Xx \times \Yy \to \RR$ be a continuous function with compact support. We have
\begin{align*}
    \int_{B_{\Xx} \times B_{\Yy}} \zeta(u,v)& \dd (\varphi,\psi)\push(\id,T)\push \mu(u,v)
    = \int_\Xx \zeta\big(\varphi(x),\psi\circ T(x)\big) \dd \mu(x) \\
    &= \int_{u \in B_{\Xx}} \int_{x \in \varphi^{-1}(u)} \zeta\big(u, (\psi\circ\sigma_\Yy^{-1}\circ \tilde{T}_u \circ \sigma_\Xx (x)\big) \dd \mu_u(x) \dd \varphi \push \mu(u) \\
    &= \int_{u \in B_{\Xx}} \int_{y \in \psi^{-1}(t_B(u))} \zeta\big(u, t_B(u)\big) \dd \nu_{t_B(u)}(y) \dd \varphi \push \mu(u) \\
    &= \int_{u \in B_{\Xx}} \zeta(u, t_B(u)) \dd \varphi\push\mu(u) \\
    &= \int_{B_{\Xx} \times B_{\Yy}} \zeta(u,v) \dd (\id,t_B)\push \varphi \push\mu(u,v),
\end{align*}
proving the required equality and thus that $T$ is a Monge map between $\mu$ and $\nu$.
\end{itemize}

\subsection{Proof of \texorpdfstring{\Cref{theo:fibers-main}}{Theorem 2.4}}
\label{subsec:proof-fibers-main}

We recall that we note $\mu' \defeq \Phi\push\mu$ and $\nu' \defeq \Phi\push\nu$.
We also denote by $B$ the image of $\varphi = p_B \circ \Phi$, so that $\mu',\nu'$ are supported on $B \times F \subset B_0 \times F$.

\begin{itemize}[wide=0pt]
    \item \textit{Step 1: Construction of the structured Monge map.} Given that $\varphi\push\mu\ll\vol_{B_0}$ and that $B_0$ is a complete (separable) Riemannian manifold, by \cref{prop:twist} and \cref{rem:simple} there exists a unique optimal transport plan $\pi \opt_{ B }$ between $\varphi\push\mu$ and $\varphi \push \nu$ for the cost $\tilde c$ and it is induced by a map $t_{ B }: B_0 \to B_0$ of the form $t_{ B }=\tilde c\text{-}\exp_u(\nabla f)$, with $f$ being $\tilde c$-convex on $B_0$.
    
By \cref{lemma:decomp-manifold}, we know that any transport plan $\pi \in \Pi(\mu,\nu)$ satisfying $(\varphi,\varphi) \push \pi = (\id, t_B)\push\mu$ must be optimal.
Therefore, if $\pi$ happens to be induced by a map $T$, that is $\pi = (\id, T)\push\mu$, we will obtain a Monge map between $\mu$ and $\nu$.
To build such a $T$, we proceed as in \Cref{subsec:proof-abstract}: we define a Monge map $T_u$ between $\mu'_u$ and ${\nu'_{\smash{t_B(u)}}}$ for $\varphi\push\mu$-a.e.~$u$
and build a global map between $\mu'$ and $\nu'$ by (roughly) setting $T(u,x) = T_u(x)$.
As explained in \Cref{subsec:proof-abstract}, proving the measurability of such $T$ requires care.

\item \textit{Step 2: Transport between the fibers.} For $\varphi\push\mu$-a.e.~$u$, $\mu'_u\ll \vol_F$ and the optimal cost between $\mu'_u$ and $\nu'_{\smash{t_ B (u)}}$ is finite by assumption.
Whenever $\mu'_u$ has a density, we can therefore apply \cref{prop:quad-cost-manifold-villani} between $\mu'_u$ and $\nu'_{\smash{t_ B (u)}}$ with the cost $d_{ F }^2$ to establish the existence of a plan $\pi_u$ between these fibers that is induced by a map $T_u: F \to F$, which can be expressed as $T_u(v)=\exp_v(\nabla g_u(v))$ with $g_u$ being $d_ F ^2/2$-convex on $F$.

\item \textit{Step 3: Measurability of the global map.} Now that we have built structured maps $T_u$ between corresponding fibers (through $t_B$), it remains to prove the existence of a measurable map $T : B_0 \times F \to B_0 \times F$ transporting $\mu'$ onto $\nu'$ satisfying $T(u,x) = (t_B(u),T_u(x))$ for $\varphi\push\mu$-almost every $u$ and $\mu'_u$-almost every $x$.
For this, we need an adaptation of \cref{prop:fontbonaRd} to the manifold setting.
Namely, we have the following:

\begin{proposition}[Measurable selection of maps, manifold case]
    \label{prop:selection-manifold}
    Let $M$ be a complete Riemannian manifold and $(B, \Sigma, m)$ a measure space.
    Consider a measurable function $B\ni u \mapsto (\mu_u, \nu_u) \in \Pp(M)^2$. Assume that for $m$-almost every $u\in B$, $\mu_u\ll\vol_M$ and $\mu_u$ and $\nu_u$ have a finite transport cost. Let $T_u$ denote the (unique by \cref{prop:quad-cost-manifold-villani}) optimal transport map induced by the squared geodesic distance cost $d_M^2$ on $M$ between $\mu_u$ and $\nu_u$.
    Then there exists a measurable function $(u,x)\mapsto T(u,x)$ such that for $m$-almost every $u$, $T(u,x)=T_{u}(x)$, $\mu_{u}$-a.e.
\end{proposition}

This can be shown by adapting the proof of Theorem~1.1 in \cite{fontbona2010measurability} to the manifold setting, and most steps adapt seamlessly.
We provide a sketch of proof below.
A complete proof, where we stress the points that need specific care in adaptation, is deferred to \Cref{sec:appendix-proof-fontbona-manifold}.

\begin{proof}[Sketch of proof of \cref{prop:selection-manifold}]
The map $(u,x) \mapsto T(u,x)$ is defined as being the limit of a Cauchy sequence (on a complete metric space) of measurable maps $(u,x) \mapsto \smash{T^{(k)}}(u,x)$ which are built as $L^1$-approximations of the $T_u(x)$, yielding the relation $T(u,x) = T_u(x)$ for a.e.~$u,x$ when $k \to \infty$. 
More precisely, given a partition $(A_{nk})_n$ of $M$ made of cells of diameters smaller than $2^{-k}$, and given $a_{nk} \in A_{nk}$ any fixed point, $\smash{T^{(k)}}$ is precisely defined as $\smash{T^{(k)}}(u,x) = a_{nk}$, where $a_{nk}$ is the single point such that $T_u(x) \in A_{nk}$. 
Since on each set of the form $\{(u,x) \mid T_u(x) \in A_{nk} \}$, one has $d_M(\smash{T^{(k)}}(u,x), T_u(x)) \leq 2^{-k}$ (the diameter of the cell), it follows that $\smash{T^{(k)}}(u,x) \to T_u(x)$, at least for $m$-a.e.~$u$ and $\mu_u$-a.e.~$x$. 

The key point is to ensure the measurability of the map $\smash{T^{(k)}}$, which boils down to proving that the set $\{(u,x) \mid T_u(x) \in A_{nk}\}$ is measurable. 
This is possible by relying on the theory of measurable set-valued maps \cite[Chapters 5 and 14]{rockafellar2009variational}, that we carefully adapt from the Euclidean to the manifold setting in \cref{sec:measurable_set_valued}.
\end{proof}

We can apply this proposition with the manifold $M$ being the common fiber $F$ on which the $\mu'_u,\smash{\nu'_{t_ B (u)}}$ are supported for $\varphi\push\mu$-a.e.~$u$, and for which we have access to the (unique) Monge map $T_u$.
It gives the existence of a global map $t_F$ satisfying $t_F(u,v) = T_u(v)$ for $\varphi\push\mu$-a.e.~$u$, and $\mu'_u$-a.e.~$v$, and we can thus define the measurable map $T(u,x) = (t_B(u), t_F(u,x))$.

An elementary computation using the pushforward and disintegration definitions as well as the facts that $T_{u\pushonly}(\mu'_{u})=\nu'_{\smash{t_ B (u)}}$ and $t_{ B \pushonly}(\varphi\push\mu)=\varphi\push\nu$ then guarantees that
$T$ sends $\mu'$ to $\nu'$, and $T_E\defeq\Phi^{-1}\circ T\circ\Phi$ therefore sends $\mu$ to $\nu$. Since $(\varphi,\varphi)\push(\id,T_E)\push\mu=\pi_{ B }\opt$,
by \cref{lemma:decomp-manifold}, $T_E$ is an optimal map between $\mu$ and $\nu$.
\end{itemize}

\section{Applications to the inner product and squared distance GW problems}
\label{subsec:applications}

In this section, we apply the results of \Cref{subsec:general_thm} to our two problems of interest, \cref{eqn:GW-inner-prod} and \cref{eqn:GW-squared-distance}, to obtain the existence of Gromov--Monge maps for these GW problems in suitable settings. 
The proofs can be sketched in the following way: for both of these problems,
\begin{itemize}[nolistsep,wide,leftmargin=*]
    \item Let $\pi\opt$ be an optimal plan for the GW problem.
    \item Consider the associated linearized problem \cref{eq:linearized} at $\pi\opt$, which is a standard optimal transportation problem. From \Cref{subsec:intro_GW_and_OT}, we know that any solution of this OT problem ($\pi\opt$ being one of them) must be a solution of the initial GW problem.
    \item Observe that the linearized versions of \cref{eqn:GW-inner-prod} and \cref{eqn:GW-squared-distance} reduce to optimal transportation problems with costs satisfying the assumptions of \cref{theo:fibers-main}, thus guaranteeing the existence of structured Gromov--Monge maps.
\end{itemize}
As such, we only focus on detailing this last point in the following subsections. 

\begin{remark}[Non-uniqueness for GW]
Our results provide the existence of Gromov--Monge maps for the GW problems, but their construction depends on the linearization (that is, of the initial $\pi\opt$ selected). Thus, even though we would have uniqueness of solutions for the linearized problem as given e.g.~by the subtwist condition, this does not give uniqueness of the solution to the corresponding GW problem: the GW problem may have several optimizers, each leading to a different linearized problem for which solutions, unique or not, may differ. 
\end{remark}

\subsection{The inner product cost}
\label{subsec:applications_innerProduct}
    Expanding the integrand in \cref{eqn:GW-inner-prod} and using the fact that $\iint\langle x,x'\rangle^2 \dd\pi\dd\pi=\iint\langle x,x'\rangle^2 \dd\mu\dd\mu$ is constant (the same goes for the terms that depend on $\nu$), one gets the equivalent problem
    \begin{equation*}
        \min _{\pi \in \Pi(\mu,\nu)} \iint-\langle x, x'\rangle\langle y, y'\rangle \dd\pi(x, y) \dd\pi(x', y').
    \end{equation*}
    This problem is not invariant to translations but it is to the action of $O_n(\RR) \times O_d(\RR)$. Any optimal correspondence plan $\pi\opt$ is also an optimal transport plan for the linearized problem \cref{eq:linearized} with cost
    \begin{align*}
        C_{\pi\opt}(x,y)=-\int\langle x,x'\rangle\langle y, y'\rangle \dd\pi\opt(x', y')=-\Big\langle\int(y' x'^\top)x \dd\pi\opt(x', y'),y\Big\rangle\defeqinv -\langle M\opt x,y\rangle,
    \end{align*}
    where $M\opt\defeq\int y' x'^\top \dd\pi\opt(x', y')\in \RR^{d\times n}$. This linearized cost satisfies the twist condition if and only if $M\opt$ is of full rank, hence in this case $\pi\opt$ is the only solution of the linearized problem and it is induced by a map, and Theorem~4.2.3 from \cite{vayer2020contribution} gives a result on the structure of this map. We can actually generalize this to the case where $M\opt$ is arbitrary, as follows.
    \begin{theorem}[Existence of a Gromov--Monge map for GW with inner product cost]
        \label{theorem:inner-main}
        Let $n\geq d$ and $\mu, \nu\in \Pp(\RR^n)\times\Pp(\RR^d)$ two measures with compact supports. Suppose that $\mu\ll\Leb_n$. Then there exists an optimal correspondence map $T:\RR^n\to\RR^d$ for \cref{eqn:GW-inner-prod}, which can be written as
        \begin{equation}
            T=O_\Yy^\top \circ T_0\circ p_{\RR^d}\circ O_\Xx,
            \label{eq:T0-form-1}
        \end{equation}
        where $(O_\Xx,O_\Yy)\in O_n(\RR)\times O_d(\RR)$, $p_{\RR^d}: \RR^n \to \RR^d$ projects on the $d$ first coordinates, 
        and
        \begin{equation}
            T_0:\RR^d\to\RR^d,\qquad (x_1,\dots,x_d)\mapsto\big(\nabla f\circ\Si(x_1,\dots,x_h), \nabla g_{x_1,\dots,x_h}(x_{h+1},\dots,x_d)\big),
            \label{eq:T0-form-2}
        \end{equation}
        where $h\leq d$, $\Si\in \RR^{h\times h}$ is diagonal with positive entries, and $f: \RR^h \to \RR$ and all the $g_{x_1,\dots,x_h}: \RR^{d-h}\to\RR$ are convex functions.
    \end{theorem}
    \noindent To show this, we will need two simple lemmas that essentially state that \citeauthor{brenier1987decomposition}'s theorem on optimal \emph{transport} maps holds when the inner product is perturbed by homeomorphisms. We state them now and prove them in \cref{subsec:proof-scaled-brenier}.
        \begin{lemma}
            \label{lemma:reparam}
            Let $\mu,\nu\in \Pp(E)$ and let $\psi_1,\psi_2:E\to F$ be homeomorphisms. Let $\tilde c:F\times F\to\RR$ and consider the cost $c(x,y)= c( \psi_1(x),\psi_2(y))$. Then a map is optimal for the cost $c$ between $\mu$ and $\nu$ if and only if it is of the form $\psi_2^{-1}\circ T\circ\psi_1$ with $T:F\times F$ optimal for the cost $\tilde c$ between $\psi_{1\pushonly}\mu$ and $\psi_{2\pushonly}\nu$.
        \end{lemma}
        \noindent The second lemma is a simple corollary of the first:
        \begin{lemma}
            \label{lemma:scaled-Brenier}
            Let $h\geq 1$ and $\mu,\nu\in \Pp(\RR^h)$ with compact supports and such that $\mu\ll\Leb_h$. Consider the cost $c(x, y)= -\langle \psi_1(x),\psi_2(y)\rangle$ where $\psi_1,\psi_2:\RR^h\to \RR^h$ are diffeomorphisms.
            Then there exists a unique optimal transport plan between $\mu$ and $\nu$ for the cost $c$, and it is induced by a map $t:\RR^h\to \RR^h$ of the form $t=\psi_2^{-1}\circ\nabla f\circ\psi_1$, with $f:\RR^h\to\RR$ convex.
        \end{lemma}

    \noindent We now prove \cref{theorem:inner-main}.
    \begin{proof}[Proof of \cref{theorem:inner-main}]
        Using a singular value decomposition, we write $M\opt=O_\Yy^\top \Si O_\Xx\in\RR^{d\times n}$ with $(O_\Xx,O_\Yy)\in O_n(\RR)\times O_d(\RR)$ orthogonal matrices and $\Si\in\RR^{d\times n}$ diagonal with non-negative entries. The cost $C_{\pi\opt}$ then becomes
        \begin{equation*}
        C_{\pi\opt}(x, y)=-\langle \smash{O_\Yy^\top}\Si O_\Xx x,y\rangle =-\langle \Si O_\Xx x, O_\Yy y\rangle.
        \end{equation*}
        Using \cref{lemma:reparam}, the problem becomes an optimal transportation problem between $\mu'\defeq O_{\Xx\pushonly}\mu$ and $\nu'\defeq O_{\Yy\pushonly}\nu$. Up to permutation of the rows and columns of $O_\Xx$ and $O_\Yy$, we can assume $\si_1 \geq \dots \geq \si_h>0$ with $h\defeq \rk(M\opt)\leq d$, and the problem therefore becomes $\min_{\tilde\pi} \langle c_\Si, \tilde\pi\rangle$ for $\tilde\pi \in \Pi(\mu', \nu')$, where $c_\Si(\tilde x,\tilde y)=-\sum_{i=1}^h \si_i \tilde x_i \tilde y_i\defeqinv \tilde c( p(\tilde x), p(\tilde y))$, $p$ being the orthogonal projection on $\RR^h$.
        We reduce to the case where both measures live in the same space by noting that since $c_\Si(\tilde x,\tilde y)=c_\Si(p_{\RR^d}(\tilde x), \tilde y)$ for all $\tilde x$ and $\tilde y$, any map $T_0$ optimal between $\mu''\defeq p_{\RR^d\pushonly}\mu'$ and $\nu'$ will induce a map $T=T_0\circ p_{\RR^d}$ optimal between $\mu'$ and $\nu'$
        (by \cref{lemma:decomp-manifold}).
        One can then induce an optimal map between $\mu$ and $\nu$ by composing with $O_\Xx$ and $O_\Yy^\top$ (\cref{lemma:reparam}), hence \Cref{eq:T0-form-1}.
    The existence of such a map $T_0$, optimal between $\mu''$ and $\nu'$  and satisfying \cref{eq:T0-form-2}, follows from the application of \cref{theo:fibers-main} for $E\defeq E_0\defeq \RR^d=\RR^h\times\RR^{d-h}\defeqinv B_0\times F$ and $\varphi\defeq p$. Indeed, $B_0$ and $F$ are complete Riemannian manifolds, the cost $\tilde c$ satisfies the twist condition on $B_0\times B_0$, $p\push\mu''\ll\Leb_h$, and $\mu''_u\ll\Leb_{d-h}$ as a conditional probability for a.e.~$u$.\\
    We then make $t_B$ explicit.
    One has that $c_\Si(x,y)=-\langle \tilde\Si x,y\rangle$, where $\tilde\Si=\diag({\si_i})_{1\leq i\leq h}$. As $p\push\mu''$ has a density, we can apply \cref{lemma:scaled-Brenier} stated above with $(\psi_1,\psi_2)=(\tilde\Si,\id)$ to obtain the existence of a unique optimal transport plan $\pi_B\opt$ between $p\push\mu''$ and $p\push\nu'$ for the cost $c_\Si$, which is induced by a map $t_B:B\to B$ of the form $t_B=\nabla f\circ\tilde\Si$, with $f$ convex.
    \end{proof}
    \begin{remark}[Recovering Theorem~4.2.3 from \cite{vayer2020contribution}]
            Applying \cref{theorem:inner-main} when $h=d$, the optimal correspondence map between $O_{\Xx\pushonly}\mu$ and $O_{\Yy\pushonly}\nu$ can be written as $T_0\circ p_{\RR^d}$ with $T_0=\nabla f\circ \tilde\Si$. The induced optimal correspondence map between $\mu$ and $\nu$ is then
                \begin{align*}
                    T=\smash{O_\Yy^\top}\circ (\nabla f\circ \tilde\Si\circ p_{\RR^d})\circ O_\Xx 
                    =\smash{O_\Yy^\top}\circ (\nabla f\circ \Si)\circ O_\Xx 
                    = \nabla( f\circ O_\Yy)\circ \smash{O_\Yy^\top}\circ \Si \circ O_\Xx = \nabla\tilde f\circ M\opt,
                \end{align*}
            where $\tilde f\defeq f\circ O_\Yy$ is convex.
        \end{remark}

\subsection{The squared distance cost}
\label{subsec:applications_quadratic}
        The \cref{eqn:GW-squared-distance} problem
        is invariant by translation of $\mu$ and $\nu$. With no loss of generality, we assume that both measures are centered. Expanding the integrand gives
        \begin{equation*}
        \big|\|x-x'\|^2-\|y-y'\|^2\big|^2=\|x-x'\|^4+\|y-y'\|^4-2\|x-x'\|^2\|y-y'\|^2,
        \end{equation*}
        and the first two terms only depend on $\mu$ and $\nu$, not on $\pi$. Expanding the remaining term yields nine terms. Two of them also lead to a constant contribution: $-\|x\|^2\|y'\|^2$ and $-\|x'\|^2\|y\|^2$. Four lead to vanishing integrals since $\mu$ and $\nu$ are centered: $2\|x\|^2 \langle y,y'\rangle$, $2\|x'\|^2 \langle y,y'\rangle$, $2\|y\|^2 \langle x,x'\rangle$ and $2\|y'\|^2 \langle x,x'\rangle$. The remaining three terms then yield the following equivalent problem:
        \begin{equation*}
            \min _{\pi \in \Pi(\mu,\nu)} \int-\|x\|^2\|y\|^2\dd\pi(x, y)+2\iint-\langle x, x'\rangle\langle y, y'\rangle \dd\pi(x, y) \dd\pi(x', y').
        \end{equation*}
        Any optimal correspondence plan $\pi\opt$ is also an optimal transport plan for the linearized problem \cref{eq:linearized} with cost
        \begin{align*}
            C_{\pi\opt}(x,y)=-\|x\|^2\|y\|^2-4\int\langle x, x'\rangle\langle y, y'\rangle \dd\pi\opt(x', y') \defeqinv-\|x\|^2\|y\|^2-4\langle M\opt x,y\rangle,
        \end{align*}
        where $M\opt\defeq\int y'x'^\top \dd\pi\opt(x', y')\in \RR^{d\times n}$.
        In the cases where the rank of $M\opt$ is $d$, this linearized cost satisfies both the subtwist and the $2$-twist conditions, yielding an optimal $2$-map that is also a map/anti-map, by compactness of the support of $\mu$ and $\nu$ and when $\mu$ has a density. Similarly, when the rank of $M\opt$ is $d-1$, the cost satisfies the $2$-twist condition, yielding an optimal $2$-map structure. In the case where $\rk M\opt\leq d-2$, none of the twist conditions hold and there is \textit{a priori} no obvious reason for the existence of an optimal correspondence map. Perhaps surprisingly, this property can actually be guaranteed.
        \begin{theorem}[Existence of a Gromov--Monge map or $2$-map for GW with squared distance cost]
        \label{theorem:quad-main}
            Let $n\geq d$ and $\mu, \nu\in \Pp(\RR^n)\times\Pp(\RR^d)$ two measures with compact supports. Suppose that $\mu\ll\Leb_n$. Let $\pi\opt$ be a solution of \cref{eqn:GW-squared-distance} and $M\opt\defeq\int y' x'^\top \dd\pi\opt(x', y')$. Then:
            \begin{enumerate}[(i),leftmargin=*]
                \item if $\rk M\opt=d$, there exists an optimal correspondence $2$-map, which is also a map/anti-map;
                \item if $\rk M\opt=d-1$, there exists an optimal correspondence $2$-map;
                \item if $\rk M\opt\leq d-2$, there exists an optimal correspondence map, which can be written as
                    \begin{equation*}
                        T:\RR^n\to\RR^n,\qquad T=O_\Yy^\top \circ T_0\circ O_\Xx,\footnote{We consider $\nu$ as a measure on $\RR^n$ with $d$-dimensional support.}
                    \end{equation*}
                    where $(O_\Xx,O_\Yy)\in O_n(\RR)^2$. Writing any $x\in\RR^n$ as $x=(x_H,x_\perp)\in\RR^h\times\RR^{n-h}$ and defining a function $\Phi$ akin to a partial polar change of coordinates
                    \begin{equation}
                        \label{eq:bigPhi}
                        \Phi: \RR^n\to (\RR^h\times \RR_{\geq 0})\times S^{n-h-1},\qquad x\mapsto \big(\smash{\,\underbrace{(x_H,\|x_\perp\|^2)}_{x_B}},\,\smash{\underbrace{x_\perp/\|x_\perp\|}_{x_F}}\,\big),
                    \end{equation}
                    one can write the map $T_0$ as
                    \begin{equation*}
                        T_0:\RR^n\to\RR^n,\qquad \Phi\circ T_0(x)=\big(\tilde c\text{-}\exp_{x_B}(\nabla f(x_B)), \exp_{x_F}(\nabla g_{x_B}(x_F))\big),
                    \end{equation*}
                    where explicit expressions for $\tilde c$ and $\tilde c$-$\exp$ are given in \cref{eq:ctilde} and \cref{eq:ctildeexp} respectively, $h=\rk M\opt\leq d-2$, $f: \RR^{h+1} \to \RR$ is $\tilde c$-convex and all the $g_{x_B}: S^{n-h-1} \to \RR$ are $d_{S^{n-h-1}}^2/2$-convex.
            \end{enumerate}
            Therefore, regardless of the rank of $M\opt$, there always exists an optimal correspondence $2$-map. 
        \end{theorem}

        \begin{proof}[Proof of (i)]
            We show that in the case where $M\opt$ is of full rank, both the subtwist and the $2$-twist conditions are satisfied. For the subtwist condition, remark that the equation
            \begin{equation*}
            0= \nabla_x C_{\pi\opt}(x,y_1)-\nabla_x C_{\pi\opt}(x,y_2)=-2\big(\|y_1\|^2-\|y_2\|^2\big)x-4(M\opt)^\top(y_1-y_2)
            \end{equation*}
            has no solution in $x$ if $\|y_1\|=\|y_2\|$ and a unique one otherwise. For the $2$-twist condition, remark that any $y\in\Yy$ satisfying $-\big(\|y\|^2-\|y_0\|^2\big)x_0-2 (M\opt)^\top(y-y_0)=0$ is fully determined by $r\defeq\|y\|^2-\|y_0\|^2$ as
            \begin{equation}
                \label{eq:y-determined}
                y=y_0-\frac12 r(M\opt)^{-1}x_0.
            \end{equation}
            Taking the norm in \cref{eq:y-determined} shows that $r$ is a solution of a polynomial of degree two, and therefore that there are always at most two such points $y$. Both subtwist and $2$-twist conditions are therefore satisfied. The uniqueness of optimal plans ensured by the subtwist condition then guarantees that the unique optimal transport plan for the linearized cost is induced both by a $2$-map and a map/anti-map.

            \noindent \textit{Proof of (ii).}
            We show that in the case where $\rk M\opt= d-1$, the $2$-twist condition is satisfied. Let $y\in\Yy$ such that $\|y\|^2 x_{0}+ (M\opt)^\top y=\|y_{0}\|^2 x_{0}+ (M\opt)^\top y_{0}\defeqinv v$. Up to singular value decomposition, suppose $M\opt$ rectangular diagonal in $\RR^{d\times n}$ with sorted singular values and write $\tilde \Si =\diag(\sigma_{1},\dots,\sigma_{h})\in\RR^{h\times h}$ with $h \defeq \rk(M\opt)$. Decompose each vector $z$ of $\RR^n$ or $\RR^d$ as $z=(z_H,z_\perp)$, where $z_H\in\RR^h$ and $z_\perp$ contains the remaining coordinates. The equation becomes: 
            \begin{equation*}
            \begin{cases}
            \|y\|^2 x_H+\tilde\Si y_H & \!\!\!\!=v_H \\
            \|y\|^2 x_{\perp} & \!\!\!\!=v_{\perp} .
            \end{cases}
            \end{equation*}
            This imposes $y$ to live in the intersection of a $(d-1)$-dimensional sphere and of a $(d-r)$-dimensional affine subspace of $\RR^d$. As $r=d-1$, the latter subspace is a line and $y$ therefore belongs to a set of at most $2$ points, proving that the $2$-twist condition is satisfied.

            \noindent \textit{Proof of (iii).}
            The case $\rk M\opt\leq d-2$ is a consequence of \cref{theo:fibers-main} and the proof is as follows.
            We consider the measure $\nu$ as a measure on $\RR^n$ with $d$-dimensional support.
            Similarly to the inner product case, using a singular value decomposition, the cost $C_{\pi\opt}$ becomes
            \begin{equation*}
            C_{\pi\opt}(x, y)=-\|x\|^2\|y\|^2-\langle O_\Yy^\top\Si O_\Xx x,y\rangle =-\|O_\Xx x\|^2\|O_\Yy y\|^2-\langle \Si O_\Xx x, O_\Yy y\rangle.
            \end{equation*}
            By \cref{lemma:reparam} the problem reduces to $\min_{\tilde\pi} \langle c_\Si, \tilde\pi\rangle$ for $\tilde\pi \in \Pi(O_{\Xx\pushonly}\mu, O_{\Yy\pushonly}\nu)$, where $c_\Si(x,y)\defeq -\|x\|^2\|y\|^2-\langle \Si x, y\rangle$.
            Up to permutation of the rows and columns of $O_\Xx$ and $O_\Yy$, we further assume $\si_1 \geq \dots \geq \si_h>0$. Writing any $z\in\RR^n$ as $z=(z_H,z_\perp)\in\RR^h\times\RR^{n-h}$,
            \begin{align}
            c_\Si(x,y) & =-\|x_H\|^2 \|y_H\|^2 -\|x_H\|^2 \|y_\perp\|^2 -\|x_\perp\|^2 \|y_H\|^2 -\|x_\perp\|^2 \|y_\perp\|^2 -\langle \tilde \Si x_H, y_H\rangle \nonumber \\
                 &= \tilde c(\phi(x),\phi(y)), 
            \end{align}
            where $\tilde\Sigma$ is the same as in the proof of (ii) above, $\phi: x\mapsto (x_H,\|x_\perp\|^2)$, and $\tilde c:(\RR^h\times\RR)\times (\RR^h\times\RR)\to\RR$ is defined as
            \begin{equation}
                \label{eq:ctilde}
                \tilde c\big((u,r),(v,s)\big)\mapsto -\|u\|^2 \|v\|^2 -s\|u\|^2 -r\|v\|^2 -rs-\langle \tilde \Si u, v\rangle.
            \end{equation}
            Indeed, the cost $c_\Si(x,y)$ depends on the values of $x_H$ and $y_H$, but merely on the squared norm of $x_\perp$ and $y_\perp$. A direct computation proves that $\tilde c$ satisfies the twist condition.\\
            Now, the same argument as in \cref{example:fiber-main} applies, but this time with $E_0 = \RR^h\times\RR^{n-h}$, $E = E_0\backslash (\RR^h\times\{0\})$, $B_0 = \RR^h\times\RR$ and $F = S^{n-h-1} = \{x \in E_0\mid \|x_\perp\|=1\}$. Using \cref{theo:fibers-main}, one obtains the existence of a Monge map $T_0:\RR^n\to\RR^n$ between $\mu$ and $\nu$ for the cost $c_\Sigma$. It has a structure akin to \Cref{eq:structureMongeMap} and decomposes into a map on the basis $B_0 = \RR^{h+1}$ and a map on each fiber $F = S^{n-h-1}$:
                \begin{equation*}
                    \Phi\circ T_0(x)=\Big(\tilde c\text{-}\exp_{(x_H,\,\|x_\perp\|^2)}\big(\nabla f(x_H,\|x_\perp\|^2)\big), \,\exp_{x_\perp/\|x_\perp\|}\big(\nabla g_{(x_H,\|x_\perp\|^2)}(x_\perp/\|x_\perp\|)\big)\Big),
                \end{equation*}
            where $\Phi$ is defined in \cref{eq:bigPhi}.
            Note that the $\tilde c$-exponential map is given in closed form by
            \begin{equation}
                \label{eq:ctildeexp}
                \tilde c\text{-}\exp_{(u,r)}(w,t)=\begin{pmatrix}\tilde \Sigma^{-1}(w-2tu)\\t - \|\tilde\Sigma^{-1}(w-2tu) \|^2 \end{pmatrix}\in\RR^h\times\RR,
            \end{equation}
            where $(u,r)\in\RR^h\times\RR$ is the base point and $(w,t)\in\RR^h\times\RR$ the tangent vector.
            Lastly, note that the case where $M\opt = 0$ has not been explicitly treated. In this case, the cost is simply $c(x,y)=-\|x\|^2\|y\|^2=\tilde c(\phi(x),\phi(y))$ and the strategy above directly applies.\qedhere
        \end{proof}

\subsection{Complementary study of GW with squared distance cost in the one-dimensional case}
\label{subsec:quadra1D}

   The \cref{eqn:GW-squared-distance} problem being invariant by translation, we assume that measures $\mu$ and $\nu$ below are centered. In the one-dimensional case $\Xx,\Yy\subset\RR$, the linearized GW problem \cref{eq:linearized} reads, with $\pi\opt$ an optimal correspondence plan:
\begin{equation}
    \label{eq:gw-1d-cont}
    \min_{\pi\in\Pi(\mu,\nu)} \int_{\Xx\times\Yy} (-x^2y^2-4mxy)\dd\pi(x,y),\quad\text{where }m=\int_{\Xx\times\Yy} x'y'\dd\pi\opt(x',y'),
\end{equation}
and for any plan $\pi\in\Pi(\mu,\nu)$ (not necessarily optimal), we denote by $m(\pi):=\int xy\dd\pi(x,y)$ the \emph{covariance} of $\pi$.
Let us denote the OT cost function associated with \cref{eq:gw-1d-cont} by $c_m(x,y)\defeq-x^2y^2-4mxy$, as it only depends on the real parameter $m$.
Let us apply \cref{theorem:quad-main}:
when $m\neq0$, there exists an optimal transport plan that is a $2$-map and a map/anti-map, and when $m=0$, there exists an optimal transport plan that is a $2$-map.
In the following sections, we study the tightness of this result, asking if there are cases where the optimal plan for this cost can be a map. In the one-dimensional case, the \emph{submodularity} property of the cost function, or Spence--Mirrlees condition, will prove useful. It guarantees the optimality of the \emph{non-decreasing rearrangement} $\pimon$, that match the quantiles of $\mu$ and $\nu$ with each other  \cite{carlier2012optimal,santambrogio2015optimal,peyre2019computational}. We also define the \emph{non-increasing rearrangement} $\piantimon$, that match the quantiles in reverse, and call these two plans the \emph{monotone rearrangements}.
\begin{definition}[Submodular cost]
    \label{def:submod}
    Let $\Xx,\Yy\subset\RR$ be two intervals. A function $c:\Xx\times\Yy\to\RR$ is said to be \emph{submodular} if for all $(x,y) \in \Xx\times\Yy$ and $\delta_1,\delta_2\geq0$ such that $(x + \delta_1,y + \delta_2) \in \Xx\times\Yy$,
\begin{equation*}
 c(x,y) + c(x + \delta_1,y+\delta_2) \leq c(x,y+\delta_2) + c(x + \delta_1,y).
\end{equation*}
If $c$ is twice differentiable, this global condition on the rectangle $\Xx \times \Yy$ is equivalent to the local condition
    \begin{equation}
        \tag{Submod}
        \text{for all } (x,y)\in\Xx\times\Yy,\quad \partial_{xy}c(x,y)\leq0.
        \label{eq:submod}
    \end{equation}
    \emph{Supermodularity} is defined with the reversed inequality.
\end{definition}
\noindent We now state the well-known consequence of submodularity on the structure of optimal plans \cite{carlier2012optimal,mccann2012glimpse,santambrogio2015optimal}.
\begin{proposition}[Optimality of monotone rearrangements]
\label{prop:submod}
Let $\mu,\nu \in\Pp(\Xx)\times\Pp(\Yy)$ of finite transport cost. If $c$ satisfies \cref{eq:submod}, then $\pimon$ is an optimal plan for \cref{eq:OT}, unique if the inequality is strict everywhere on $\Xx \times \Yy$.
Similarly, {supermodularity} induces the optimality of $\piantimon$.
\end{proposition}
The linearized GW cost with parameter $m\geq0$ is submodular on the region $\submod\defeq\{(x,y)\mid xy\geq -m \}$ and supermodular on $\supmod\defeq\{(x,y)\mid xy\leq -m \}$ (see \cref{fig:submod} for an illustration), so we cannot directly apply this proposition. 
    \begin{figure}
        \centering
        \begin{subfigure}[t]{.49\textwidth}
            \centering
            \begin{tikzpicture}[scale=1]
    \def\lim{0.7}
    \def\C{0.1}
    \def\eps{0.1}
    \def\samples{50}
    \def\xmargin{0.2}
    \begin{axis}[
            axis x line=middle,
            axis y line=middle,
            xticklabels=\empty,
            yticklabels=\empty,
            xlabel={$x$},
            ylabel={$y$},
            xmin=-\lim-\xmargin,
            ymin=-\lim,
            xmax=\lim+\xmargin,
            ymax=\lim,
            axis on top,
            axis equal
            ]
        \draw[draw=none, name path=up] (-\lim-\xmargin,\lim) -- (0,\lim);
        \draw[draw=none, name path=upright] (0,\lim+1) -- (\lim+\xmargin,\lim+1);
        \draw[draw=none, name path=down] (0,-\lim) -- (\lim+\xmargin,-\lim);
        \draw[draw=none, name path=downleft] (-\lim-\xmargin,-\lim-1) -- (0,-\lim-1);
        \draw[draw=none, name path=axisleft] (-\lim-\xmargin,0) -- (0,0);
        \draw[draw=none, name path=axisright] (0,0) -- (\lim+\xmargin,0);
        \draw[draw=none, name path=axisleftA] (-\lim-\xmargin,0) -- (0-\eps,0);
        \draw[draw=none, name path=axisleftB] (0-\eps,0) -- (0,0);
        \draw[draw=none, name path=axisrightA] (0+\eps,0) -- (\lim+\xmargin,0);
        \draw[draw=none, name path=axisrightB] (0,0) -- (0+\eps,0);
        \draw[draw=none, name path=upB] (0-\eps,\lim+1) -- (0,\lim+1);
        \draw[draw=none, name path=downB] (0,-\lim-1) -- (0+\eps,-\lim-1);
        \addplot[color=tabblue, domain=-\lim-\xmargin:-\eps, samples=\samples, name path=S1, thick]{-\C/x};
        \addplot[color=tabblue, domain=\eps:\lim+\xmargin, samples=\samples, name path=S2, thick]{-\C/x};
        \node[tabblue] at (0.18,-\lim+0.05) [right] {$x\mapsto -m/x$};
        \addplot[fill=tabblue!10, fill opacity=1] fill between[of=axisleftA and S1];
        \addplot[fill=tabblue!10, fill opacity=1] fill between[of=axisright and upright];
        \addplot[fill=tabblue!10, fill opacity=1] fill between[of=downleft and axisleft];
        \addplot[fill=tabblue!10, fill opacity=1] fill between[of=S2 and axisrightA];
        \addplot[fill=tabblue!10, fill opacity=1] fill between[of=axisleftB and upB];
        \addplot[fill=tabblue!10, fill opacity=1] fill between[of=downB and axisrightB];
        \node[opacity=.7] at (0.5,0.5) {$\nearrow$};
        \node[opacity=.7] at (-0.5,-0.5) {$\nearrow$};
        \def\xa{-0.45}
        \def\xb{0.2}
        \def\ya{-0.2}
        \def\yb{0.1}
        \def\xab{-0.65}
        \def\xbb{-0.22}
        \def\yab{0.3}
        \def\ybb{0.6}
        \draw[tabred, dashed] (\xab,\yab) rectangle (\xbb,\ybb);
        \draw[fill=white] (\xa,\ya) rectangle (\xb,\yb);
        \draw[tabred, fill=tabred!10] (\xa,\ya) rectangle (\xb,\yb);
        \def\size{1.6}
        \fill[tabred] (\xa,\yb) circle[radius=\size pt];
        \fill[tabred] (\xb,\ya) circle[radius=\size pt];
        \fill[tabred, opacity=.6] (\xab,\yab) circle[radius=\size pt];
        \fill[tabred, opacity=.6] (\xbb,\ybb) circle[radius=\size pt];
        \fill[tabred] (\xa,\ya) node[above right] {$R$};
        \fill[tabblue] (\lim+\xmargin-.02,\lim-.02+.05) node[below left] {$\submod$};
        \node[opacity=.8, below left] at (\lim+\xmargin-.02,\ya) {$\supmod$};
        \draw[tabred] (\xa,\yb) node [above]  {\footnotesize $(x_0,y_0)$};
        \draw[tabred] (\xb,\ya) node [below]  {\footnotesize $(x_1,y_1)$};
    \end{axis}
\end{tikzpicture}
            \vspace{-7mm}
        \end{subfigure}
        \hfill
        \begin{subfigure}[t]{.49\textwidth}
            \centering
            \begin{tikzpicture}[scale=1]
    \def\lim{0.7}
    \def\C{-0.1}
    \def\eps{0.1}
    \def\samples{50}
    \def\xmargin{0.2}
    \begin{axis}[
            axis x line=middle,
            axis y line=middle,
            xticklabels=\empty,
            yticklabels=\empty,
            xlabel={$x$},
            ylabel={$y$},
            xmin=-\lim-\xmargin,
            ymin=-\lim,
            xmax=\lim+\xmargin,
            ymax=\lim,
            axis on top,
            axis equal
            ]

        \draw[draw=none, name path=up] (-\lim-\xmargin,\lim) -- (0,\lim);
        \draw[draw=none, name path=upright] (0,\lim+1) -- (\lim+\xmargin,\lim+1);
        \draw[draw=none, name path=down] (0,-\lim) -- (\lim+\xmargin,-\lim);
        \draw[draw=none, name path=downleft] (-\lim-\xmargin,-\lim-1) -- (0,-\lim-1);
        \draw[draw=none, name path=axisleft] (-\lim-\xmargin,0) -- (0,0);
        \draw[draw=none, name path=axisright] (0,0) -- (\lim+\xmargin,0);
        \draw[draw=none, name path=axisleftA] (-\lim-\xmargin,0) -- (0-\eps,0);
        \draw[draw=none, name path=axisleftB] (0-\eps,0) -- (0,0);
        \draw[draw=none, name path=axisrightA] (0+\eps,0) -- (\lim+\xmargin,0);
        \draw[draw=none, name path=axisrightB] (0,0) -- (0+\eps,0);
        \draw[draw=none, name path=upB] (0-\eps,\lim+1) -- (0,\lim+1);
        \draw[draw=none, name path=downB] (0,-\lim-1) -- (0+\eps,-\lim-1);
        \addplot[color=tabblue, domain=-\lim-\xmargin:-\eps, samples=\samples, name path=S1, thick]{-\C/x};
        \addplot[color=tabblue, domain=\eps:\lim+\xmargin, samples=\samples, name path=S2, thick]{-\C/x};
        \addplot[fill=tabblue!10, fill opacity=1] fill between[of=downleft and S1];
        \addplot[fill=tabblue!10, fill opacity=1] fill between[of=S2 and upright];
        \node[opacity=.7] at (-0.5,0.5) {$\searrow$};
        \node[opacity=.7] at (0.5,-0.5) {$\searrow$};
        \fill[tabblue] (\lim+\xmargin-.02,\lim-.02+.05) node[below left] {$\submod$};
        \def\ya{-0.2}
        \node[opacity=.8, below left] at (\lim+\xmargin-.02,\ya) {$\supmod$};
    \end{axis}
\end{tikzpicture}
            \vspace{-7mm}
        \end{subfigure}
        \caption{Submodularity region $\submod$ (\textcolor{tabblue}{blue}) and supermodularity region $\supmod$ for the linearized GW cost $c_m$ with parameter $m>0$ \capleft\ or $m<0$ \capright, with an example of a rectangle $R\subset S$ (\textcolor{tabred}{red}) defined by a decreasing matching $(x_0,y_0)\to(x_1,y_1)$ that allows to obtain the monotonicity of optimal plans on $\submod$ (see proof of \cref{prop:monotony-rectangle}). The same argument does not apply to rectangles in $\supmod$ (\textcolor{tabred}{dashed red}), see \cref{rem:Sbar}. Up ({\scriptsize$\nearrow$}) and down ({\scriptsize$\searrow$}) arrows are drawn in regions where the monotonicity of optimal correspondence plans is given by \cref{prop:monotony-rectangle}.}
        \label{fig:submod}
    \end{figure}
In general, for a given cost in dimension one, there are regions of sub and supermodularity and this property gives only few information without knowledge of any particular structure of the regions: to be applicable, submodularity requires a rectangular region.
Given the form of the regions of sub and supermodularity in our particular case, we can state the following.
\begin{proposition}[Monotonicity of optimal plans on a subdomain]
\label{prop:monotony-rectangle}
Let $m> 0$ and denote by $\pi\opt$ an optimal transportation plan for the cost $c_m$.
Then, $[\pi\opt]_{|\submod}$ (the plan restricted to the submodularity region) is the non-decreasing rearrangement between its marginals.
Similarly, if $m < 0$, $[\pi\opt]_{|\supmod}$ (the plan restricted to the supermodularity region) is the non-increasing rearrangement. If $m = 0$, there are four regions of sub/supermodularity, and on each of these regions, the plan is a monotone rearrangement.
\end{proposition}
\noindent \cref{fig:submod} illustrates the regions, with up ({\scriptsize$\nearrow$}) and down ({\scriptsize$\searrow$}) arrows,  where the monotonicity of optimal correspondence plans is given by this proposition.
\begin{proof}
In this proof, we use the fact that if $\pi\opt$ is optimal, then it is also optimal when restricted on a domain between the corresponding marginals.
In particular, when $m > 0$, the plan $[\pi\opt]_{|\submod}$ is necessarily optimal for the cost $c_m$ between its marginals. Consider $(x_0,y_0)$ and $(x_1,y_1)$ in $\supp(\smash{[\pi\opt]_{|\submod}})$ such that $x_0 < x_1$ and $y_0 > y_1$. The particular geometry of the space $\submod$ ensures that in this particular configuration, the rectangle $R = [x_0,x_1]\times[y_1, y_0]$ is necessarily contained in $\submod$ (see \cref{fig:submod}). As a consequence, the strict submodularity of the cost applies on $R$ to prove that $[\pi\opt]_{|R}$ is the non-decreasing rearrangement plan, contradicting the configuration.
The proof goes similarly when $m < 0$. 
The case $m = 0$ yields four different regions of sub/supermodularity which are quadrants. Since these regions are unions of rectangles on which the cost is strictly sub/supermodular, the argument above applies.
\end{proof}

\begin{remark}[Rectangular regions]
\label{rem:Sbar}
Let us underline that it is not possible to apply the argument to the whole region $\submod$ since any two points of $\submod$ are not necessarily contained in a rectangle included in $\submod$. 
The proof above only uses the fact that this property is used for some particular configuration of points, namely whenever $x_0 < x_1$ and $y_0 > y_1$, which is (somewhat fortunately) exactly the configuration we want to consider when wondering if a map may be non-increasing. 
For instance, this argument does not apply to $\supmod$, although the property is of course satisfied on every rectangle contained in it.
\end{remark}

Before going into further details on our complementary study, we recall the discrete formulation of \cref{eq:OT} in dimension one.
Given two sets $\{ x_{1},\dots,x_{N} \}$ and $\{ y_{1},\dots,y_{M} \}$ of elements of $\RR$ and two probability vectors $a$ and $b$, the \cref{eq:OT} problem between the discrete measures $\mu=\sum_{i=1}^{N}a_{i}\delta_{x_{i}}$ and $\nu=\sum_{j=1}^{M}b_{j}\delta_{y_{j}}$ reads
        \begin{equation*}
            \min_{\pi\in U(a,b)}\ \langle C,\pi\rangle,
        \end{equation*}
where $U(a,b)\defeq \{ \pi\in\RR^{N\times M}\mid \pi\one_M=a,\pi^\top\one_N=b\}$ is the \emph{transport polytope}, $C=(c(x_i,y_j))_{i,j}$ is the cost matrix and $\langle \cdot,\cdot\rangle$ is the Frobenius inner product.
In the case of the linearized problem \cref{eq:gw-1d-cont}, we denote by $\smash{C_{\gw(m)}}$ the cost matrix, that has entries $\smash{(C_{\gw(m)})_{i,j}=-x_i^2y_j^2-4m x_iy_j}$ with $m=\langle C_{xy}, \pi\opt\rangle$ and $(C_{xy})_{i,j}=x_iy_j$. In our notation, the index of any cost matrix $C$ is its associated OT cost function.

In the following sections, we study the optimality of the monotone rearrangements $\pimon$ and $\piantimon$. It is worth noting that by submodularity of the map $(x,y)\mapsto -xy$, these two plans have respective covariances $m_\text{min}$ and $m_\text{max}$, where
        \begin{align}
            \label{eq:m-min-max}
            \begin{cases}
                m_\text{min}&\!\!\!\!\!= \min_{\pi}\,\langle C_{xy}, \pi\rangle\\
                m_\text{max}&\!\!\!\!\!= \max_{\pi}\,\langle C_{xy}, \pi\rangle
            \end{cases},\quad\text{with }(C_{xy})_{i,j}=x_iy_j,
        \end{align}
and that for any correspondence plan $\pi$, the value of its covariance $m(\pi)$ lies in the interval $[m_\text{min},m_\text{max}]$.

\noindent We provide in the following a complementary study of the squared distance cost in dimension one, namely
\begin{enumerate}[(i),nolistsep]
    \item a procedure to find counter-examples to the optimality of the monotone rearrangements;
    \item empirical evidence for the tightness of \cref{theorem:quad-main};
    \item a proof of the instability of having a monotone rearrangement as an optimal correspondence plan;
    \item a new result on the optimality of the monotone rearrangements when the measures are composed of two distant parts.
\end{enumerate}
All experiments are reproducible and the code can be found on GitHub\footnote{The link of the code is \url{https://github.com/theodumont/monge-gromov-wasserstein}.}.
\subsubsection{Adversarial computation of non-monotone optimal correspondence plans}
\label{subsec:quadra1D_adversarial}
Theorem~3.1 of \cite{titouan2019sliced} claims that in the one-dimensional discrete case with $N=M$ and $a=b=\one_N$, the optimal solution of \cref{eq:QAP} is either the non-decreasing rearrangement $\pimon$ or the non-increasing one $\piantimon$. 
It seems to be the case with high probability empirically when generating random discrete measures.
While this claim is true for $N=1,2$ and $3$,
a counter-example for $N\geq 7$ points has recently been exhibited in \cite{beinert2022assignment}. We further propose a procedure to automatically obtain additional counter-examples, demonstrating empirically that such adversarial measures occupy a non-negligible place in the space of discrete measures.
We propose to perform a gradient descent over the space of discrete measures on $\Xx\times\Yy$ using an objective function that favors the strict sub-optimality of the monotone rearrangements. Let us now detail this procedure.

For $N\geq 1$, we consider the set of discrete measures over $\Xx\times \Yy=\RR\times\RR$ with $N$ points and uniform mass, i.e.~of the form $\smash{\pi=\frac{1}{N}\sum_{i=1}^{N}\delta_{(x_{i},y_{i})}}$. Such plans $\pi$ can be seen as the identity mapping between vectors $X=(x_1,\dots,x_N)$ and $Y=(y_1,\dots,y_N)$, and we therefore note $\pi=\id(X,Y)$. Denoting by $c_{\gw}$ the functional that takes a correspondence plan and returns its cost on the GW problem, we then define $\Ff:\RR^N\times\RR^N\to\RR$ by
\begin{equation*}
\Ff(X,Y)\defeq c_{\gw}(\pi)-\min \left\{ c_{\gw}(\pimon), c_{\gw}(\piantimon)\right\},
\end{equation*}
where $\pi=\id(X,Y)$ and $\pimon$ and $\piantimon$ are the monotone rearrangements between $X$ and $Y$.
This quantifies how well the plan $\pi$ compares to the two monotone rearrangements. We generate $N$ points at random in $[0,1]^2$ and then perform a simple gradient descent over the positions of the points $(X,Y)=(x_{i},y_{i})_i$ to minimize the objective function $\Ff(X,Y)$ over $(X,Y) \in (\RR^N)^2$.
We include an early-stopping threshold $-\tau<0$ since when $\Ff(\pi)$ becomes negative---i.e.~we found an adversarial example---, the objective function often decreases exponentially fast.
The procedure can be found in \cref{algorithm:gd} below. We implemented it using PyTorch's automatic differentiation \cite{pytorch} and used \cite{blondel2020fast} to implement a differentiable sorting operator \texttt{sort} to compute the monotone rearrangements.
Adversarial plans $\pi\adv=\id(X\adv,Y\adv)$ obtained by \cref{algorithm:gd} are not \textit{a priori} optimal for the GW cost between their marginals; but they have at least a better cost than the monotone rearrangements since $\Ff(X\adv,Y\adv)< 0$, proving the sub-optimality of the latter.

\begin{figure}[h]
\centering
\begin{minipage}{.8\linewidth}
    \begin{algorithm}[H]
    \flushleft
    \caption{Simple gradient descent over the positions $(x_i)_i$ and $(y_i)_i$.}
    \label{algorithm:gd}
    \vspace{1mm}
    \textbf{Parameters:}
    \begin{itemize}[nolistsep,wide=5pt]
        \item $N\geq1$: number of points of the measures
        \item $N_\text{iter}\geq1$: maximum number of iterations
        \item $\eta>0$: step size
        \item $\tau>0$: early stopping threshold
    \end{itemize}
    \vspace{2mm}
    \textbf{Algorithm:}
    \begin{algorithmic}[1]
        \State $X\gets$ $N$ random values in $[0,1]$, then centered
        \State $Y\gets$ $N$ random values in $[0,1]$, then centered
        \For{$i\in\{1,\dots,N_\text{iter}\}$}
        \State $\pimon\gets\texttt{id(sort(}X\texttt{)},\texttt{sort(}Y\texttt{))}$ \Comment{\texttt{id} is the identity mapping}
        \State $\piantimon\gets\texttt{id(sort(}X\texttt{)},\texttt{reverse(sort(}Y\texttt{)))}$
        \State $\pi_{\phantom{\text{mon}}}\gets\texttt{id(}X,Y\texttt{)}$
        \State $\Ff(X,Y)\gets \texttt{GW(}\pi\texttt{)}-\texttt{min(GW(}\pimon\texttt{)},\texttt{GW(}\piantimon\texttt{))}$ \Comment{\texttt{GW} computes $c_{\gw}$}
        \State \textbf{if} $\Ff(X,Y)\leq -\tau$ \textbf{then} stop \Comment{early stopping}
        \State $(X,Y)\gets (X,Y)-\eta\nabla\Ff(X,Y)$  \Comment{step of gradient descent}
        \EndFor
        \State return $\pi\adv=\texttt{id(}X,Y\texttt{)}$
        \end{algorithmic}
        \vspace{2mm}
    \textbf{Output:} a plan $\pi\adv$ with better GW cost than $\pimon$ and $\piantimon$
  \end{algorithm}
\end{minipage}
\end{figure}

On \Cref{fig:res-GD} is displayed an example of adversarial plans obtained following this procedure. It can be observed that during the descent, the plan $\pi$ has difficulties getting out of what seems to be a saddle point consisting of being (close to) a monotone rearrangement between its marginals. Moreover, it is worth noting that the marginals of our typical adversarial plans, such as the one of \cref{fig:res-GD}, are often similar to the counter-example proposed in \cite{beinert2022assignment}, where both measures have their mass concentrated near zero, except for one outlier for $\mu$, and two for $\nu$, one on each tail.

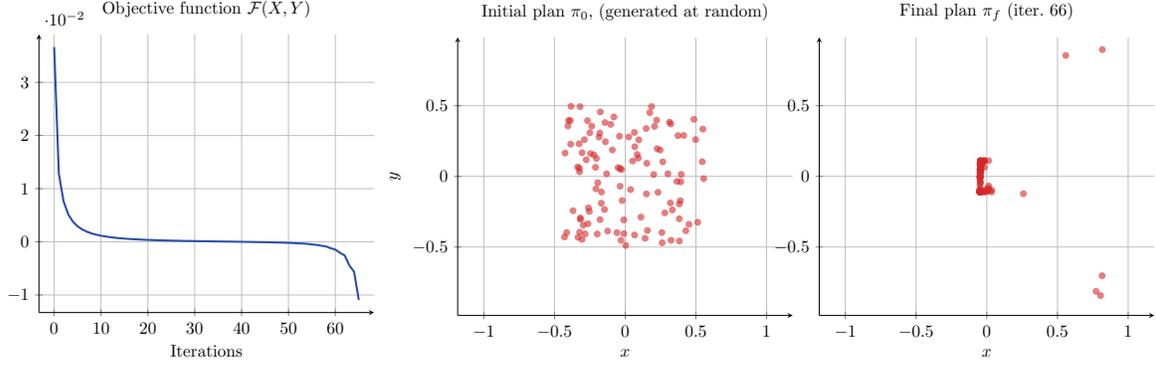
\begin{figure}[h]
    \centering
    \begin{subfigure}[b]{.32\linewidth}
        \centering
\begin{tikzpicture}[scale=.65]

    \begin{axis}[
    title={Objective function $\Ff(X,Y)$},
    xlabel={Iterations},
    xmin=-3.25, xmax=68.25,
    ymin=-0.0133060135878623, ymax=0.0390243952162564,
    grid = major,
    axis lines=left,
    ]
    \addplot[very thick, no marks, tabblue]
    table {%
    0 0.0366457402706146
    1 0.0128296026960015
    2 0.00764693738892674
    3 0.00516638578847051
    4 0.0037919997703284
    5 0.00292218313552439
    6 0.0023217792622745
    7 0.00190227513667196
    8 0.00157574948389083
    9 0.00134416075889021
    10 0.00114074978046119
    11 0.000998637056909502
    12 0.000866047223098576
    13 0.00076767144491896
    14 0.000676982395816594
    15 0.000603142834734172
    16 0.000542129680979997
    17 0.000482635397929698
    18 0.000441085547208786
    19 0.000394769042031839
    20 0.000358618883183226
    21 0.000328242749674246
    22 0.000294765632133931
    23 0.000271784490905702
    24 0.000246440700720996
    25 0.000222068309085444
    26 0.000205642892979085
    27 0.000186464632861316
    28 0.000164534663781524
    29 0.000155049579916522
    30 0.000136664952151477
    31 0.000118492491310462
    32 0.000111368106445298
    33 9.13372787181288e-05
    34 7.99107947386801e-05
    35 7.14088673703372e-05
    36 5.07369404658675e-05
    37 4.42669843323529e-05
    38 3.13138007186353e-05
    39 1.63891818374395e-05
    40 5.7496945373714e-06
    41 -1.32657296489924e-05
    42 -1.94765452761203e-05
    43 -4.64202603325248e-05
    44 -5.04592899233103e-05
    45 -7.56503432057798e-05
    46 -8.77406564541161e-05
    47 -0.000117279996629804
    48 -0.000136992719490081
    49 -0.000163182790856808
    50 -0.000196536770090461
    51 -0.000234442006330937
    52 -0.000290850817691535
    53 -0.000317854166496545
    54 -0.000384356826543808
    55 -0.00048568716738373
    56 -0.000578868668526411
    57 -0.000735051580704749
    58 -0.000863353256136179
    59 -0.0012065467890352
    60 -0.00146772898733616
    61 -0.00207047536969185
    62 -0.00256216758862138
    63 -0.00445789750665426
    64 -0.00563399493694305
    65 -0.0109273586422205
    };
    \end{axis}

    \end{tikzpicture}
        \vspace{-9mm}
    \end{subfigure}
    \begin{subfigure}[b]{.32\linewidth}
        \centering
\begin{tikzpicture}[scale=.65]

    \begin{axis}[
    title={Initial plan $\pi_0$, (generated at random)},
    xlabel={$x$},
    ylabel={$y$},
    xmin=-0.860483883693814, xmax=0.860483883693814,
    ymin=-0.98424434363842, ymax=0.98424434363842,
    grid = major,
    axis lines=left,
    axis equal,
    ]
    \addplot [draw=tabred, fill=tabred, mark=*, only marks,opacity=.6, draw opacity=0]
    table{%
    x  y
    0.555029630661011 -0.0158676505088806
    -0.236768633127213 0.355503797531128
    0.544643044471741 0.104050576686859
    -0.194337338209152 -0.0446888208389282
    -0.283418506383896 -0.407166182994843
    -0.326179891824722 0.230781435966492
    -0.253536850214005 -0.248354732990265
    -0.139955550432205 0.245159804821014
    -0.176041215658188 0.456286549568176
    0.549169182777405 0.334675550460815
    -0.333650857210159 -0.431270480155945
    0.207398444414139 0.353976368904114
    0.156865984201431 -0.382931172847748
    -0.250565379858017 0.308950901031494
    0.261155754327774 -0.470198333263397
    -0.210368067026138 0.0642983317375183
    -0.169332355260849 -0.189902603626251
    0.392874985933304 -0.0384848713874817
    -0.289017468690872 0.260072827339172
    0.382877200841904 -0.194872617721558
    0.322315603494644 0.371125638484955
    0.05231574177742 0.108898401260376
    0.33347424864769 -0.237155020236969
    -0.31975182890892 0.493428528308868
    -0.0101818144321442 -0.30777645111084
    -0.221232026815414 0.152289867401123
    0.147978037595749 0.338643074035645
    -0.317791134119034 -0.303987145423889
    -0.400682240724564 0.395698249340057
    -0.146978765726089 -0.23452764749527
    0.512488603591919 -0.324599802494049
    -0.30556783080101 -0.445916712284088
    -0.0358024537563324 -0.0697241425514221
    -0.0912812650203705 0.187010586261749
    -0.323942273855209 0.0577062368392944
    0.280057102441788 -0.258328318595886
    -0.168274790048599 -0.110755920410156
    0.382387191057205 -0.300583899021149
    -0.29683730006218 -0.345050394535065
    0.149961620569229 -0.124405741691589
    -0.0355304777622223 0.0557541251182556
    -0.0407645404338837 0.284307897090912
    0.0247271358966827 0.27869176864624
    -0.0798503458499908 0.420502126216888
    0.00401011109352112 -0.488527834415436
    -0.413885325193405 -0.397468686103821
    0.395864278078079 0.0138669610023499
    0.0377726256847382 -0.093356728553772
    0.222262471914291 0.398219287395477
    -0.027001291513443 0.0481876134872437
    0.173529058694839 0.45062130689621
    0.095820277929306 0.258765935897827
    -0.185678154230118 0.278037309646606
    0.152429193258286 0.0925498604774475
    0.498135596513748 0.260421872138977
    -0.429192751646042 -0.429079830646515
    -0.0217974483966827 -0.17097008228302
    0.312465757131577 0.383726358413696
    0.0647160112857819 0.311473965644836
    0.247042804956436 0.186284482479095
    -0.275886982679367 0.11732143163681
    -0.0585298240184784 -0.398380994796753
    -0.131526082754135 0.0180906653404236
    -0.382843405008316 0.495713412761688
    0.428420037031174 -0.38513195514679
    0.318877905607224 -0.187862992286682
    -0.0301567018032074 -0.4527707695961
    0.485664933919907 0.403507351875305
    0.373182088136673 0.288419723510742
    -0.198172718286514 -0.407697439193726
    0.0928529798984528 0.129496157169342
    -0.263608783483505 -0.335467755794525
    -0.3370221555233 0.065843939781189
    0.0660794675350189 -0.414897620677948
    -0.204000979661942 0.12734979391098
    0.450773566961288 -0.339365720748901
    0.231547206640244 -0.112659335136414
    0.323924750089645 -0.452231049537659
    0.11105141043663 -0.288233816623688
    -0.207259505987167 -0.088228166103363
    0.0673986375331879 0.209961175918579
    0.383201092481613 -0.45745986700058
    0.259228736162186 -0.397888660430908
    -0.405101984739304 0.356888294219971
    -0.369065493345261 -0.243187546730042
    -0.124457567930222 -0.386952579021454
    -0.102398365736008 0.367921233177185
    -0.30541530251503 0.167475700378418
    -0.425343066453934 0.165306985378265
    0.201642423868179 0.0133430957794189
    -0.0095919668674469 -0.405273020267487
    0.185787290334702 0.493967235088348
    -0.386108547449112 0.228136360645294
    -0.322127968072891 -0.394426882266998
    -0.143171042203903 0.381069421768188
    0.413891106843948 0.289939701557159
    0.225552827119827 0.196329057216644
    -0.245900183916092 0.161448657512665
    0.0844146907329559 0.154719233512878
    -0.179586321115494 0.306278169155121
    0.301403194665909 0.0181658864021301
    -0.177426904439926 -0.306294500827789
    -0.38782075047493 0.394672214984894
    -0.0506928861141205 0.0619080066680908
    0.140196830034256 -0.43829607963562
    0.365351229906082 -0.035691499710083
    0.38993713259697 -0.171832323074341
    0.265565484762192 0.102889955043793
    -0.260365158319473 -0.223024189472198
    -0.323181599378586 0.0323825478553772
    -0.317005842924118 -0.293777823448181
    -0.266747921705246 0.394475519657135
    };
    \end{axis}

    \end{tikzpicture}
        \vspace{-9mm}
    \end{subfigure}
    \begin{subfigure}[b]{.32\linewidth}
        \centering
\begin{tikzpicture}[scale=.65]

    \begin{axis}[
    title={Final plan $\pi_f$ (iter.\@ 66)},
    xlabel={$x$},
    xmin=-0.860483883693814, xmax=0.860483883693814,
    ymin=-0.98424434363842, ymax=0.98424434363842,
    grid = major,
    axis lines=left,
    axis equal,
    ]
    \addplot [draw=tabred, fill=tabred, mark=*, only marks,opacity=.6, draw opacity=0]
    table{%
    x  y
    -0.0475812740623951 -0.00340059632435441
    -0.0441303513944149 0.0977057367563248
    -0.0442190244793892 0.0718558132648468
    -0.0466805137693882 0.00833576358854771
    -0.0477646812796593 -0.108987741172314
    -0.0413817279040813 0.109386973083019
    -0.0483411774039268 -0.109251156449318
    -0.048307828605175 0.0838327631354332
    -0.0371303409337997 0.0613369047641754
    0.817175030708313 0.897216498851776
    -0.0482831038534641 -0.108428820967674
    -0.0447965525090694 0.110614284873009
    -0.036015722900629 -0.108963027596474
    -0.0438485108315945 0.0945742800831795
    0.0132665485143661 -0.0915598645806313
    -0.0454493500292301 0.0470575354993343
    -0.0483983010053635 -0.0836576446890831
    -0.0449897795915604 -0.00183166423812509
    -0.0481329299509525 0.0934909284114838
    0.00847463868558407 -0.0838896632194519
    -0.0201860945671797 0.110789969563484
    -0.0454405024647713 0.0466846711933613
    0.00794699043035507 -0.10198587924242
    -0.0438002422451973 0.0828163623809814
    -0.0487846024334431 -0.109610706567764
    -0.0438351295888424 0.0812704414129257
    -0.0448389202356339 0.0490087643265724
    -0.047769483178854 -0.109528876841068
    -0.0364673547446728 0.0948797687888145
    -0.0489671863615513 -0.109020784497261
    0.804644048213959 -0.843340396881104
    -0.0490020290017128 -0.109648950397968
    -0.0476549714803696 0.0018444717861712
    -0.0438759811222553 0.0812242925167084
    -0.0450367592275143 0.047038622200489
    -0.00652758218348026 -0.10457718372345
    -0.0476841442286968 -0.0261788815259933
    0.0341730751097202 -0.0964658558368683
    -0.0482552126049995 -0.105813428759575
    -0.0450278036296368 -0.0431392602622509
    -0.0453220792114735 0.0466888211667538
    -0.0454730838537216 0.10557959228754
    -0.0446379743516445 0.0966566875576973
    -0.0439731553196907 0.109578840434551
    -0.0489734932780266 -0.10903537273407
    -0.046284195035696 -0.109579287469387
    -0.0462511293590069 0.0215715114027262
    -0.0470862984657288 -0.0139435715973377
    -0.0438619293272495 0.10637903958559
    -0.0459479764103889 0.0467469207942486
    -0.0447905547916889 0.0636114850640297
    -0.0379205122590065 0.0914457887411118
    -0.0147720724344254 0.110659137368202
    -0.0452795699238777 0.0466407015919685
    0.0134735815227032 0.110991790890694
    -0.0483037307858467 -0.100258551537991
    -0.0478320382535458 -0.0571895390748978
    -0.0187774132937193 0.110866993665695
    -0.0449938513338566 0.0993728265166283
    -0.0436838045716286 0.0743295475840569
    -0.0154984556138515 0.0642513930797577
    -0.0482148230075836 -0.0983421355485916
    -0.046513669192791 0.0302988514304161
    -0.0447973273694515 0.101852312684059
    0.772858798503876 -0.81375914812088
    -0.0150513425469398 -0.0950996652245522
    -0.0488341301679611 -0.109095200896263
    0.558734118938446 0.855513453483582
    -0.0306842885911465 0.102204032242298
    -0.048664890229702 -0.108449772000313
    -0.0476398505270481 0.0466734692454338
    -0.0487281531095505 -0.102116376161575
    -0.0462375953793526 0.0470677465200424
    -0.0489896237850189 -0.109523631632328
    -0.0437434874475002 0.0641783475875854
    0.815857708454132 -0.703193187713623
    -0.0450068674981594 -0.0431612208485603
    0.0342482104897499 -0.110243208706379
    -0.0157681405544281 -0.108991578221321
    -0.0476663447916508 -0.0103251747786999
    -0.0437650494277477 0.0711103230714798
    0.259046137332916 -0.122416362166405
    0.0134525373578072 -0.0671508312225342
    -0.0443049035966396 0.110298775136471
    -0.0489722192287445 -0.109457820653915
    -0.0462229363620281 -0.108776450157166
    -0.0439266040921211 0.103450559079647
    -0.0452461540699005 0.0952032804489136
    -0.0445888079702854 0.109978541731834
    -0.0464168898761272 0.0215739142149687
    -0.0488363690674305 -0.109651520848274
    -0.0439438261091709 0.107009597122669
    -0.0251717530190945 0.102297611534595
    -0.020452544093132 -0.109555952250957
    -0.0438580699265003 0.106445752084255
    -0.0192825943231583 0.110856600105762
    -0.0436716377735138 0.074300192296505
    -0.035883292555809 0.0879584029316902
    -0.0453389137983322 0.0485817119479179
    -0.0447676591575146 0.109777428209782
    -0.0464421845972538 0.0215416569262743
    -0.0471973158419132 -0.109632275998592
    -0.0438513122498989 0.106677636504173
    -0.0439114719629288 0.0467464290559292
    -0.0359955914318562 -0.109192289412022
    -0.0482445061206818 -0.00181822525337338
    0.0022273832000792 -0.107356525957584
    -0.0451836325228214 0.0466696694493294
    -0.0480019599199295 -0.109633311629295
    -0.0447049923241138 0.0470405519008636
    -0.0483166500926018 -0.109632670879364
    -0.0442694015800953 0.0842185094952583
    };
    \end{axis}

    \end{tikzpicture}
        \vspace{-9mm}
    \end{subfigure}
    \caption{Gradient descent results with parameters $N=122$, $\eta=26$, $\tau=2$.}
    \label{fig:res-GD}
\end{figure}

Furthermore, examining the optimal correspondence plan for these adversarial examples allows to exhibit cases where it is not a map, providing empirical evidence for the following conjecture.
        \begin{conjecture}[Tightness of \cref{theorem:quad-main}]
            \label{conj:tight}
            There exist measures $\mu$ and $\nu$ for which no optimal correspondence plan for \cref{eqn:GW-squared-distance} is a map, but always a union of two graphs (that of two maps or one map and one anti-map). This result holds even if $\mu$ has a density, classical assumption for the existence of an optimal transport map.
        \end{conjecture}
In order to approximate numerically the case of a measure that has density with respect to the Lebesgue measure, we convolve our adversarial measures $\mu\adv=(X\adv,\one_N)$ and $\nu\adv=(Y\adv,\one_N)$ with a Gaussian measure of standard deviation $\sigma$ and represent them in Eulerian coordinates $a\adv$ and $b\adv$; that is we evaluate the closed form densities on a fine enough regular grid.
When $\sigma$ is large, the optimal correspondence plan is probably induced by a monotone rearrangement, as it is the case very frequently empirically; on the contrary, if $\sigma$ is sufficiently small, i.e.~when the measures are very close to their sum of Dirac measures discrete analogous, the optimal correspondence plan should not be a monotone rearrangement, by construction of $\mu\adv$ and $\nu\adv$.

        \begin{remark}[Counter-examples from \cref{algorithm:gd} sometimes agree with \cref{conj:tight}]
            Because of the adversarial nature of $\pi\adv$ for the sub-optimality of $\pimon$ and $\piantimon$, we know that when $\sigma$ is sufficiently small, the optimal correspondence plan is not a monotone rearrangement. Still, it could be the case that this optimal plan is a map, but not a monotone one, and there is \textit{a priori} no reason to believe that $\pi\adv$ will agree with \cref{conj:tight}. Surprisingly, it sometimes does, as the numerical experiments below suggest.
        \end{remark}

        In order to find the optimal correspondence plan $\pi\opt$ between $\mu\adv$ and $\nu\adv$, we leverage the fact that $\pi\opt$ is a solution of its associated linearized problem. Therefore, a minimizer of the GW functional is given by
        \begin{equation}
            \argmin_{\pi\opt_m}\ \Big\{ \gw(\pi_{m}\opt)\ \big\vert\ \pi_m\opt \in \argmin_{\pi \in U(a\adv,b\adv)}\langle C_{\gw(m)},\pi\rangle ,\, m \in [m_\text{min},m_\text{max}]\Big\}  ,
        \end{equation}
        where $\smash{(C_{\gw(m)})_{i,j}=-x_i^2y_j^2-4m x_iy_j}$.
        We therefore compute both $m_\text{min}$ and $m_\text{max}$ by solving the linear programs in \cref{eq:m-min-max}, discretize the interval $[m_\text{min},m_\text{max}]$ with $N_{m}$ points, and solve the corresponding linear optimization problem for every value of the parameter $m$ and evaluate the GW cost on each optimal plan for the given parameter $m$. We then check if the optimal plan exhibits a $2$-map structure.
        The procedure is described in \cref{algorithm:2-map}.

    \begin{figure}[h]
        \centering
        \begin{minipage}{.8\linewidth}
            \begin{algorithm}[H]
            \flushleft
                \caption{Generating $2$-maps from adversarial examples.}
                \label{algorithm:2-map}
            \vspace{1mm}
            \textbf{Input:} an adversarial plan $\pi\adv=\id(X\adv,Y\adv)$ obtained from \cref{algorithm:gd}

            \vspace{2mm}
            \noindent\textbf{Parameters:}
            \begin{itemize}[nolistsep,wide=5pt]
                \item $\sigma$: standard deviation of convolution
                \item $N_{x}$: discretization precision
                \item $N_{m}$: discretization precision of the interval $[m_\text{min},m_\text{max}]$
            \end{itemize}
            \vspace{2mm}
            \textbf{Algorithm:}
            \begin{algorithmic}[1]
                \State $a\adv\gets \texttt{convolution}(X\adv,\sigma,N_{x})$
                \State $b\adv\gets \texttt{convolution}(Y\adv,\sigma,N_{x})$ \Comment{(optional for $\nu\adv$)}
                \State $m_{\vphantom{\text{max}}\text{min}}\gets \min_{\pi\in U(a,b)}\ \langle C_{xy},\pi\rangle$ \Comment{solve linear programs}
                \State $m_{\vphantom{\text{min}}\text{max}}\gets \max_{\pi\in U(a,b)}\ \langle C_{xy},\pi\rangle$
                \State \texttt{scores} $\gets \texttt{[]}$
                \For{$m\in\{m_\text{min},\dots,m_\text{max}\}$} \Comment{with $N_{m}$ points}
                    \State $\pi\opt_{m}\gets \argmin_{\pi\in U(a,b)}\ \langle C_{\gw(m)},\pi\rangle$ \Comment{solve linear program}
                    \State append $\gw(\pi\opt_{m})$ to \texttt{scores}
                \EndFor
                \State $\pi\opt\gets\argmax_\pi$ \texttt{scores} \Comment{take best plan for GW}
                \State $b\gets$ ``$\pi\opt$ is a $2$-map''
                \State return $\pi\opt$, $b$
                \end{algorithmic}
                \vspace{2mm}
            \textbf{Outputs:}
            \begin{itemize}[nolistsep,wide=5pt]
                \item $\pi\opt$: optimal plan for GW
                \item $b$: boolean asserting if $\pi\opt$ is a $2$-map
            \end{itemize}
          \end{algorithm}
        \end{minipage}
      \end{figure}

        We display the results on \cref{fig:res-2-map}, where we plot the optimal correspondence plan $\pi\opt$
        in two cases, starting from an adversarial plan:
        \begin{enumerate}[(i),nolistsep,leftmargin=*]
            \item with both marginals $\mu\adv$ and $\nu\adv$ convolved as to simulate densities;
            \item with only the first marginal $\mu\adv$ convolved and the second marginal $\nu\adv$ being a discrete measure with uniform mass.
        \end{enumerate}
        To facilitate the reading of \cref{fig:res-2-map}, we draw a blue pixel at locations $x$ on the discretized $x$-axis if $x$ has two images (making $\pi\opt$ a $2$-map), and at locations $y$ on the $y$-axis if $y$ has two pre-images (making $\pi\opt$ a anti-$2$-map).
        In both cases, we observe that $\pi\opt$ is not a map but a $2$-map instead, similarly to \cite[Section 4.5]{chiappori2010hedonic}. Note that in case (ii), $\nu$ being atomic, there cannot exist a map from $\nu$ to $\mu$, so in both (i) and (ii) we numerically exhibit an instance where there is \textit{a priori} no map from neither $\mu$ to $\nu$ nor $\nu$ to $\mu$.
        We also plot the submodularity regions of the linearized GW cost function with parameter $m(\pi\opt)$ as an overlay and we observe that when the optimal plan gives mass to a region where the cost is submodular or supermodular, it has respectively a non-decreasing or non-increasing behavior in this region.
        \begin{figure}[h]
            \centering
            \begin{subfigure}[b]{.49\linewidth}
                \centering
                \includegraphics[width=\textwidth]{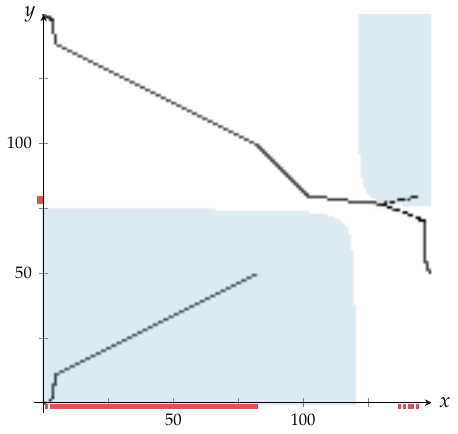}
                \vspace{-6mm}
                \caption*{(i) Convolving both marginals.}
                \label{fig:res-2-map-a}
            \end{subfigure}
            \hfill
            \begin{subfigure}[b]{.49\linewidth}
                \centering
                \includegraphics[width=\textwidth]{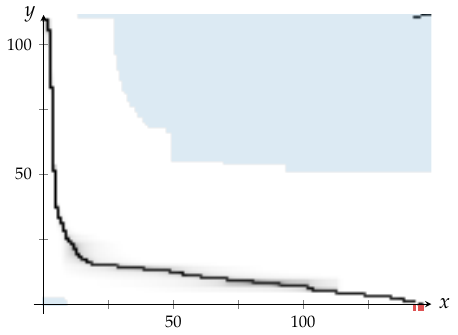}
                \vspace{-6mm}
                \caption*{(ii) Only convolving the first marginal.}
                \label{fig:res-2-map-b}
            \end{subfigure}
                \caption{Optimal correspondence plan (in log scale) obtained with our procedure, with $2$-map and anti-$2$-map coordinates (\textcolor{tabred}{red}) and submodularity regions (\textcolor{tabblue}{blue}). Parameters: $\sigma=\smash{5.10^{-3}}$, $N_{x}=150$, $N_{m}=2000$.
                }
            \label{fig:res-2-map}
        \end{figure}

\subsubsection{Empirical instability of the optimality of monotone rearrangements}
\label{subsec:quadra1D_instability}
The above study demonstrates that there exist probability measures $\mu$ and $\nu$ for which property
    \begin{equation*}
        P(\mu,\nu)\text{ : }\quad \pimon \text{ or } \piantimon \text{ is an optimal correspondence plan between }\mu\text{ and }\nu
    \end{equation*}
    does not hold. However, as it is very likely in practice when generating discrete measures at random, one could ask if property $P$ is at least \emph{stable}, i.e.~if when we have $\mu_0$ and $\nu_0$ satisfying $P(\mu_0,\nu_0)$ there is a small ball around $\mu_0$ and $\nu_0$ (for a given distance, say Wasserstein-$2$) inside which property $P$ remains valid. A negative answer to this---besides, in the symmetric case---is given by the counter-example by \citeauthor{beinert2022assignment}~\cite{beinert2022assignment} with an increasing number of points:
    \begin{proposition}[Instability of the optimality of monotone rearrangements]
        \label{prop:instability}
        There exists two \emph{symmetric} measures $\mu,\nu$ on $\RR$ and sequences $(\mu_n)_n$, $(\nu_n)_n$ that weakly converge to $\mu$, $\nu$ such that optimal plans $\pi_n$ between $\mu_n$ and $\nu_n$ are never a monotone rearrangement.
    \end{proposition}
    \begin{proof}
        We consider $\mu=\nu=\delta_0$ and the discrete measures $\mu_n=\frac{1}{n}\sum_{i=1}^n\delta_{x_i}$ and $\nu_n=\frac{1}{n}\sum_{i=1}^n\delta_{y_i}$ defined as follows for $n\geq7$:
        \begin{equation*}
        x_{i}\defeq \begin{cases}
        -1 & \text{for }i=1 \\
        (i-\frac{n+1}{2})\frac1{n^2} & \text{for }i=2,\dots,n-1 \\
        1 & \text{for }i=n
        \end{cases}\quad \text{and}\quad y_{i}\defeq \begin{cases}
        -1 & \text{for }i=1 \\
        -1+\frac1{n^2} & \text{for }i=2 \\
        (i-2)\frac1{n^2} & \text{for }i=3,\dots ,n
        \end{cases}
        \end{equation*}
        which is simply the counter-example from \citeauthor{beinert2022assignment}~\cite{beinert2022assignment} with $n$ points and $\epsilon_n=1/n^2$. Since $n\geq7$, $\epsilon_n<2/(n-3)$ and the identity or anti-identity mappings are not optimal between $\mu_n$ and $\nu_n$.
    By direct computation,
    \begin{equation*}
    \W_2^2(\delta_0,\mu_n)= O(2/n+\epsilon_n^2n^2)\xrightarrow[]{n\to\infty}0,
    \end{equation*}
    and the exact same goes for $\nu=\delta_0$ and $\nu_n$.
    \end{proof}
    One can actually obtain non-degenerate (although not symmetric anymore) examples of such measures $\mu$, $\nu$. We start from the counter-example given in \cite{beinert2022assignment} with $N=7$ points and $\varepsilon=10^{-2}$, that we convolve with a Gaussian measure of standard deviation $\sigma$ as before. We then plot as a function of $m\in[m_\text{min},m_\text{max}]$ the (true) GW cost of a plan $\pi_m\opt$ optimal for the linearized GW problem: $\pi\opt_m\in\argmin_{\pi}\,\langle C_{\gw(m)},\pi\rangle$. The minimum values of this graph are attained by the covariances of optimal correspondence plans, as explained in \Cref{subsec:quadra1D_adversarial}. Hence if $\sigma$ is small, optimal correspondence plans are not monotone by construction and the minima are not located on the boundary of the domain.
    On the contrary, when $\sigma$ is large, 
    a monotone rearrangement is optimal again. To study the phase transition, we plot on \Cref{fig:no-stab} the landscape of $m\mapsto \gw(\pi\opt_m)$ while gradually increasing the value of $\sigma$.

    \begin{figure}[h]
        \centering
        \begin{subfigure}[b]{.24\linewidth}
            \centering
\begin{tikzpicture}[scale=.45]

\begin{axis}[
    xlabel={$m$},
    xmin=-0.194213820111278, xmax=0.194213820111278,
    ymin=-0.290182883647815, ymax=-0.263666083991494,
    grid = major,
    axis lines=left,
    ]
    \addplot[ultra thick, no marks, tabblue]
table {%
-0.155742431435461 -0.275380102915412
-0.153651929000085 -0.275380102915317
-0.15156142656471 -0.275380102915663
-0.149470924129335 -0.275380102987755
-0.147380421693959 -0.275380102911988
-0.145289919258584 -0.275380102915687
-0.143199416823209 -0.275380102915632
-0.141108914387834 -0.27538010291483
-0.139018411952458 -0.275380102911644
-0.136927909517083 -0.275380102912229
-0.134837407081708 -0.275380102911161
-0.132746904646332 -0.275380102911209
-0.130656402210957 -0.275380102913941
-0.128565899775582 -0.27538010291118
-0.126475397340206 -0.275380102928466
-0.124384894904831 -0.275380102912072
-0.122294392469456 -0.275380102958045
-0.12020389003408 -0.275380102977565
-0.118113387598705 -0.275380102915782
-0.11602288516333 -0.275380102924645
-0.113932382727954 -0.275380102937224
-0.111841880292579 -0.275380102924637
-0.109751377857204 -0.275380102921386
-0.107660875421828 -0.274984380875121
-0.105570372986453 -0.27307776242602
-0.103479870551078 -0.267583241731402
-0.101389368115703 -0.264871393066782
-0.0992988656803272 -0.269701906821325
-0.0972083632449519 -0.273747309218643
-0.0951178608095766 -0.276763076781454
-0.0930273583742013 -0.278452821547287
-0.0909368559388259 -0.279813236302683
-0.0888463535034506 -0.280699860365686
-0.0867558510680753 -0.28176530988425
-0.0846653486327 -0.282626579964681
-0.0825748461973247 -0.283285925684313
-0.0804843437619494 -0.284162356477211
-0.0783938413265741 -0.284596226405426
-0.0763033388911987 -0.285460495863862
-0.0742128364558234 -0.285695449385922
-0.0721223340204481 -0.285892615172939
-0.0700318315850728 -0.286140632459021
-0.0679413291496975 -0.286291203951305
-0.0658508267143222 -0.286528382400582
-0.0637603242789468 -0.286661385181496
-0.0616698218435715 -0.286874006233074
-0.0595793194081962 -0.287008310556664
-0.0574888169728209 -0.287203025927852
-0.0553983145374456 -0.28731435847981
-0.0533078121020703 -0.287490792079257
-0.051217309666695 -0.287601645619975
-0.0491268072313197 -0.287759349811789
-0.0470363047959443 -0.287847550132569
-0.044945802360569 -0.287987820154035
-0.0428552999251937 -0.288070168457747
-0.0407647974898184 -0.288188981177386
-0.0386742950544431 -0.288253888132255
-0.0365837926190678 -0.288362513098762
-0.0344932901836924 -0.288504547888517
-0.0324027877483171 -0.288552644712416
-0.0303122853129418 -0.2886241279915
-0.0282217828775665 -0.288666286619716
-0.0261312804421912 -0.288726360913962
-0.0240407780068159 -0.288802376062399
-0.0219502755714406 -0.288830628490637
-0.0198597731360652 -0.288868156612365
-0.0177692707006899 -0.288913102027494
-0.0156787682653146 -0.288926480081151
-0.0135882658299393 -0.288954796580086
-0.011497763394564 -0.288966902554207
-0.00940726095918867 -0.288976148368565
-0.00731675852381336 -0.288977574572528
-0.00522625608843805 -0.288976667376545
-0.00313575365306271 -0.288957787719671
-0.00104525121768739 -0.288923093581993
0.00104525121768792 -0.288923093581994
0.00313575365306323 -0.288957787719671
0.00522625608843855 -0.288976667376545
0.00731675852381386 -0.288977574572528
0.00940726095918917 -0.288976148368565
0.0114977633945645 -0.288966902554208
0.0135882658299398 -0.288954796580086
0.0156787682653151 -0.28892648008115
0.0177692707006904 -0.288913102027494
0.0198597731360657 -0.288868156612365
0.0219502755714411 -0.288830628490636
0.0240407780068164 -0.288802376062399
0.0261312804421917 -0.288726360913963
0.028221782877567 -0.288666286619716
0.0303122853129423 -0.2886241279915
0.0324027877483176 -0.288552644712416
0.034493290183693 -0.288504547888517
0.0365837926190683 -0.288362513098762
0.0386742950544436 -0.288253888132256
0.0407647974898189 -0.288188981177385
0.0428552999251942 -0.288070168457747
0.0449458023605695 -0.287987820154035
0.0470363047959448 -0.287847550132569
0.0491268072313202 -0.287759349811789
0.0512173096666955 -0.287601645619975
0.0533078121020708 -0.287490792079257
0.0553983145374461 -0.28731435847981
0.0574888169728214 -0.287203025927852
0.0595793194081967 -0.287008310556665
0.061669821843572 -0.286874006233074
0.0637603242789474 -0.286661385181495
0.0658508267143227 -0.286528382400582
0.067941329149698 -0.286291203951305
0.0700318315850733 -0.286140632459021
0.0721223340204486 -0.285892615172939
0.0742128364558239 -0.285695449385923
0.0763033388911993 -0.285460495863862
0.0783938413265746 -0.284596226405427
0.0804843437619499 -0.284162356477211
0.0825748461973252 -0.283285925684313
0.0846653486327005 -0.282626579964682
0.0867558510680758 -0.28176530988425
0.0888463535034511 -0.280699860365683
0.0909368559388264 -0.279813236302684
0.0930273583742018 -0.278452821547288
0.0951178608095771 -0.276763076781455
0.0972083632449524 -0.273747309218644
0.0992988656803277 -0.269701906821326
0.101389368115703 -0.264871393066782
0.103479870551078 -0.267583241731403
0.105570372986454 -0.273077762426022
0.107660875421829 -0.274984380875123
0.109751377857204 -0.275380102921387
0.11184188029258 -0.275380102924638
0.113932382727955 -0.275380102937226
0.11602288516333 -0.275380102924646
0.118113387598706 -0.275380102915784
0.120203890034081 -0.275380102977567
0.122294392469456 -0.275380102958047
0.124384894904831 -0.275380102912074
0.126475397340207 -0.275380102928468
0.128565899775582 -0.275380102911182
0.130656402210957 -0.275380102913942
0.132746904646333 -0.275380102911211
0.134837407081708 -0.275380102911163
0.136927909517083 -0.27538010291223
0.139018411952459 -0.275380102911646
0.141108914387834 -0.275380102914831
0.143199416823209 -0.275380102915633
0.145289919258585 -0.275380102915689
0.14738042169396 -0.275380102911989
0.149470924129335 -0.275380102987756
0.151561426564711 -0.275380102915665
0.153651929000086 -0.275380102915319
0.155742431435461 -0.275380102915413
};
\addplot [thick, dashed]
table {%
-0.155742431435461 -0.275380102915412
-0.155742431435461 -0.290182883647815
} node[above right] {$m_\text{min}$};
\addplot [thick, dashed]
table {%
0.155742431435461 -0.275380102915413
0.155742431435461 -0.290182883647815
} node[above left] {$m_\text{max}$};
\addplot[ultra thick, tabblue, only marks, mark=*]
table {%
-0.155742431435461 -0.275380102915412
0.155742431435461 -0.275380102915413
};
\end{axis}
\end{tikzpicture}
            \vspace{-8mm}
            \caption*{$\sigma_1=8.10^{-3}$}
        \end{subfigure}
        \begin{subfigure}[b]{.24\linewidth}
            \centering
\begin{tikzpicture}[scale=.45]

\begin{axis}[
    xlabel={$m$},
    xmin=-0.194213820111278, xmax=0.194213820111278,
    ymin=-0.298275232893912, ymax=-0.279511366424687,
    grid = major,
    axis lines=left,
    ]
    \addplot[ultra thick, no marks, tabblue]
table {%
-0.16323869942782 -0.29420231353058
-0.161047575945567 -0.29420231353281
-0.158856452463314 -0.294202313530255
-0.156665328981062 -0.294202313542228
-0.154474205498809 -0.29420231354487
-0.152283082016557 -0.294202313543036
-0.150091958534304 -0.294202313546145
-0.147900835052051 -0.294202313597041
-0.145709711569799 -0.294202313531916
-0.143518588087546 -0.294202313530501
-0.141327464605293 -0.294202313554674
-0.139136341123041 -0.294202313529766
-0.136945217640788 -0.294202313549925
-0.134754094158536 -0.294202313578277
-0.132562970676283 -0.294202313577799
-0.13037184719403 -0.294202313574221
-0.128180723711778 -0.294202313540318
-0.125989600229525 -0.294202313539246
-0.123798476747272 -0.294202313530204
-0.12160735326502 -0.294202313533479
-0.119416229782767 -0.294202313542627
-0.117225106300515 -0.2942023135335
-0.115033982818262 -0.294202313530682
-0.112842859336009 -0.29420231355232
-0.110651735853757 -0.294202313532388
-0.108460612371504 -0.29420231352961
-0.106269488889252 -0.293662003464201
-0.104078365406999 -0.283835622656074
-0.101887241924746 -0.282329170559471
-0.0996961184424937 -0.280364269446016
-0.0975049949602411 -0.282663909440542
-0.0953138714779885 -0.286618941589105
-0.0931227479957358 -0.287823038417687
-0.0909316245134832 -0.288747284047929
-0.0887405010312306 -0.28940306775179
-0.086549377548978 -0.290045396928863
-0.0843582540667254 -0.290690228706021
-0.0821671305844728 -0.29133998132021
-0.0799760071022201 -0.291660691648057
-0.0777848836199675 -0.292296848159351
-0.0755937601377149 -0.292611398746729
-0.0734026366554623 -0.292928548312571
-0.0712115131732097 -0.293541044594295
-0.0690203896909571 -0.29384092002434
-0.0668292662087044 -0.294140977350911
-0.0646381427264518 -0.294430575669961
-0.0624470192441992 -0.29471422157062
-0.0602558957619466 -0.294992688276769
-0.058064772279694 -0.295263736570535
-0.0558736487974413 -0.295523587445989
-0.0536825253151887 -0.295775806092435
-0.0514914018329361 -0.296028200829168
-0.0493002783506835 -0.296264321705584
-0.0471091548684309 -0.296487481539318
-0.0449180313861783 -0.296535813038205
-0.0427269079039257 -0.29668954933424
-0.040535784421673 -0.296869332477093
-0.0383446609394204 -0.296909103237188
-0.0361535374571678 -0.29702454875784
-0.0339624139749152 -0.297063059685438
-0.0317712904926626 -0.297151951060367
-0.0295801670104099 -0.297248783882658
-0.0273890435281573 -0.297272446193699
-0.0251979200459047 -0.297325150689385
-0.0230067965636521 -0.297339704430846
-0.0208156730813995 -0.297377963155265
-0.0186245495991469 -0.297408687000258
-0.0164334261168942 -0.297413887443075
-0.0142423026346416 -0.297422329872583
-0.012051179152389 -0.29742003337139
-0.0098600556701364 -0.297416842358672
-0.00766893218788378 -0.297390534298154
-0.00547780870563117 -0.297313637251217
-0.00328668522337855 -0.297266341387402
-0.00109556174112593 -0.297192861103505
0.00109556174112671 -0.297192861103506
0.00328668522337933 -0.297266341387402
0.00547780870563194 -0.297313637251217
0.00766893218788456 -0.297390534298153
0.00986005567013717 -0.297416842358671
0.0120511791523898 -0.297420033371391
0.0142423026346424 -0.297422329872583
0.016433426116895 -0.297413887443074
0.0186245495991476 -0.297408687000259
0.0208156730814003 -0.297377963155266
0.0230067965636529 -0.297339704430846
0.0251979200459055 -0.297325150689386
0.0273890435281581 -0.2972724461937
0.0295801670104107 -0.297248783882659
0.0317712904926633 -0.297151951060367
0.033962413974916 -0.297063059685437
0.0361535374571686 -0.29702454875784
0.0383446609394212 -0.296909103237188
0.0405357844216738 -0.296869332477093
0.0427269079039264 -0.296689549334241
0.044918031386179 -0.296535813038206
0.0471091548684316 -0.296487481539318
0.0493002783506843 -0.296264321705586
0.0514914018329369 -0.296028200829168
0.0536825253151895 -0.295775806092436
0.0558736487974421 -0.29552358744599
0.0580647722796948 -0.295263736570536
0.0602558957619474 -0.294992688276769
0.0624470192442 -0.29471422157062
0.0646381427264526 -0.294430575669961
0.0668292662087052 -0.294140977350912
0.0690203896909578 -0.29384092002434
0.0712115131732105 -0.293541044594295
0.0734026366554631 -0.292928548312572
0.0755937601377157 -0.292611398746728
0.0777848836199683 -0.29229684815935
0.0799760071022209 -0.29166069164806
0.0821671305844735 -0.291339981320211
0.0843582540667261 -0.290690228706022
0.0865493775489788 -0.290045396928864
0.0887405010312314 -0.289403067751791
0.090931624513484 -0.288747284047931
0.0931227479957366 -0.28782303841769
0.0953138714779893 -0.286618941589107
0.0975049949602418 -0.282663909440544
0.0996961184424945 -0.280364269446017
0.101887241924747 -0.282329170559473
0.104078365407 -0.283835622656077
0.106269488889252 -0.293662003464205
0.108460612371505 -0.294202313529613
0.110651735853758 -0.29420231353239
0.11284285933601 -0.294202313552323
0.115033982818263 -0.294202313530684
0.117225106300515 -0.294202313533503
0.119416229782768 -0.294202313542629
0.121607353265021 -0.294202313533483
0.123798476747273 -0.294202313530207
0.125989600229526 -0.294202313539249
0.128180723711779 -0.29420231354032
0.130371847194031 -0.294202313574224
0.132562970676284 -0.294202313577802
0.134754094158536 -0.29420231357828
0.136945217640789 -0.294202313549927
0.139136341123042 -0.294202313529769
0.141327464605294 -0.294202313554678
0.143518588087547 -0.294202313530504
0.145709711569799 -0.294202313531918
0.147900835052052 -0.294202313597044
0.150091958534305 -0.294202313546148
0.152283082016557 -0.294202313543038
0.15447420549881 -0.294202313544873
0.156665328981063 -0.29420231354223
0.158856452463315 -0.294202313530257
0.161047575945568 -0.294202313532812
0.16323869942782 -0.294202313530582
};
\addplot [thick, dashed]
table {%
-0.16323869942782 -0.29420231353058
-0.16323869942782 -0.298275232893912
};
\addplot [thick, dashed]
table {%
0.16323869942782 -0.294202313530582
0.16323869942782 -0.298275232893912
};
\addplot[ultra thick, tabblue, only marks, mark=*]
table {%
-0.16323869942782 -0.29420231353058
0.16323869942782 -0.294202313530582
};
\end{axis}

\end{tikzpicture}
            \vspace{-8mm}
            \caption*{$\sigma_2=8.8.10^{-3}$}
        \end{subfigure}
        \begin{subfigure}[b]{.24\linewidth}
            \centering
\begin{tikzpicture}[scale=.45]

\begin{axis}[
    xlabel={$m$},
    xmin=-0.194213820111278, xmax=0.194213820111278,
    ymin=-0.307315740861217, ymax=-0.287759251986567,
    grid = major,
    axis lines=left,
    ]
    \addplot[ultra thick, no marks, tabblue]
table {%
-0.167274033518982 -0.304517718566687
-0.165028744478459 -0.304517718568742
-0.162783455437935 -0.304517718570363
-0.160538166397412 -0.304517718568679
-0.158292877356889 -0.304517718637609
-0.156047588316366 -0.304517718567961
-0.153802299275842 -0.304517718573297
-0.151557010235319 -0.304517718625788
-0.149311721194796 -0.304517718580195
-0.147066432154273 -0.304517718580947
-0.144821143113749 -0.3045177185773
-0.142575854073226 -0.304517718586349
-0.140330565032703 -0.304517718597216
-0.13808527599218 -0.304517718598273
-0.135839986951656 -0.304517718565065
-0.133594697911133 -0.304517718568947
-0.13134940887061 -0.304517718639642
-0.129104119830087 -0.30451771856813
-0.126858830789564 -0.304517718587666
-0.12461354174904 -0.304517718565426
-0.122368252708517 -0.304517718594188
-0.120122963667994 -0.304517718568907
-0.117877674627471 -0.304517718565332
-0.115632385586947 -0.30451771856642
-0.113387096546424 -0.304517718567446
-0.111141807505901 -0.304517718565243
-0.108896518465378 -0.304517718545104
-0.106651229424854 -0.296954358568909
-0.104405940384331 -0.295086143931261
-0.102160651343808 -0.290640125735379
-0.0999153623032845 -0.288557274208142
-0.0976700732627613 -0.289340746876728
-0.095424784222238 -0.292098838299968
-0.0931794951817148 -0.293285863763
-0.0909342061411915 -0.294109239275629
-0.0886889171006683 -0.294839822874232
-0.0864436280601451 -0.295372997474895
-0.0841983390196218 -0.295905292043517
-0.0819530499790986 -0.296441039473682
-0.0797077609385753 -0.296965635697388
-0.0774624718980521 -0.29722601702299
-0.0752171828575288 -0.297745192259393
-0.0729718938170056 -0.29800311597306
-0.0707266047764823 -0.298306263952383
-0.0684813157359591 -0.298742860686472
-0.0662360266954358 -0.298981938238562
-0.0639907376549126 -0.299210806282578
-0.0617454486143894 -0.299433059286501
-0.0595001595738661 -0.299655688711825
-0.0572548705333429 -0.299863500888779
-0.0550095814928196 -0.300063151232501
-0.0527642924522964 -0.300254465410563
-0.0505190034117731 -0.300439775460125
-0.0482737143712499 -0.300609584333982
-0.0460284253307266 -0.300650778686716
-0.0437831362902034 -0.300757744896106
-0.0415378472496801 -0.300886329007365
-0.0392925582091569 -0.30100221947025
-0.0370472691686336 -0.301111586503159
-0.0348019801281104 -0.30121391027529
-0.0325566910875872 -0.301305175367807
-0.0303114020470639 -0.301384993481121
-0.0280661130065407 -0.301490227132168
-0.0258208239660174 -0.301555586045821
-0.0235755349254942 -0.301569823953897
-0.0213302458849709 -0.301591192122879
-0.0190849568444477 -0.301598322595057
-0.0168396678039244 -0.301603028075171
-0.0145943787634012 -0.30159432680194
-0.0123490897228779 -0.301586095549947
-0.0101038006823547 -0.30154924826321
-0.00785851164183143 -0.301491211580021
-0.0056132226013082 -0.301398189251549
-0.00336793356078494 -0.301304339396397
-0.0011226445202617 -0.301198781007987
0.00112264452026153 -0.301198781007986
0.0033679335607848 -0.301304339396397
0.00561322260130803 -0.30139818925155
0.0078585116418313 -0.301491211580021
0.0101038006823545 -0.30154924826321
0.0123490897228778 -0.301586095549947
0.014594378763401 -0.301594326801939
0.0168396678039243 -0.301603028075171
0.0190849568444475 -0.301598322595057
0.0213302458849708 -0.301591192122879
0.023575534925494 -0.301569823953897
0.0258208239660173 -0.301555586045821
0.0280661130065405 -0.301490227132168
0.0303114020470638 -0.301384993481121
0.032556691087587 -0.301305175367807
0.0348019801281103 -0.30121391027529
0.0370472691686335 -0.301111586503159
0.0392925582091567 -0.301002219470249
0.04153784724968 -0.300886329007365
0.0437831362902032 -0.300757744896106
0.0460284253307265 -0.300650778686716
0.0482737143712497 -0.300609584333981
0.050519003411773 -0.300439775460125
0.0527642924522962 -0.300254465410563
0.0550095814928195 -0.300063151232501
0.0572548705333427 -0.299863500888779
0.059500159573866 -0.299655688711825
0.0617454486143892 -0.299433059286501
0.0639907376549124 -0.299210806282578
0.0662360266954357 -0.298981938238561
0.0684813157359589 -0.298742860686472
0.0707266047764822 -0.298306263952382
0.0729718938170054 -0.298003115973061
0.0752171828575287 -0.297745192259393
0.0774624718980519 -0.29722601702299
0.0797077609385752 -0.296965635697387
0.0819530499790984 -0.296441039473682
0.0841983390196217 -0.295905292043516
0.0864436280601449 -0.295372997474895
0.0886889171006681 -0.294839822874232
0.0909342061411914 -0.294109239275629
0.0931794951817147 -0.293285863763004
0.0954247842222379 -0.292098838299968
0.0976700732627611 -0.289340746876729
0.0999153623032844 -0.288557274208142
0.102160651343808 -0.290640125735379
0.104405940384331 -0.295086143931261
0.106651229424854 -0.296954358568909
0.108896518465377 -0.304517718545103
0.111141807505901 -0.304517718565242
0.113387096546424 -0.304517718567445
0.115632385586947 -0.304517718566419
0.11787767462747 -0.304517718565331
0.120122963667994 -0.304517718568907
0.122368252708517 -0.304517718594187
0.12461354174904 -0.304517718565425
0.126858830789563 -0.304517718587666
0.129104119830087 -0.30451771856813
0.13134940887061 -0.304517718639641
0.133594697911133 -0.304517718568946
0.135839986951656 -0.304517718565065
0.13808527599218 -0.304517718598272
0.140330565032703 -0.304517718597214
0.142575854073226 -0.30451771858635
0.144821143113749 -0.3045177185773
0.147066432154273 -0.304517718580947
0.149311721194796 -0.304517718580195
0.151557010235319 -0.304517718625787
0.153802299275842 -0.304517718573296
0.156047588316366 -0.30451771856796
0.158292877356889 -0.304517718637609
0.160538166397412 -0.304517718568679
0.162783455437935 -0.304517718570363
0.165028744478459 -0.304517718568742
0.167274033518982 -0.304517718566687
};
\addplot [thick, dashed]
table {%
-0.167274033518982 -0.304517718566687
-0.167274033518982 -0.307315740861217
};
\addplot [thick, dashed]
table {%
0.167274033518982 -0.304517718566687
0.167274033518982 -0.307315740861217
};
\addplot[ultra thick, tabblue, only marks, mark=*]
table {%
-0.167274033518982 -0.304517718566687
0.167274033518982 -0.304517718566687
};
\end{axis}

\end{tikzpicture}
            \vspace{-8mm}
            \caption*{$\sigma_3=10^{-2}$}
        \end{subfigure}
        \begin{subfigure}[b]{.24\linewidth}
            \centering
\begin{tikzpicture}[scale=.45]

\begin{axis}[
    xlabel={$m$},
    xmin=-0.194213820111278, xmax=0.194213820111278,
    ymin=-0.335194226489427, ymax=-0.304544811107574,
    grid = major,
    axis lines=left,
    ]
    \addplot[ultra thick, no marks, tabblue]
table {%
-0.17655801828298 -0.331891980335707
-0.174188111997303 -0.331891980319862
-0.171818205711625 -0.331891980283592
-0.169448299425948 -0.331891980223302
-0.16707839314027 -0.331891980243792
-0.164708486854592 -0.331891980263956
-0.162338580568915 -0.331891980286217
-0.159968674283237 -0.331891980247376
-0.15759876799756 -0.331891980285448
-0.155228861711882 -0.331891980281795
-0.152858955426204 -0.331891980278678
-0.150489049140527 -0.33189198026967
-0.148119142854849 -0.331891980230003
-0.145749236569172 -0.331891980218381
-0.143379330283494 -0.33189198026904
-0.141009423997816 -0.331891980235271
-0.138639517712139 -0.331891980245376
-0.136269611426461 -0.331891980236535
-0.133899705140784 -0.331891980234786
-0.131529798855106 -0.331891980237326
-0.129159892569429 -0.331891980281328
-0.126789986283751 -0.331891980316321
-0.124420079998073 -0.331891980319116
-0.122050173712396 -0.33189198032223
-0.119680267426718 -0.331891980219255
-0.117310361141041 -0.331891980221669
-0.114940454855363 -0.331891980219858
-0.112570548569685 -0.331891980305998
-0.110200642284008 -0.331891980243904
-0.10783073599833 -0.331287847826922
-0.105460829712653 -0.330572527047089
-0.103090923426975 -0.326160446720041
-0.100721017141297 -0.323476623849002
-0.0983511108556199 -0.318617854744208
-0.0959812045699423 -0.314451563818012
-0.0936112982842647 -0.311780522830264
-0.0912413919985871 -0.307290484313763
-0.0888714857129096 -0.30631739164939
-0.086501579427232 -0.305931684452569
-0.0841316731415544 -0.305847057261295
-0.0817617668558768 -0.306225122056043
-0.0793918605701992 -0.306719707898304
-0.0770219542845216 -0.30745833888048
-0.074652047998844 -0.307963295225527
-0.0722821417131664 -0.308345071579
-0.0699122354274888 -0.308689571387006
-0.0675423291418113 -0.308823636646417
-0.0651724228561337 -0.309090908398651
-0.0628025165704561 -0.309307871553507
-0.0604326102847785 -0.309543536678146
-0.0580627039991009 -0.309783972735727
-0.0556927977134233 -0.30990452907952
-0.0533228914277457 -0.310067206398183
-0.0509529851420681 -0.310221809017009
-0.0485830788563905 -0.31032126586718
-0.046213172570713 -0.310412832980105
-0.0438432662850354 -0.310491619071646
-0.0414733599993578 -0.310595864116286
-0.0391034537136802 -0.310664756134108
-0.0367335474280026 -0.310714651604707
-0.034363641142325 -0.310754824592714
-0.0319937348566474 -0.310772806989539
-0.0296238285709698 -0.31078712506486
-0.0272539222852923 -0.310790032166037
-0.0248840159996147 -0.310778157852204
-0.0225141097139371 -0.310761086218631
-0.0201442034282595 -0.310735123517862
-0.0177742971425819 -0.310694345336652
-0.0154043908569043 -0.310627228156053
-0.0130344845712267 -0.310579712144007
-0.0106645782855491 -0.310501633602625
-0.00829467199987158 -0.310311126112664
-0.00592476571419398 -0.310172326882364
-0.00355485942851638 -0.310074688943368
-0.00118495314283881 -0.309914842104535
0.00118495314283878 -0.309914842104535
0.00355485942851638 -0.310074688943368
0.00592476571419395 -0.310172326882363
0.00829467199987155 -0.310311126112665
0.0106645782855491 -0.310501633602625
0.0130344845712267 -0.310579712144006
0.0154043908569043 -0.310627228156053
0.0177742971425819 -0.310694345336651
0.0201442034282595 -0.310735123517862
0.0225141097139371 -0.310761086218631
0.0248840159996147 -0.310778157852205
0.0272539222852922 -0.310790032166037
0.0296238285709698 -0.31078712506486
0.0319937348566474 -0.310772806989539
0.034363641142325 -0.310754824592714
0.0367335474280026 -0.310714651604707
0.0391034537136802 -0.310664756134108
0.0414733599993578 -0.310595864116285
0.0438432662850354 -0.310491619071646
0.0462131725707129 -0.310412832980105
0.0485830788563905 -0.31032126586718
0.0509529851420681 -0.310221809017009
0.0533228914277457 -0.310067206398182
0.0556927977134233 -0.309904529079519
0.0580627039991009 -0.309783972735727
0.0604326102847785 -0.309543536678146
0.0628025165704561 -0.309307871553507
0.0651724228561336 -0.309090908398651
0.0675423291418112 -0.308823636646417
0.0699122354274888 -0.308689571387006
0.0722821417131664 -0.308345071579
0.074652047998844 -0.307963295225526
0.0770219542845216 -0.307458338880479
0.0793918605701992 -0.306719707898304
0.0817617668558768 -0.306225122056043
0.0841316731415543 -0.305847057261295
0.086501579427232 -0.305931684452569
0.0888714857129095 -0.306317391649391
0.0912413919985871 -0.307290484313762
0.0936112982842647 -0.311780522830264
0.0959812045699423 -0.314451563818012
0.0983511108556199 -0.318617854744208
0.100721017141297 -0.323476623849002
0.103090923426975 -0.326160446720041
0.105460829712653 -0.330572527047089
0.10783073599833 -0.331287847826922
0.110200642284008 -0.331891980243903
0.112570548569685 -0.331891980305998
0.114940454855363 -0.331891980219857
0.117310361141041 -0.33189198022167
0.119680267426718 -0.331891980219255
0.122050173712396 -0.331891980322229
0.124420079998073 -0.331891980319116
0.126789986283751 -0.331891980316321
0.129159892569429 -0.331891980281327
0.131529798855106 -0.331891980237326
0.133899705140784 -0.331891980234787
0.136269611426461 -0.331891980236535
0.138639517712139 -0.331891980245377
0.141009423997816 -0.331891980235271
0.143379330283494 -0.331891980269039
0.145749236569172 -0.33189198021838
0.148119142854849 -0.331891980230003
0.150489049140527 -0.33189198026967
0.152858955426204 -0.331891980278678
0.155228861711882 -0.331891980281795
0.15759876799756 -0.331891980285447
0.159968674283237 -0.331891980247376
0.162338580568915 -0.331891980286216
0.164708486854592 -0.331891980263955
0.16707839314027 -0.331891980243793
0.169448299425948 -0.331891980223302
0.171818205711625 -0.331891980283591
0.174188111997303 -0.331891980319861
0.17655801828298 -0.331891980335706
};
\addplot [thick, dashed]
table {%
-0.17655801828298 -0.331891980335707
-0.17655801828298 -0.335194226489427
};
\addplot [thick, dashed]
table {%
0.17655801828298 -0.331891980335706
0.17655801828298 -0.335194226489427
};
\addplot[ultra thick, tabblue, only marks, mark=*]
table {%
-0.17655801828298 -0.331891980335707
0.17655801828298 -0.331891980335706
};
\end{axis}

\end{tikzpicture}
            \vspace{-8mm}
            \caption*{$\sigma_4=3.10^{-2}$}
        \end{subfigure}
        \caption{Evolution of the graph of $m\mapsto \gw(\pi\opt_m)$ when making $\sigma$ vary on the counter-example of \citeauthor{beinert2022assignment}~\cite{beinert2022assignment} with $N=7$ points and $\varepsilon=10^{-2}$. Parameters: $N_{x}=100$, $N_{m}=150$.}
        \label{fig:no-stab}
    \end{figure}
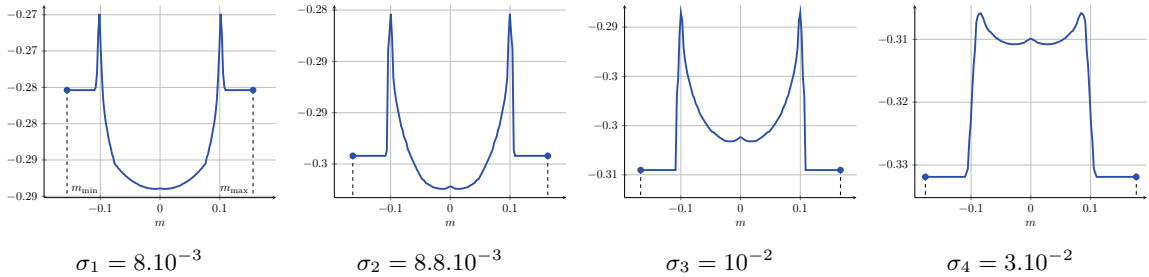

    Looking at \cref{fig:no-stab}, it is worth noting that there is an incentive for optimal plans at covariances close to $m_\text{min}$ or $m_\text{max}$ to be the monotone rearrangements, as the horizontal portions of the plot suggest. More importantly, it can be observed that when $\sigma=\sigma_3$ or $\sigma_4$, the monotone rearrangements are optimal, as their covariances realize the minimum of $m\mapsto \gw(\pi\opt_m)$; unlike for $\sigma_1$ and $\sigma_2$, for which the minimum value of the plot is located near zero. Hence there should exist a $\sigma_0\in(\sigma_2,\sigma_3)$ for which the convolved measures have both $\pimon$, $\piantimon$ and another $\pi_0$ as optimal correspondence plans; it is direct that property $P$ does not hold in the neighborhood of these specific measures $\mu_0$ and $\nu_0$.

\subsubsection{A positive result for measures with two components}
\label{subsec:quadra_1D_positive}
In the following, $\mu_1$, $\mu_2$, $\nu_1$ and $\nu_2$ are four probability measures supported on a compact interval $A \subset \RR$. Denote $\Delta \defeq \diam(A)$, and fix $t \in (0,1)$ and $K > \Delta$. Let $\tau_K : x \mapsto x + K$ denote the translation by $K$, and $A+K = \tau_K(A) = \{x + K\mid x \in A\}$. Now, introduce the measures
\begin{equation*}
 \mu \defeq (1-t) \mu_1 + t \tau_{K\pushonly} \mu_2 \qquad \text{and}  \qquad \nu \defeq (1-t) \nu_1 + t \tau_{K\pushonly} \nu_2.
 \end{equation*}
Note that $\mu_1$ and $\tau_{K\pushonly} \mu_2$ have disjoint supports, and same for $\nu_1$ and $\tau_{K\pushonly}\nu_2$. We have the following result.

\begin{proposition}[Optimality of monotone rearrangements for measures with two components]
    \label{prop:measure-separation}
    For $K$ large enough, the unique optimal plan for the squared distance cost between  $\mu$ and $\nu$ is given by one of the two monotone rearrangements (non-decreasing or non-increasing).
\end{proposition}

\begin{remark}[Long-range correspondences]
Conditions of \cref{prop:measure-separation} illustrate that monotone rearrangements are favored when $\mu$ and $\nu$ both contain one or more outliers.
The proof actually shows the importance of long-range correspondences or global effects over the local correspondences on the plan. In other words, even though monotone rearrangements may not be optimal locally, global correspondences favor them. Moreover, these global correspondences have proportionally more weight in the GW functional since the cost is the squared difference of the squared distances. In conclusion, pairs of points that are at long distances tend to be put in correspondence. In turn, this correspondence, as shown in the proof, favors monotone rearrangements. This argument gives some insight into the fact that a monotone rearrangement is often optimal and it is made quantitative by \cref{prop:measure-separation}.
\end{remark}
\noindent We first prove the following lemma:
\begin{lemma}
\label{lemma:measure-separation}
In the setting described above, there exists $K_0>0$ such that if $K \geq K_0$,
every $\pi$ optimal plan for $\gw(\mu,\nu)$ can be decomposed as $\pi = \pi_1 + \pi_2$, where either:
\begin{enumerate}[(i)]
    \item $\pi_1$ is supported on $A \times A$ and $\pi_2$ on $(A + K) \times (A + K)$ (that is, we separately transport $\mu_1$ to $\nu_1$ and $\tau_{K\pushonly}\mu_2$ to $\tau_{K\pushonly}\nu_2$), or
    \item $\pi_1$ is supported on $A \times (A+K)$ and $\pi_2$ on $A \times (A + K)$ (that is, we separately transport $\mu_1$ to $\tau_{K\pushonly}\nu_2$ and $\mu_2$ to $\tau_{K\pushonly}\nu_1$).
\end{enumerate}
Furthermore, whenever $t \neq \frac{1}{2}$, only (i) can occur.
\end{lemma}

    \begin{figure}[h]
        \centering
        \begin{tikzpicture}[line cap=round,line join=round,scale=0.7]
    \def\margin{.5}
    \def\K{7}
    \def\diaml{4}
    \def\height{1.5}
    \def\vertical{4}
    \def\spacearrow{.4}
    \def\borellength{.75}
    \def\shade{15}
    \def\ellip{1}
    \def\spacesq{1.2}
    
    \filldraw [fill=black!5, draw=none] (0,0) rectangle (\diaml,\diaml);
    \filldraw [fill=black!5, draw=none] (\K,\K) rectangle (\K+\diaml,\K+\diaml);
    \draw[tabred,fill=tabred!\shade]
        (0,0)..controls +(-70:\diaml*0.1) and +(180:\diaml*0.1)..
        (\diaml/5,-0.5*\height)..controls +(0:\diaml*0.1) and +(180:\diaml*0.1)..
        (2*\diaml/5,-0.5*\height)..controls +(0:\diaml*0.1) and +(180:\diaml*0.2)..
        (3*\diaml/4,-1*\height)..controls +(0:\diaml*0.2) and +(-110:\diaml*0.1)..
        (\diaml,0);
    \draw[tabred,fill=tabred!\shade,xshift=\K cm]
        (0,0)..controls +(-70:\diaml*0.1) and +(180:\diaml*0.3)..
        (\diaml/3,-.75*\height)..controls +(0:\diaml*0.3) and +(-140:\diaml*0.3)..
        (\diaml,0);
    \draw node[tabred,below] at (\K/2+\diaml/2,0) {$\mu$};
    \draw node[tabred,below right] at (.2,0) {$\mu_1$};
    \draw node[tabred,below right] at (.2+\K,0) {$\tau_{K\pushonly}\mu_2$};
    \draw[tabblue,fill=tabblue!\shade]
        (0,0)..controls +(160:\diaml*0.1) and +(-90:\diaml*0.1)..
        (-1*\height,\diaml/4)..controls +(90:\diaml*0.1) and +(-90:\diaml*0.1)..
        (-0.5*\height,\diaml/2)..controls +(90:\diaml*0.1) and +(-90:\diaml*0.1)..
        (-2/3*\height,3*\diaml/4)..controls +(90:\diaml*0.1) and +(200:\diaml*0.1)..
        (0,\diaml);
    \draw[tabblue,fill=tabblue!\shade,yshift=\K cm]
        (0,0)..controls +(160:\diaml*0.1) and +(-90:\diaml*0.1)..
        (-4/5*\height,\diaml/4)..controls +(90:\diaml*0.1) and +(-90:\diaml*0.1)..
        (-0.4*\height,\diaml/2)..controls +(90:\diaml*0.1) and +(-90:\diaml*0.1)..
        (-4/5*\height,3*\diaml/4)..controls +(90:\diaml*0.1) and +(200:\diaml*0.1)..
        (0,\diaml);
    \draw node[tabblue,left] at (0,\K/2+\diaml/2) {$\nu$};
    \draw node[tabblue,above right,rotate=90] at (0,0.2) {$\nu_1$};
    \draw node[tabblue,above right,rotate=90] at (0,0.2+\K) {$\tau_{K\pushonly}\nu_2$};
    \draw (0,0-\margin) -- (0,\K/2+\diaml/2-\ellip);
    \draw[dashed] (0,\K/2+\diaml/2-\ellip) -- (0,\K/2+\diaml/2+\ellip);
    \draw[->] (0,\K/2+\diaml/2+\ellip) -- (0,\K+\diaml+\margin);
    \draw node[above right] at (0,\K+\diaml+\margin) {$\mathcal{Y}$};
    \draw (0-\margin,0) -- (\K/2+\diaml/2-\ellip,0);
    \draw[dashed] (\K/2+\diaml/2-\ellip,0) -- (\K/2+\diaml/2+\ellip,0);
    \draw[->] (\K/2+\diaml/2+\ellip,0) -- (\K+\diaml+\margin,0);
    \draw node[above right] at (\K +\diaml+\margin,0) {$\mathcal{X}$};
    \draw[<->,opacity=1,dashed] (-\height-\spacearrow,0) -- (-\height-\spacearrow,\diaml) node[midway,above,rotate=90] {$\Delta=\diam(A)$};
    \draw[<->,opacity=1,dashed] (-\height-3*\spacearrow,\diaml/2) -- (-\height-3*\spacearrow,\K+\diaml/2) node[midway,above,rotate=90] {$K$};
    \draw[opacity=1,dotted] (\diaml,0) -- (\diaml,\K+\diaml);
    \draw[opacity=1,dotted] (\K,0) -- (\K,\K+\diaml);
    \draw[opacity=1,dotted] (0,\diaml) -- (\K+\diaml,\diaml);
    \draw[opacity=1,dotted] (0,\K) -- (\K+\diaml,\K);
    \draw (\diaml/2,\diaml/2) node (pi11) {$\pi_{11}$};
    \draw (\diaml/2+\K,\diaml/2) node (pi21) {$\pi_{21}$};
    \draw (\diaml/2,\diaml/2+\K) node (pi12) {$\pi_{12}$};
    \draw (\diaml/2+\K,\diaml/2+\K) node (pi22) {$\pi_{22}$};
    \draw[->] (\K+\diaml/2-\spacesq,\diaml/2) -- (\diaml/2+\spacesq,\diaml/2) node[midway,above] {$\tilde\pi_{11}$};
    \draw[->] (\diaml/2,\K+\diaml/2-\spacesq) -- (\diaml/2,\diaml/2+\spacesq) node[midway,right] {$\tilde\pi_{11}$};
    \draw[->] (\diaml/2+\K,\diaml/2+\spacesq) -- (\diaml/2+\K,\K+\diaml/2-\spacesq) node[midway,left] {$\tilde\pi_{22}$};
    \draw[->] (\diaml/2+\spacesq,\K+\diaml/2) -- (\K+\diaml/2-\spacesq,\K+\diaml/2) node[midway,above] {$\tilde\pi_{22}$};
\end{tikzpicture}
        \caption{Visual sketch of the proof of \cref{lemma:measure-separation}.}
        \label{FigPositiveResult}
    \end{figure}

\begin{proof}
Consider first the case $t = 1/2$.
To alleviate notation, we introduce $A_1 \defeq A$ and $A_2 \defeq A+K$. We can now decompose any plan $\pi$ as $\pi_{11} + \pi_{12} + \pi_{21} + \pi_{22}$ where $\pi_{ij}$ denotes the restriction of the plan $\pi$ to the product $A_i \times A_j$, for $i,j\in\{1,2\}$. Let us also denote by $r$ the mass of $\pi_{12}$. One has $0 \leq r\leq 1/2$ and by symmetry, one can choose $r\leq 1/4$---otherwise, we switch the roles of $A_1$ and $A_2$ for $\nu$ since the cost is invariant by isometries. Remark that, due to marginal constraints, the total mass of $\pi_{11}$ and $\pi_{22}$ is $1/2 - r$ and the mass of $\pi_{21}$ is $r$. Therefore, it is possible to consider a transport plan $\tilde \pi_{11}$ between the first marginal of $\pi_{12}$ and the second marginal of $\pi_{21}$. Similarly, let $\tilde \pi_{22}$ be a transport plan between the first marginal of $\pi_{21}$ and the second marginal of $\pi_{12}$. We then define a competitor plan
$\tilde \pi  \defeq \pi_{11} + \tilde \pi_{11} + \pi_{22} + \tilde \pi_{22}$.
Slightly overloading notation, we introduce
\begin{equation}\label{EqBilinearGW}
\gw(\pi,\gamma) \defeq \int c \dd \pi \otimes \gamma.
\end{equation}
The first step is to get a lower bound on the term $\gw(\pi,\pi)$.
We expand it by bilinearity:
\begin{equation*}
\gw(\pi,\pi) = \sum_{i,j,i',j'}  \gw(\pi_{ij},\pi_{i'j'}) = \sum_{i,j} \gw(\pi_{ii},\pi_{jj}) + R,
\end{equation*}
where $R$ is the remainder that contains $12$ terms from which one can identify two types. (i) $8$ terms are of the type $\gw(\pi_{12},\pi_{11}) \geq r(1/2-r)(K^2 - \Delta^2)^2$. Indeed, since one compares pairs of points $(x,x')$ and $(y,y')$ satisfying $(x,y) \in A_1 \times A_1$ and $(x',y') \in A_1 \times A_2$, $(x - x')^2$ is upper-bounded by $\Delta^2$ and $(y - y')^2$ lower-bounded by $K^2$. The bound above then follows after integrating.
(ii) The $4$ remaining terms are of the type
$\gw(\pi_{12},\pi_{21}) \geq 0$. Summing everything yields $R \geq 8 r(1/2-r)(K^2 - \Delta^2)^2$.
We now upper-bound the competitor. Similarly, one has
\begin{equation*}
 \gw(\tilde\pi,\tilde\pi)  = \smash{\sum_{i,j}} \gw(\pi_{ii},\pi_{jj}) + \tilde R\,
\end{equation*}
where
\begin{equation*}
\tilde R = 2\gw(\tilde \pi_{11},\pi_{22} + \tilde \pi_{22}) +2 \gw(\tilde \pi_{22},\pi_{11} + \tilde \pi_{11}) + 2\gw(\pi_{11},\tilde \pi_{11}) + 2\gw(\pi_{22},\tilde \pi_{22}).
\end{equation*}
The last two terms are upper-bounded by $2r(1/2-r) \Delta^2$, since one compares squared distances within $A_1$.
By elementary inequalities again (see \cref{FigPositiveResult}), the first two terms can be upper-bounded by $r(2K\Delta + \Delta^2)^2$. Note that the total mass of the plan $\pi_{11} + \tilde \pi_{11}$ is $1/2$, which explains why $(1/2-r)$ does not appear.
Therefore, the difference between the two values of GW is
\begin{equation}\label{EqComparison}
    \gw(\pi,\pi) - \gw(\tilde \pi,\tilde \pi) \geq r\left(8 (1/2-r)(K^2 - \Delta^2)^2 - 4 (1/2-r) \Delta^2  - 2(2K\Delta + \Delta^2)^2\right).
\end{equation}
Then, since $1/2 - r\geq 1/4$, the polynomial function in $K$ on the right-hand side of \Cref{EqComparison} tends to $+\infty$ uniformly in $r\in [0,\frac 14]$ when $K$ goes to infinity, and the result follows: there exists $K>0$ such that the polynomial function above is non-negative, for instance, $K\defeq\max(0,K_0)$ where $K_0$ is the largest root of the polynomial.

The proof in the case $t > 1/2$ is even simpler: since $t- r > t-1/2$, there is no choice in the matching of the two measures and it is determined by the corresponding masses. One can then directly apply the argument above. The same reasoning applies to the case $t < 1/2$.
\end{proof}

\noindent We now prove \cref{prop:measure-separation}.

\begin{proof}[Proof of \cref{prop:measure-separation}]
Thanks to \cref{lemma:measure-separation}, we know that we can restrict to transportation plans $\pi = \pi_1 + \pi_2$ where, up to flipping $\nu$, we can assume that $\pi_1$ is supported on $A \times A$ and $\pi_2$ on $(A+K) \times (A+K)$.\footnote{Note: this is where the choice is made, as in the proof of \cref{lemma:measure-separation}, between the non-decreasing and non-increasing rearrangements. Using this convention, the non-decreasing rearrangement is shown to be optimal.}
Using again the bilinear form $\gw(\pi,\gamma)$ defined in \cref{EqBilinearGW},
the objective value reached by any transport plan $\pi = \pi_1 + \pi_2$ actually decomposes as
\begin{equation*}
\gw(\pi,\pi) = \gw(\pi_1,\pi_1) + 2\gw(\pi_1,\pi_2) + \gw(\pi_2,\pi_2).
\end{equation*}
Now, assume that we found $\pi_2\opt$ optimal. Let us minimize in $\pi_1$ the resulting quadratic problem:
\begin{equation*}
\min_{\pi_1}\ \gw(\pi_1,\pi_1) + 2\gw(\pi_1,\pi_2\opt).
\end{equation*}
We know that if $\pi_1\opt$ is a minimizer of this quantity, it must also be a solution of the \emph{linear} problem
\begin{equation*}
\min_{\pi_1}\ \gw(\pi_1, \pi_1\opt) + \gw(\pi_1,\pi_2\opt),
\end{equation*}
which is exactly the optimal transportation problem for the cost
\begin{multline*}
c(x,y) = \int_{A\times A} \big((x-x')^2 - (y-y')^2\big)^2 \dd \pi_1\opt(x',y') + 2 \int_{(A+K)^2} \big((x-x'')^2 - (y-y'')^2\big)^2 \dd \pi_2\opt(x'',y'') .
\end{multline*}
Now, using the relation $\big((x-x'')^2 - (y-y'')^2\big)^2 = \big((x-y)-(x''-y'')\big)^2 \big((x+y) - (x''+y'')\big)^2$ and the fact that $\pi_2\opt$ is a transportation plan between $\tau_{K\pushonly}\mu_2$ and $\tau_{K\pushonly} \nu_2$ so that we can make a change of variable, observe that
\begin{multline*}
    c(x,y) = \int_{A\times A} \big((x-x')^2 - (y-y')^2\big)^2 \dd \pi_1\opt(x',y') \\+ \int_{A \times A} \big((x-y)-(x''-y'')\big)^2 \big((x+y) - (x'' + y'' +2K)\big)^2 \dd (\tau_{-K},\tau_{-K})\push\pi_2\opt(x'',y'').
\end{multline*}
Now, observe that $\partial_{xy} c(x,y)$ is a polynomial function in $K,x,y$ whose dominant term in $K$ is simply $-2K^2$, and recall that $A$ is compact, so that this polynomial function is bounded in $(x,y)$. We conclude that for all $(x,y) \in A\times A$,
\begin{equation*}
\partial_{xy} c(x,y) = -2 K^2 + O(K) < 0 \qquad \text{for $K$ large enough.}
\end{equation*}
The plan $\pi_1\opt$ is optimal for a submodular cost, and by \cref{prop:submod} must be the non-decreasing rearrangement between $\mu_1$ and $\nu_1$. By symmetry, so is $\pi_2\opt$.
\end{proof}

\addtocontents{toc}{\protect\setcounter{tocdepth}{0}}
\section*{Acknowledgements}
\addtocontents{toc}{\protect\setcounter{tocdepth}{1}}
We thank Facundo Mémoli and Gabriel Peyré for their helpful remarks. We also thank the anonymous reviewers who substantially helped us improve the clarity of the paper. 
This work was supported by the Bézout Labex (New Monge Problems), funded by ANR, reference ANR-10-LABX-58.

\addtocontents{toc}{\protect\setcounter{tocdepth}{1}}
\printbibliography

\newpage
\appendix
\addcontentsline{toc}{section}{\textbf{Appendix}}
\addtocontents{toc}{\protect\setcounter{tocdepth}{0}}

\section{Proofs of \texorpdfstring{\cref{lemma:reparam,lemma:scaled-Brenier}}{Lemmas 3.3 and 3.4}}
\label{subsec:proof-scaled-brenier}

In this section, we prove \cref{lemma:reparam,lemma:scaled-Brenier}. Let us first recall the statement of \cref{lemma:reparam}:
\begin{nblemma}{3.3}
    Let $\mu,\nu\in \Pp(E)$ and let $\psi_1,\psi_2:E\to F$ be homeomorphisms. Let $\tilde c:F\times F\to\RR$ and consider the cost $c(x,y)= c( \psi_1(x),\psi_2(y))$. Then a map is optimal for the cost $c$ between $\mu$ and $\nu$ if and only if it is of the form $\psi_2^{-1}\circ T\circ\psi_1$ with $T:F\times F$ optimal for the cost $\tilde c$ between $\psi_{1\pushonly}\mu$ and $\psi_{2\pushonly}\nu$.
\end{nblemma}

        \begin{proof}
            Remark that the continuity of $\psi_1$, $\psi_2$ and their inverse ensures their measurability. We have the following equalities:
            \begin{align*}
                \argmin_{\pi \in \Pi(\mu,\nu)} \int \tilde c(\psi_1 (x), \psi_2(y))\dd\pi(x,y)& =\argmin_{\pi \in \Pi(\mu,\nu)} \int \tilde c( u,v)\dd(\psi_1,\psi_2)\push\pi(u,v)  \\
                & =(\psi_1 ^{-1} ,\psi_2^{-1} )\push \argmin_{\tilde\pi \in \Pi(\psi_{1\pushonly} \mu, \psi_{2\pushonly}\nu)} \int \tilde c(u,v)\dd\tilde\pi(u,v)
            \end{align*}
            since the mapping $(\psi_1 ^{-1} ,\psi_2^{-1} )\push$
            is a one-to-one correspondence from $\Pi(\psi_{1\pushonly} \mu, \psi_{2\pushonly}\nu)$ to $\Pi(\mu,\nu)$ by bijectivity of $\psi_1$ and $\psi_2$. This bijectivity ensures that any optimal deterministic transport plan $\tilde\pi\opt$ between $\psi_{1\pushonly}\mu$ and $\psi_{2\pushonly}\nu$ induces an optimal deterministic transport plan $\pi\opt$ between $\mu$ and $\nu$, and \textit{vice versa}. Writing $\tilde\pi\opt =(\id,T)\push(\psi_{1\pushonly}\mu)$, this plan $\pi\opt$ is given by
            \begin{align*}
                \pi \opt &= (\psi_1 ^{-1} ,\psi_2^{-1} )\push\tilde \pi \opt  \\
                    & =(\psi_1 ^{-1} ,\psi_2^{-1} )\push(\id,T)\push\psi_{1\pushonly} \mu \\
                    & =(\id,\psi_2^{-1}\circ T\circ\psi_1)\push\mu. \qedhere
            \end{align*}
        \end{proof}
\noindent We now recall the statement of \cref{lemma:scaled-Brenier}:
\begin{nblemma}{3.4}
    Let $h\geq 1$ and $\mu,\nu\in \Pp(\RR^h)$ with compact supports and such that $\mu\ll\Leb_h$. Consider the cost $c(x, y)= -\langle \psi_1(x),\psi_2(y)\rangle$ where $\psi_1,\psi_2:\RR^h\to \RR^h$ are diffeomorphisms.
    Then there exists a unique optimal transport plan between $\mu$ and $\nu$ for the cost $c$, and it is induced by a map $t:\RR^h\to \RR^h$ of the form $t=\psi_2^{-1}\circ\nabla f\circ\psi_1$, with $f:\RR^h\to\RR$ convex.
\end{nblemma}
        \begin{proof}
            As $\psi_{1\pushonly}\mu$ has a density with respect to the Lebesgue measure since $\psi_1$ is a diffeomorphism and $\psi_{1\pushonly}\mu$ and $\psi_{2\pushonly}\nu$ have compact support, \citeauthor{brenier1987decomposition}'s theorem states that there exists a unique optimal transport plan between $\psi_{1\pushonly}\mu$ and $\psi_{2\pushonly}\nu$ and that it is induced by a map $\nabla f$, where $\phi$ is a convex function. Using \cref{lemma:reparam} then yields the result.
        \end{proof}
        \begin{remark}[Hypothesis of \cref{lemma:scaled-Brenier}]
            In the proof of \cref{lemma:scaled-Brenier}, we only needed (i) $\psi_1$, $\psi_2$ and their inverse to be measurable, (ii) $\psi_{1\pushonly}\mu$ to have a density with respect to the Lebesgue measure, and (iii) $\psi_{1\pushonly}\mu$ and $\psi_{2\pushonly}\nu$ to have compact support. Imposing $\psi_1$ to be a diffeomorphism and $\psi_2$ to be a homeomorphism ensures both (i) and (ii) and is natural to expect.
        \end{remark}

\section{Measurable selection of maps in the manifold setting}

\subsection{Measurability of set-valued maps}
\label{sec:measurable_set_valued}

    Let $X$ and $U$ be two topological spaces, and let $\Bb$ denote the Borel $\sigma$-algebra on $X$. A set-valued map $S$ is a map from $X$ to $P(U)$ (the set of subsets of $U$) and will be denoted by $S : X \rightrightarrows U$.
    The idea is to introduce a notation that is consistent with the case where $S(x) = \{u\}$ for all $x$ in $X$, where we want to retrieve the standard case of maps $X \to U$. Definitions are taken from \citeauthor{rockafellar2009variational}~\cite{rockafellar2009variational}, where measurability is studied when $U = \RR^n$. Most results and proofs adapt to a more general setting, in particular when $U$ is a complete Riemannian manifold $M$, which we shall assume in the following.
    For the sake of completeness, we provide all the proofs and highlight those that require specific care when replacing $\RR^n$ by a manifold.
    Of importance for our proofs, we define the \emph{pre-image} of a set $B \subset U$ by $S$,
    \begin{equation*}
    S^{-1}(B) = \{ x \in X \mid S(x) \cap B \neq \varnothing\},
    \end{equation*}
    and the \emph{domain} of $S$, defined as $S^{-1}(U)=\{ x \in X \mid S(x) \neq \varnothing \}$.
    We will often use the following relation: if a set $A$ can be written as $A = \bigcup_k A_k$, then $S^{-1}(A) = \bigcup_k S^{-1}(A_k)$. Indeed,
    \begin{equation*}
    x \in S^{-1}(A) \iff S(x) \cap A \neq \varnothing \iff \exists k,\, S(x) \cap A_k \neq \varnothing \iff \exists k,\, x \in S^{-1}(A_k) \iff x \in \textstyle\bigcup_k S^{-1}(A_k).
    \end{equation*}
    A set-valued map $S : X \rightrightarrows U$ is said to be \emph{measurable} if, for any open set $O \subset U$,
    \begin{equation*}
        S^{-1}(O) \in \Bb.
    \end{equation*}
    Note that if $S$ is measurable (as a set-valued map), then its domain must be measurable as well (as an element of $\Aa$).
    We say that $S : X \rightrightarrows U$ is \emph{closed-valued} if $S(x)$ is a closed subset of $U$ for all $x \in X$.

    \begin{proposition}[{\cite[Theorem~14.3.c]{rockafellar2009variational}}]
    \label{prop:equiv-measure-closedvalued}
        A closed-valued map $S : X \rightrightarrows U$ is measurable if and only if the property ``$S^{-1}(B) \in \Bb$'' holds either:
        \begin{enumerate}[label=(\alph*),nolistsep,wide, leftmargin=*]
        \item for all $B \subset U$ open (the definition);
        \item for all $B \subset U$ compact;
        \item for all $B \subset U$ closed.
        \end{enumerate}
    \end{proposition}
    \begin{proof}\
    \begin{itemize}[nolistsep,wide, leftmargin=*]
    \item (a) $\Longrightarrow$ (b): For a compact $B \subset U$, let $B_k = \{ x \in U \mid d(x,B) < k^{-1} \}$ for $k \geq 0$ (which is open). Note that
    \begin{equation*}
    x \in S^{-1}(B) \iff S(x) \cap B \neq \varnothing \iff S(x) \cap B_k \neq \varnothing \text{ for all $k$,}
    \end{equation*}
    because $S(x)$ is a closed set. Hence $S^{-1}(B) = \bigcap_k S^{-1}(B_k)$. All the $S^{-1}(B_k)$ are measurable, and so is $S^{-1}(B)$ as a countable intersection of measurable sets.

    \item (b) $\Longrightarrow$ (a): Fix $O$ an open set of $U$. As we assumed $U$ to be a complete separable Riemannian manifold, $O$ can be written as a countable union of compact balls: $O = \bigcup_n \overline{B(x_n, r_n)}$.

    \item (b) $\Longrightarrow$ (c): Immediate.

    \item (c) $\Longrightarrow$ (b): A closed set $B$ can be obtained as a countable union of compact sets by letting $B = \bigcup_n B \cap \overline{B(x_0, n)}$ for some $x_0$.
    Hence $S^{-1}(B) = \bigcup_n S^{-1}(B \cap \overline{B(x_0,n)})$ is in $\Bb$. \qedhere
    \end{itemize}
    \end{proof}

    Now, we state a result on operations that preserve the measurability of closed-set valued maps. The proof requires adaptation from the one of \cite{rockafellar2009variational} because the latter uses explicitly the fact that one can compute Minkowski sums of sets (which may not make sense on a manifold).

    \begin{proposition}[{\cite[Proposition 14.11]{rockafellar2009variational}}, adapted to the manifold case]
    \label{prop:intersection-of-measurable}
        Let $S_1$ and $S_2 : X \rightrightarrows U$ be two measurable closed-set valued maps.
        Then
        \begin{itemize}[nolistsep,wide, leftmargin=*]
            \item $P : x \mapsto S_1(x) \times S_2(x)$ is measurable as a closed-valued map in $U \times U$ (equipped with the product topology).
            \item $Q : x \mapsto S_1(x) \cap S_2(x)$ is measurable.
        \end{itemize}
    \end{proposition}

    \begin{proof}
    The first point can be proved in the same spirit as the proof proposed by \citeauthor{rockafellar2009variational}. Namely, let $O'$ be an open set in $U \times U$. By definition of the product topology, $O'$ can be obtained as $\bigcup_n O_1^{(n)} \times O_2^{(n)}$ where $O_1^{(n)}$ and $O_2^{(n)}$ are open sets in $U$. Then $P^{-1}(O') = \bigcup_n P^{-1}(O_1^{(n)} \times O_2^{(n)})$. Now, observe that
    \begin{equation*}
    P^{-1}(A \times B) = \{ x\mid S_1(x) \times S_2(x) \in A \times B\} = \{x\mid S_1(x) \in A \text{ and } S_2(x) \in B \} = S_1^{-1}(A) \cap S_2^{-1}(B),
    \end{equation*}
    so that finally, $P^{-1}(O') = \bigcup_n S_1^{-1}(O_1^{(n)}) \cap S_2^{-1}(O_2^{(n)})$ is measurable as a countable union of (finite) intersection of measurable sets (given that $S_1$ and $S_2$ are measurable). Note that this does not require $S_1$ and $S_2$ to be closed-valued.

    Now, let us focus on the second point, which requires more attention. Thanks to \cref{prop:equiv-measure-closedvalued}, it is sufficient to show that $Q^{-1}(C) \in \Bb$ for any compact set $C \subset U$.
    In \cite{rockafellar2009variational}, this is done by writing
    \begin{multline*}
    Q^{-1}(C) = \{ x\mid S_1(x) \cap S_2(x) \cap C \neq \varnothing \} = \{x\mid R_1(x) \cap R_2(x) \neq \varnothing \} \\= \{x\mid 0 \in (R_1(x) - R_2(x)) \} = (R_1 - R_2)^{-1}(0),
    \end{multline*}
    where $R_j(x) = S_j(x) \cap C$ (which is also closed-valued), and using the fact that the (Minkowski) difference of measurable closed-valued maps is measurable as well \cite[Proposition 14.11.c]{rockafellar2009variational}.
    To adapt this idea (we cannot consider Minkowski differences in our setting), we introduce the diagonal $\Delta = \{(u,u)\mid u \in U\} \subset U \times U$. Now, observe that
    \begin{equation*}
    R_1(x) \cap R_2(x) \neq \varnothing \iff (R_1(x) \times R_2(x)) \cap \Delta \neq \varnothing,
    \end{equation*}
    that is $x \in R^{-1}(\Delta)$, where $R(x) = R_1(x) \times R_2(x)$. Since the maps $R_1$ and $R_2$ are measurable closed-valued maps (inherited from $S_1,S_2$), so is $R$ according to the previous point. And since $\Delta$ is closed, $R^{-1}(\Delta) = Q^{-1}(C)$ is measurable.
    \end{proof}

\subsection{Proof of \texorpdfstring{\Cref{prop:selection-manifold}}{Proposition 2.10}}
\label{sec:appendix-proof-fontbona-manifold}
Let us first recall the result.
\begin{nbproposition}{2.10}[Measurable selection of maps, manifold case]
    Let $M$ be a complete Riemannian manifold and $(B, \Sigma, m)$ a measure space.
    Consider a measurable function $B\ni u \mapsto (\mu_u, \nu_u) \in P(M)^2$. Assume that for $m$-almost every $u\in B$, $\mu_u\ll\vol_M$ and $\mu_u$ and $\nu_u$ have a finite transport cost. Let $T_u$ denote the (unique by \cref{prop:quad-cost-manifold-villani}) optimal transport map induced by the squared distance cost $d_M^2$ on $M$ between $\mu_u$ and $\nu_u$.
    Then there exists a measurable function $(u,x)\mapsto T(u,x)$ such that for $m$-almost every $u$, $T(u,x)=T_{u}(x)$, $\mu_{u}$-a.e.
\end{nbproposition}

The proof is essentially an adaptation of the one of Theorem~1.1 from \citeauthor{fontbona2010measurability}~\cite{fontbona2010measurability}, with additional care required since we do not have access to a linear structure on the manifold $M$. It relies on the measurability of set-valued maps (see \cite[Chapters 5 and 14]{rockafellar2009variational} and \cref{sec:measurable_set_valued} for a summary).
In the following, we consider a partition $(A_{nk})_n$ of $M$ made of cells with diameter $\leq 2^{-k}$ and such that $(A_{nj})_n$ is a refinement of $(A_{nk})_n$ in sense that each $A_{nk}$ is itself partitioned by some $A_{n'j}$. 
The crucial point regarding measurability is the following proposition.

\begin{proposition}
    \label{prop:Bnk-measurable}
    The set $B_{nk} = \{ (u,x)\mid T_u(x) \in A_{nk} \}$ is measurable.
\end{proposition}

\noindent Its proof relies on a core lemma:

\begin{lemma}
\label{lemma:ferme-mesurable}
    Let $F \subset M$ be a closed set. Then the set $B_F = \{ (u,x)\mid T_u(x) \in F \}$ is measurable.
\end{lemma}

\noindent The key will be to identify this set as the domain of a measurable set-valued map, see \cref{sec:measurable_set_valued}.

\begin{proof}[Proof of \cref{lemma:ferme-mesurable}]
Observe that
\begin{equation*}
B_F = \{(u,x)\mid (\{x\} \times F) \cap \graph(T_u) \neq \varnothing\},
\end{equation*}
where $\graph(T_u) \defeq \{(x,T_u(x))\mid x \in M\}$ denotes the topological closure of the graph of the optimal transport map $T_u$ that pushes $\mu_u$ onto $\nu_u$.
Let $S_1 : (u,x) \mapsto \{x\} \times F$ and $S_2 : (u,x) \mapsto \graph(T_u)$, so that $B_F = \mathrm{dom}(S)$, where $S(x) = S_1(x) \cap S_2(x)$. According to \cref{prop:intersection-of-measurable}, given that $S_1$ and $S_2$ are closed-valued, if they are measurable, so is $S$, and so is $B_F$ as the domain of a measurable map. Establishing the measurability of these two maps can be easily done by adapting the work of \cite{fontbona2010measurability}. We give details here for the sake of completeness.
\begin{itemize}[wide=0pt]
\item \textit{Measurability of $S_1$:} Let $O \subset M \times M$ be open. 
Observe that
\begin{equation*}
S_1^{-1}(O) = \{(u,x) \mid \{x\} \times F \cap O \neq \varnothing \} = B \times \{x \mid \{x\} \times F \cap O \neq \varnothing \}. 
\end{equation*}
Fix $z \in F$.
There exists $\epsilon > 0$ such that $B(x,\epsilon) \times \{z\} \subset O$ (since $O$ is open), and thus $B(x,\epsilon) \times F \cap O \neq \varnothing$, proving that there is a neighborhood of $x$ included in $\{x \mid \{x\} \times F \cap O \neq \varnothing \}$ which is thus open, making in turn $S_1^{-1}(O)$ open (hence measurable), hence the measurability of $S_1$.
\item \textit{Measurability of $S_2$:} Given that $u \mapsto (\mu_u,\nu_u)$ is measurable by assumption, and that measurability is preserved by composition, we want to show that
\begin{enumerate}[(i)]
    \item the map $S : (\mu,\nu) \mapsto \Pi\opt(\mu,\nu)$ (the set of optimal transport plans between $\mu$ and $\nu$ for the squared distance cost on $M$) is measurable, and
    \item the map $\mathfrak{U} : \pi \in P(M^2) \mapsto \supp\pi$ satisfies that $\mathfrak{U}^{-1}(O)$ is open for any open set $O \subset P(M^2)$.
\end{enumerate}
From these two points, we obtain the measurability of $(U \circ S)^{-1}(O)$, thus the measurability of $S_2$.

To get (i), observe first that $S$ is closed-valued, so that it is sufficient to prove that $S^{-1}(C)$ is measurable for any closed set $C \subset P(M^2)$ according to \cref{prop:equiv-measure-closedvalued}. Let $C \subset P(M^2)$ be closed. Then, $S^{-1}(C) = \{(\mu,\nu)\mid\Pi\opt(\mu,\nu) \cap C \neq \varnothing \}$. Consider a sequence $(\mu_n, \nu_n)_n$ in $S^{-1}(C)$ that converges to $(\mu,\nu)$ for the weak topology. Let $\pi_n \in \Pi\opt(\mu_n,\nu_n) \cap C$. According to \cite[Theorem~5.20]{villani2009optimal}, $(\pi_n)_n$ admits a weak limit $\pi$ in $\Pi\opt(\mu,\nu)$; but since $C$ is closed, $\pi \in C$, so $(\mu,\nu) \in S^{-1}(C)$, which is closed (hence measurable), proving the measurability of $S$.

(ii) simply follows from the fact that
\begin{equation*}
\mathfrak{U}^{-1}(O) = \{ \pi\mid \supp\pi \cap O \neq \varnothing \} = \{ \pi\mid \pi(O) > 0 \},
\end{equation*}
which is open. Indeed, the Portmanteau theorem shows that if $\pi_n \to \pi$ (weakly) and $\pi_n(O) = 0$, then $0 = \liminf \pi_n(O) \geq \pi(O) \geq 0$, so $\pi(O) = 0$. The complementary set of $U^{-1}(O)$ is closed, that is $\mathfrak{U}^{-1}(O)$ is open. \qedhere
\end{itemize}
\end{proof}

\begin{proof}[Proof of \cref{prop:Bnk-measurable}]
Using that $A_{nk}$ can be inner-approximated by a sequence of closed set $F_j \subset A_{nk}$, we obtain a sequence of sets $(B_{F_j})_j$ such that $\bigcup_j B_{F_j} = A_{nk}$. By \cref{lemma:ferme-mesurable}, the $(B_{F_j})_j$ are measurable, so is $A_{nk}$ as the (countable) union of measurable sets.
\end{proof}

\noindent We can now prove our main proposition on measurable selections of maps. 

\begin{proof}[Proof of \cref{prop:selection-manifold}]
Recall that we assume that $M = \bigsqcup_n A_{nk}$. For each $(n,k)$, select (in a measurable way) an element $a_{nk}\in A_{nk}$. 
Then, define the map
\begin{equation*}
    T^{(k)} : (u,x) \mapsto a_{nk},\ \text{ where } T_u(x) \in A_{nk}\, .
\end{equation*}
This map is measurable: for any open $O \subset M$,
\begin{equation*}
    T^{(k),-1}(O) = \{(u,x) \mid T_u(x) \in A_{nk} \} = B_{nk},
\end{equation*}
which is measurable by \cref{prop:Bnk-measurable}.
Now, for two maps $f,g : B \times M \to M$, let $D_1$ denotes the natural $L^1$ distance on $M$, that is
\begin{equation*}
    D_1(f,g) = \int_B \int_M d\big(f(u,x), g(u,x)\big) \dd \mu_u(x)  \dd m(u).
\end{equation*}
This yields a complete metric space \cite{chiron2007definitions}, and $(T^{(k)})_k$ is a Cauchy sequence for this distance. Indeed, for $k \leq j$ two integers, recall that we assumed that $(A_{nj})_n$ is a refinement of $(A_{nk})_n$, yielding
\begin{align*}
    D_1(T^{(k)}, T^{(j)}) &= \int_{B} \int_M d\big(T^{(k)}(u,x), T^{(j)}(u,x)\big) \dd \mu_u (x) \dd m(u) \\
    &= \int_{B} \int_M \sum_n \sum_{n' :\,A_{n'j} \subset  A_{nk}}  \one_{B_{n'j}}(u,x) d(a_{nk}, a_{n'j}) \dd \mu_u(x) \dd m(u) \\
    &= \int_B \int_M \sum_n \sum_{n' :\,A_{n'j} \subset  A_{nk}} d(a_{nk},a_{n'j}) \dd \nu_u(A_{n'j}) \dd m(u) \\
    &\leq 2^{-k}
\end{align*}
where we successively used (i) that for all $u$, $\int_{M} \one_{B_{n'j}}(u,x) \dd \mu_u(x) = \nu_u(A_{n'j})$ by construction, (ii) that the diameter of the partition $A_{nk}$ is less than or equal to $2^{-k}$ and (iii) that $\nu_u$ and $m$ are probability measures. Note that claim (i) is given by
\begin{equation*}
(u,x) \in B_{n'j} \iff T_u(x) \in A_{n'j} \iff x \in \mu_u(T_u^{-1}(A_{n'j})) = T_{u \pushonly} \mu_u(A_{n'j})=\nu_u(A_{n'j}).
\end{equation*}
Now, let $T$ denote the limit of $(T^{(k)})_k$ (which is measurable). It remains to show that $T(u,x) = T_u(x)$, $m$-a.e. This can be obtained by proving that
\begin{equation*}
    \int_M g(x) f(T(x,u)) \dd \mu_u(x) = \int_M g(x) f(T_u(x)) \dd \mu_u(x),
\end{equation*}
for any pair $f,g : M \to \RR$ of bounded Lipschitz-continuous functions \cite[Lemma~2.24]{van2000asymptotic}.
As in \cite{fontbona2010measurability}, let $\|f\| \defeq \sup_{x \neq y} \frac{|f(x) - f(y)|}{d(x,y)} + \sup_x |f(x)|$. The difference between these two terms can be bounded using the partition $(A_{nk})_n$. We have for $m$-a.e.~$u$:
\begin{multline*}
    \left| \int_M g(x) f(T_u(x)) \dd \mu_u(x) - \int_M g(x) f(T(u,x)) \dd \mu_u(x) \right| \\
    \leq \left| \int_M g(x) f(T_u(x)) \dd \mu_u(x) - \int_M g(x) f(T^{(k)}(u,x)) \dd \mu_u(x) \right| + \|g\| \|f\| \int_M d\big( T^{(k)}(u,x), T(u,x) \big) \dd \mu_u(x).
\end{multline*}
Since $T^{(k)} \to T$ in $D_1$, it implies that up to a subsequence, for $m$-a.e.~$u$, the second term tends to zero:
\begin{equation*}
    \int_M d\big(T^{(k)}(u,x), T(u,x)\big) \dd \mu_u(x)\xrightarrow[]{k\to\infty}0.
\end{equation*}
To treat the first term and show that it tends to $0$ as $k \to \infty$ for a subset of $B$ with full $m$-measure, we write for $m$-a.e.~$u$:
\begin{align*}
    \left| \int_M g(x) f(T_u(x)) \dd \mu_u(x) - \int_M g(x) f(T^{(k)}(u,x)) \dd \mu_u(x) \right|
    \leq &\int_M |g(x)| \big| f(T_u(x)) - f(T^{(k)}(u,x))\big| \dd \mu_u(x) \\
    \leq &\|g\| \|f\| \int_M d\big(T_u(x),T^{(k)}(u,x)\big) \dd \mu_u(x) \\
    \leq &\|g\| \|f\|\ 2^{-k} \sum_n \nu_u(A_{nk})\xrightarrow[]{k\to\infty}0,
\end{align*}
which concludes the proof.
\end{proof}

\section{Measure disintegration}

The following definition and theorem are taken from \citeauthor{ambrosio2005gradient}~\cite{ambrosio2005gradient}. 

\label{sec:disintegration}
        \begin{definition}[Measure disintegration]
            Let $\Xx$ and $\Zz$ be two Radon spaces, $\mu  \in \Pp(\Xx)$ and $\varphi : \Xx \to \Zz$ a Borel-measurable function. A family of probability measures $\{\mu_{u}\}_{u \in \Zz} \subset \Pp(\Xx)$ is a \emph{disintegration} of $\mu$ by $\varphi$ if:
            \begin{enumerate}[(i)]
                \item the function $u \mapsto \mu_{u}$ is Borel-measurable;
                \item $\mu_{u}$ lives on the fiber $\varphi^{-1}(u)$, that is for $\varphi\push\mu$-a.e.~$u \in \Zz$,
                \begin{equation*}
                \mu_{u}\big(\Xx \setminus \varphi^{-1}(u)\big)=0,
                \end{equation*}
                and therefore $\mu_{u}(B)=\mu_{u}(B \cap \varphi^{-1}(u))$ for any Borel $B\subset \Xx$;
                \item for every measurable function $f: \Xx \rightarrow[0, \infty]$,
                \begin{equation*}
                \int_{\Xx} f(x) \dd \mu(x)=\int_{\Zz} \Big(\int_{\varphi^{-1}(u)} f(x) \dd \mu_{u}(x)\Big) \dd (\varphi\push\mu)(u).
                \end{equation*}
                In particular, for any Borel $B \subset \Xx$, taking $f$ to be the indicator function of $B$,
                \begin{equation*}
                \mu(B)=\int_{\Zz} \mu_{u}(B) \dd (\varphi\push\mu)(u).
                \end{equation*}
            \end{enumerate}
        \end{definition}
        \begin{theorem}[Disintegration theorem]
            Let $\Xx$ and $\Zz$ be two Radon spaces, $\mu  \in \Pp(\Xx)$ and $\varphi : \Xx \to \Zz$ a Borel-measurable function. There exists a $\varphi\push\mu$-a.e.~uniquely determined family of probability measures $\{\mu_{u}\}_{u \in \Zz} \subset \Pp(\Xx)$ that provides a disintegration of $\mu$ by $\varphi$.
        \end{theorem}

\end{document}